\documentclass{amsbook}

\usepackage{amsfonts}
\usepackage{amsmath}
\usepackage{amssymb}
\usepackage{amsthm}
\usepackage{fancyhdr}
\usepackage{varioref}
\usepackage{enumerate}
\usepackage{epsfig}
\usepackage{xypic}
\usepackage{picins}
\usepackage{graphicx}
\usepackage{xspace}
\usepackage
            [bookmarksopen,pdfauthor={Christian Boltner},
            pdftitle={On the Structure of Equidistant Foliations in Euclidean Space}]{hyperref}

\newcommand{\sep}{\vspace{.3cm}}

\newcommand{\nice}{well projecting\xspace}
\newcommand{\vphi}{\varphi}

\newcommand{\lt}{\left}
\newcommand{\rt}{\right}

\newcommand{\ra}{\rightarrow}

\newcommand{\eps}{\ensuremath{\varepsilon}}

\newcommand{\union}{\bigcup}

\newcommand{\point}[2]{\left(#1,#2\right)}

\newcommand{\set}[1]{\left\{#1\right\}}
\newcommand{\Set}[2]{\left\{#1\,\middle|\,#2\right\}}

\newcommand{\ID}[1]{\ensuremath{\operatorname{id}|_{#1}}}
\newcommand{\id}{\ensuremath{\operatorname{id}}}

\newcommand{\R}{\ensuremath{\mathbb{R}}}
\newcommand{\N}{\ensuremath{\mathbb{N}}}

\newcommand{\C}{\ensuremath{\mathbb{C}}}
\newcommand{\Quaternions}{\ensuremath{\mathbb{H}}}
\newcommand{\Qt}{\Quaternions}
\newcommand{\Isom}[1]{\operatorname{Isom}(#1)}
\newcommand{\IsomNull}[1]{\operatorname{Isom_0}(#1)}
\newcommand{\SO}[1]{\ensuremath{\operatorname{SO}(#1)}}
\renewcommand{\O}[1]{\ensuremath{\operatorname{O}(#1)}}
\newcommand{\U}[1]{\ensuremath{\operatorname{U}(#1)}}
\newcommand{\SU}[1]{\ensuremath{\operatorname{SU}(#1)}}
\newcommand{\Sp}[1]{\ensuremath{\operatorname{Sp}(#1)}}
\newcommand{\Spin}[1]{\ensuremath{\operatorname{Spin}(#1)}}
\renewcommand{\S}[1]{\ensuremath{\operatorname{S}^{#1}}}
\newcommand{\rk}[1]{\operatorname{rk}\left(#1\right)}
\newcommand{\Lie}[1]{\mathfrak{#1}}
\newcommand{\Centr}[1]{\operatorname{Centr}(#1)}    
\newcommand{\Norm}[1]{\operatorname{Norm}(#1)}      
\newcommand{\NormNull}[1]{\operatorname{Norm_0}(#1)}      
\newcommand{\rad}{\operatorname{rad}}

\newcommand{\Abl}[3][]{\ensuremath{\frac{d^{#1} {#2}}{d {#3}^{#1}}}}
\newcommand{\abl}[3][]{\ensuremath{\frac{\partial^{#1} {#2}}{\partial {#3}^{#1}}}}
\newcommand{\Ddt}[1][t]{\left.\Abl{}{#1}\right|_{#1=0}}
\newcommand{\ddt}[1][t]{\left.\abl{}{#1}\right|_{#1=0}}

\newcommand{\ip}[2]{\left\langle #1, #2 \right\rangle}
\newcommand{\norm}[1]{\left\|#1\right\|}
\newcommand{\image}[1]{\operatorname{im}\left(#1\right)}
\newcommand{\spann}[1]{\operatorname{span}\left[#1\right]}

\newcommand{\normal}[1][]{\nu^{#1}}
\newcommand{\osc}[1]{O^{#1}}
\newcommand{\minnormal}[1][]{\bar{\nu}^{#1}}
\newcommand{\schnitt}[1]{\Gamma\left(#1\right)} 
\newcommand{\kov}[2]{\nabla_{#1}{#2}}

\newcommand{\Deriv}[2]{{#1}_{\ast_#2}}
\newcommand{\lie}[2]{\left[#1,#2\right]}
\newcommand{\shape}[2]{S_{#1}{#2}}

\newcommand{\sff}{\alpha}   
\newcommand{\Sff}[2]{\sff\lt(#1,#2\rt)}
\newcommand{\diff}[2][]{#2_{\ast_{#1}}}

\newcommand{\groupelement}[3]{\lt(\lt(\begin{array}{c|c}#1\\\hline&#2\end{array}\rt),
                                    \begin{pmatrix}#3\\0\end{pmatrix}
                                    \rt)}

\newcommand{\closure}[1]{\overline{#1}}

\newcommand{\maxgroup}[1]{#1^{\operatorname{max}}}
\newcommand{\layergroup}[1]{G_{#1}}

\newcommand{\Hor}[1][]{\mathcal{H}_{#1}}
\newcommand{\Ver}[1][]{\mathcal{V}_{#1}}
\newcommand{\hor}[1]{#1^h}
\newcommand{\ver}[1]{#1^v}
\newcommand{\pDist}[2][]{\mathcal{E}_{#2}^{#1}}
\newcommand{\pDistN}[2][]{\mathcal{D}_{#2}^{#1}}
\newcommand{\pDistLift}[2][]{\bar{\mathcal{E}}_{#2}^{#1}}

\newcommand{\VecF}[1][]{\mathfrak{X}}
\newcommand{\HorF}[1][]{\mathfrak{H}}
\newcommand{\VerF}[1][]{\mathfrak{V}}
\newcommand{\BottF}[1][]{\mathfrak{B}}
\newcommand{\pos}{\ensuremath{\Xi}}

\newcommand{\kovH}[2]{\overset{h}{\nabla}_{#1}{#2}}
\newcommand{\kovV}[2]{\overset{v}{\nabla}_{#1}{#2}}

\newcommand{\diam}[1]{\operatorname{diam}\left(#1\right)}
\newcommand{\dist}[2]{\operatorname{dist}\lt( #1 , #2 \rt)}
\newcommand{\Dist}[2]{\left|{#1}{#2}\right|}
\newcommand{\dst}[2]{d\lt(#1,#2\rt)}
\newcommand{\dHaus}[2]{d_H\lt(#1,#2\rt)}
\newcommand{\HausConv}{\overset{d_H}{\longrightarrow}}
\newcommand{\lift}{\mathcal{L}}
\newcommand{\Lift}[2][]{\lift_{#1}\lt(#2\rt)}
\newcommand{\cone}[1]{\ensuremath{\operatorname{C}#1}}
\newcommand{\rand}[1][]{\partial_{#1}}
\newcommand{\inn}[2][]{\operatorname{int}_{#1}\lt(#2\rt)}
\newcommand{\oneill}[2]{\mathcal{O}_{#1}{#2}}
\newcommand{\oneillAd}[2]{\mathcal{O}^\ast_{#1}{#2}}
\newcommand{\onormal}[2]{N_{#1}{#2}}

\newcommand{\proj}{\mathbb{P}}
\newcommand{\projV}{\ver{\proj}}
\newcommand{\projH}{\hor{\proj}}
\newcommand{\vproj}[1][]{\tilde{\proj}_{#1}}
\newcommand{\vprojH}[1][]{\hor{\vproj[#1]}}

\newcommand{\affinefibre}{\ensuremath{F_0}}
\newcommand{\affineleaf}{\affinefibre}
\newcommand{\base}{\ensuremath{\mathbb{B}}}
\newcommand{\basedir}{\ensuremath{\Sigma_0}}
\newcommand{\intermediatespace}{\ensuremath{\mathbb{A}}}
\newcommand{\interspace}{\intermediatespace}

\newcommand{\ball}[3][]{B^{#1}_{#2}\lt(#3\rt)}
\newcommand{\plane}[3][]{E_{#1}\lt(#2,#3\rt)}
\newcommand{\tube}[2]{\tau_{#1}\lt(#2\rt)}
\newcommand{\fol}[1]{\mathcal{#1}}
\newcommand{\Fol}{\fol{F}}

\newcommand{\FolS}{\Fol^1}
\newcommand{\layer}[1]{\ensuremath{L_{#1}}}
\newcommand{\IndFol}[1][]{\ensuremath{\mathcal{\tilde{F}}_{#1}}}
\newcommand{\IndFolS}[1][]{\ensuremath{\mathcal{\tilde{F}}^1_{#1}}}
\newcommand{\IndLeaf}[2]{\ensuremath{\tilde{#1}_{#2}}}
\newcommand{\InterFol}{\fol{A}}

\newcommand{\length}[1]{L\lt(#1\rt)}
\newcommand{\hausdorffspace}[1]{\mathfrak{M}(#1)}
\newcommand{\subgrp}[1]{\mathcal{S}(#1)}

\newtheoremstyle{bfplain}
  {.7em}
  {.7em}
  {\itshape}
  {\parindent}
  {\bfseries}
  {.}
  {.5em}
  {}

\newtheoremstyle{bfdefinition}
  {.7em}
  {.7em}
  {}
  {\parindent}
  {\bfseries}
  {.}
  {.5em}
  {}

\newtheoremstyle{itremark}
  {.3em}
  {.3em}
  {}
  {\parindent}
  {\itshape}
  {.}
  {.5em}
  {}

\theoremstyle{bfplain}
\newtheorem{Thm}{Theorem}[chapter]
\newtheorem{Prop}[Thm]{Proposition}
\newtheorem{Lem}[Thm]{Lemma}
\newtheorem{Cor}[Thm]{Corollary}

\theoremstyle{bfdefinition}
\newtheorem{Def}[Thm]{Definition}
\newtheorem{Rem}[Thm]{Remark}

\theoremstyle{itremark}

\newtheorem*{Rem*}{Remark}

\numberwithin{equation}{chapter}
\numberwithin{section}{chapter}
\numberwithin{figure}{chapter}

\begin{document}

\frontmatter
\begin{titlepage}
\begin{center}
\rm \Large
{
    \
    \vspace{1cm}
    \sc\Huge
    \hrule\vspace{0.6em}
       On the Structure of Equidistant Foliations of Euclidean Space
    \vspace{0.4em}
    \hrule
}
\vspace{4cm}
        Dissertation\\
        zur Erlangung des Doktorgrades an der\\
        Mathematisch-Naturwissenschaftlichen Fakul{\"a}t\\
        der Universit{\"a}t Augsburg\\
\vspace{3cm}
vorgelegt von\\
\ \\
{Christian Boltner}\\
\vspace{4cm}
        Augsburg, Juni 2007
\end{center}
\end{titlepage}

\setcounter{page}{2}
\vspace*{\stretch{3}}
\begin{tabular}{ll}
Erstgutachter:&Prof.\ Dr.\ Ernst Heintze\\
Zweitgutachter:&Prof.\ Dr.\ Jost-Hinrich Eschenburg\\
\\
Tag der m{\"u}ndlichen Pr{\"u}fung:& 04.\ September 2007
\end{tabular}
%
%

\tableofcontents

\mainmatter

\chapter*{Introduction}

The aim of this thesis is the study of equidistant foliations of Euclidean space,
in particular answering the question whether they are homogeneous.

An \emph{equidistant foliation} of $\R^n$ is a partition $\Fol$ into
\emph{complete, smooth, connected, properly embedded} submanifolds of $\R^n$ such that for any two
leaves~$F,G\in\Fol$ and $p\in F$ the distance $d_G(p)$ does not depend on the
choice of~$p\in F$.
Such a foliation may be \emph{singular}, i.e.\ the leaves of $\Fol$ may have
different dimensions.

We point out that this is a more restrictive version of the definition of
\emph{singular Riemannian foliations} as given by \cite{Mol}.
Their leaves only need to be immersed and equidistance is therefore
only demanded locally.

The advantage of our more restrictive definition is that the space of leaves
$\base:=\R^n/\Fol$ bears a natural metric --- it is even a nonnegatively
curved Alexandrov space (cf.~\cite{BBI}) --- and the canonical projection is a
submetry.
Indeed we make heavy usage throughout this work of the Alexandrov space structure
of $\base$ and rely on the rich theory of submetries as found in \cite{Lyt}.

The most prominent examples of equidistant foliations are the orbit foliations
of isometric Lie group actions. So the natural question is whether all equidistant
foliations of $\R^n$ are homogeneous or at least which conditions imply homogeneity.

A huge and well studied class of equidistant foliations are those given by
isoparametric submanifolds and their parallel manifolds.
Homogeneity of these foliations was shown by Thorbergsson in \cite{Thor}
if the isoparametric submanifold has codimension $\geq3$.
However, there are \emph{inhomogeneous} examples --- found by Ferus, Karcher and M\"unzner
and presented in
\cite{FKM} --- if the isoparamtric submanifold has codimension 2, i.e.\ if it is
a hypersurface in a sphere.

To our knowledge these and the Hopf fibration of $\S{15}$ (with totally
geodesic fibres, isometric to $\S{7}$) are the only inhomogeneous examples of
equidistant foliations known today.
We point out that all of these inhomogeneous foliations are compact, i.e.\ they
have compact leaves.

On the other hand Gromoll and Walschap examine \emph{regular} equidistant
foliations --- which are necessarily noncompact --- in \cite{GW:1} and \cite{GW:2}.
They show that such a foliation always has an affine leaf, which they use to prove
that the foliation is homogeneous; in fact it is
given by a generalized screw motion around the affine leaf.

As all inhomogeneous examples are compact it seems reasonable to concentrate on
noncompact foliations.
Generalizing Gromoll and Walschap's result we show in this thesis that an
equidistant foliation of $\R^n$ always has an affine leaf and may be described
by a compact equidistant foliation in one normal
space of the affine leaf together with a (not necessarily homogeneous)
screw motion around that leaf.
We give conditions for homogeneity and also construct new (noncompact)
inhomogeneous examples.

A more detailed summery of this work follows:

\sep


In \textsc{Chapter~\ref{Chap:Preliminaries}} we introduce the concepts of
Alexandrov spaces, submetries and their derivatives and we
define equidistant foliations.
We present several basic results concerning these concepts ---
among others we show that the regular leaves of equidistant foliations are equifocal.

\sep

In analogy to Gromoll and Walschap's result we show in \textsc{Chapter~\ref{Chap:Existence}}
that equidistant foliations always have an affine leaf $\affineleaf$.
Using essentially Cheeger-Gromoll's soul construction
(cf.~\cite{CE})
we prove that even in the singular case $\base$ has a soul and its preimage is
an affine space.
Then an affine leaf exists if this soul is a single point.
To show this we cannot follow \cite[Sect.~2]{GW:1}
as the topological results used there rely on $\Fol$ being a fibration.
Instead we give a geometrical proof (which also gives a new proof for the regular
setting).

\sep

For any $p\in\affineleaf$ the intersection of the leaves of $\Fol$ with the
\emph{horizontal layers} $\layer{p}:=p+\normal_p\affineleaf$ yields a partition
of~$\layer{p}$ which we call $\IndFol[p]$ and all of the $\IndFol[p]$ together
give us a partition $\IndFol$ of $\R^n$.
\textsc{Chapter~\ref{Chap:Horizontal}} is dedicated to studying this induced foliation,
in par\-ti\-cu\-lar we show that each $\IndFol[p]$ is an equidistant foliation of~$\layer{p}$.

We prove that in the homogeneous case $\Fol$ is given by the orbits of
$G\times\R^k$ with $G$ a compact Lie group and $\R^k$ acting on $\R^{k+n}$ by
generalized screw motions around the axis $\affineleaf$ and
we conclude that the induced foliation $\IndFol$ is equidistant.

In the remainder of this chapter we give a characterization of when $\IndFol$
is equidistant and we show that --- provided each $\IndFol[p]$ is homogeneous ---
the $\IndFol[p]$ are isometric to each other and $\Fol$ can be described by two
data: any one of the $\IndFol[p]$ and a generalized (possibly inhomogeneous)
screw motion around~$\affineleaf$.

\sep

\textsc{Chapter~\ref{Chap:Reducibility}} deals with questions of reducibility.
We show that --- as in the case of homogeneous representation --- existence of a
non-full regular leaf implies that the minimal affine subspace containing it
is invariant under $\Fol$.
Moreover, we examine under which conditions $\Fol$ splits off a Euclidean factor.

\sep

Finally, in \textsc{Chapter~\ref{Chap:Homogeneity}} we address homogeneity of $\Fol$.
First, we consider the quotient $\intermediatespace=\R^{k+n}/\IndFol$ and show
that --- provided $\IndFol$ is equidistant --- the image of $\Fol$ under the
natural projection is an equidistant foliation of $\intermediatespace$ and is
described by the same screw motion map as $\Fol$.
Reversing this construction we give new \emph{inhomogeneous} equidistant
foliations of $\R^n$.

We close with a homogeneity result for $\Fol$ if $\IndFol[p]$
(for one and hence all $p\in\affineleaf$) is homogeneous
and if its isometry group fulfills certain conditions, e.g.\ if it is
sufficiently small.
In particular $\Fol$ is homogeneous if $\IndFol[p]$ is given by either
\begin{itemize}
\item the orbits of an irreducible representation of real or complex type,
\item the orbits of an irreducible polar action,
\item the Hopf fibration of $\S{3}$ or $\S{7}$.

\end{itemize}

\subsection*{Acknowledgements}

This work could not have been accomplished without the help of several people.

First and foremost I would like to thank my advisor Prof.\ Dr.\ Ernst Heintze
for his constant encouragement and many fruitful discussions during the last years.
I would also like to thank Prof.~Dr.~Carlos Olmos for his hospitality and
friendly support during my stay in C\'ordoba in 2004 and many useful discussions.
To Dr.~Alexander Lytchak I am indebted for many helpful suggestions on the topic of
submetries and I would like to thank Prof.~Dr.~Burkhard Wilking for his
suggestions concerning the existence of an affine leaf.
Further thanks go to Prof.~Dr.~Jost-Hinrich Eschenburg, Dr.~habil.~Andreas Kollross
and Dr.~Kerstin Weinl for
many helpful discussions and to Dipl.~Math.~Walter Freyn for proofreading this
thesis --- any remaining errors are, of course, my own.


\chapter{Preliminaries}\label{Chap:Preliminaries}

In this chapter we introduce the concepts of \emph{Alexandrov spaces}, \emph{submetries}
and \emph{equidistant foliations}, that form the basis this thesis is built on.
We present several results arising from these concepts that will be used throughout
this work. Many of these are citations from literature, sometimes equipped with
a more accessible proof, but original work is included as well.

\section{Alexandrov Spaces}

The concept of Alexandrov spaces is a generalization of Riemannian manifolds.

We only give a brief outline of what an Alexandrov space is and present some
properties relevant to this work.
For a more detailed discussion of Alexandrov spaces we refer the reader
to~\cite{BBI}.

A metric space $X$ is called a \emph{length space} if the distance between any
two points is given by the infimum of the length of curves connecting these two
points. Consequently a curve whose length equals the distance between its
endpoints is called a \emph{shortest curve} and a locally shortest curve is
called a \emph{geodesic}. If we do explicitely say anything else we always assume
a geodesic to be parametrized by arc length.

An \emph{Alexandrov space} is a length space with a lower curvature bound~$\kappa$.
This means that small geodesic triangles are always thicker (i.e.\ points on any
side are at a greater or equal distance from the opposite vertex) than a comparison
triangle with the same side lengths in the model space $M_\kappa$, which is the
2-dimensional space form of constant curvature $\kappa$.

This implies an abundance of properties (some immediate from the definition,
others requiring rather sophisticated theory) showing that Alexandrov spaces are
indeed very similar to Riemannian manifolds.

\subsubsection*{Some useful results about Alexandrov spaces}

We present a short list of results about the geometry of Alexeandrov spaces,
which will be used throughout this thesis.

\begin{itemize}
\item Geodesics in Alexandrov spaces do not branch (otherwise this would result in
``thin'' triangles, cf. \cite[Chap.~4]{BBI}).
\item The Hausdorff dimension of an Alexandrov space is either an integer or infinity
(cf. \cite[Chap.~10]{BBI}).
\item Finite dimensional complete Alexandrov spaces are \emph{proper}
(i.e.\ closed boun\-ded subsets are compact) and \emph{geodesic}
(i.e.\ any two points can be connected by a shortest curve).

Moreover, an analogue of the Hopf-Rinow theorem holds (cf. \cite[Thm.~2.5.28]{BBI}).
\item Any $n$-dimensional Alexandrov space contains an open dense subset which
is an $n$-dimensional manifold (cf. \cite[Chap.~10]{BBI}).
\end{itemize}

\begin{Rem*}
Henceforth, if we talk about an Alexandrov space we will \emph{always} assume it
to be \emph{complete and finite dimensional}.
\end{Rem*}

In geodesic spaces we commonly use the notation $\Dist{x}{y}$ for the distance
between two points instead of $\dst{x}{y}$.

For a subset $A$ of a metric space $X$ we denote by $d_A\colon X\to\R_0^+$ the
distance function $d_A(p)=\dist{A}{p}$ relative to $A$.

\subsection*{Tangent Cones}

Let $X$ be an Alexandrov space and consider two geodesics $\alpha$ and $\beta$
emanating at some point~$p\in X$.
An immediate consequence of the lower curvature bound is that the angle formed by
$\alpha$ and $\beta$ at $p$ is well defined.

We consider the set $\tilde\Sigma_p$ of equivalence classes of geodesics emanating
from $p$ where two geodesics are identified if they form a zero angle.

\begin{Def}
The \emph{space of directions $\Sigma_p$ at $p$} is the completion of $\tilde\Sigma_p$
with respect to the angle metric.

The \emph{tangent cone} $T_pX$ of $X$ at $p$ is the metric cone~$\cone{\Sigma_p}$
over $\Sigma_p$.
\end{Def}

\begin{Rem*}
The space of directions of an $n$-dimensional Alexandrov space is a compact
$(n-1)$-dimensional Alexandrov space of curvature $\geq1$.
Consequently $T_pX$ is an $n$-dimensional Alexandrov space of nonnegative curvature.

Note that in general there may be directions at $p$ not represented by any geodesic.
\end{Rem*}

\begin{Def}
We call a point $x$ in an $n$-dimensional Alexandrov space $X$ \emph{regular}
if the space of directions~$\Sigma_x$ at $x$ is isometric to the Euclidean standard
sphere $\S{n-1}$, or equivalently if $T_xX$ is isometric to $\R^n$.
\end{Def}

\begin{Rem*}
Geodesics ending at a regular point $x$ can be extended beyond $x$ and for
any~$\xi\in\Sigma_x$ there is a geodesic starting at $x$ with direction $\xi$.

Thus at regular points $x$ we can define the exponential
map~$\exp_x\colon U\subset T_xX\to X$ in the same way as for Riemannian manifolds.

We point out that the set of regular points of $X$ contains a set which is open
and dense in $X$ (cf. \cite[Chap.~10]{BBI}).
\end{Rem*}

Remember that the metric cone over $\Sigma_p$ is the topological cone over
$\Sigma_p$, i.e.\ the set $[0,\infty)\times\Sigma_p/_\sim$ where we have identified
all points of the form~$(0,\xi)$, $\xi\in\Sigma_p$, equipped with the metric
$$
    \Dist{(t,\xi)}{(s,\eta)}=t^2+s^2-2\ip{\xi}{\eta}
$$
where $\ip{\xi}{\eta}=\cos\measuredangle(\xi,\eta)$.
This places an isometric copy of $\Sigma_p$ at distance~1 from the apex~0.

We present some further notation:

For $v=(t,\xi)\in T_pX$ and $s\geq0$ we denote by $sv$ the vector $(st,\xi)\in T_pX$.

We usually write $|v|$ as a shorthand for the distance $\Dist{v\,}{0}$ between
$v$ and the apex $0$ of the cone.

Let $\xi,\eta\in\Sigma_p$ be directions which enclose an angle $<\pi$ and let
$\gamma$ be a shortest curve in $\Sigma_p$ connecting them.
Then the cone over $\gamma$ can be embedded isometrically into $\R^2$, via $\phi$,
say.
Thus for $v=t\xi$ and $w=s\eta$ we define
$$
    v+w:=\phi^{-1}(\phi(v)+\phi(w)).
$$
Of course this depends on the choice of $\gamma$ and is only useful if $\gamma$
is unique.
Note, however, that we get the usual relation
$$
    |v+w|^2=|v|^2+|w|^2+2\ip{v}{w}
$$
where $\ip{v}{w}:=ts\ip{\xi}{\eta}$.

Finally if $A$ is a subset of $\Sigma_p$ we call
$\Set{\xi\in\Sigma_p}{\dist{\xi}{A}\geq\frac{\pi}{2}}$ the \emph{polar set}
of~$A$.

\section{Submetries}

Submetries are a generalization of the notion of linear projections and
Riemannian submersions to metric spaces.

\begin{Def}
Let $f\colon X\to Y$ be a mapping between metric spaces. Then $f$ is called a
\emph{submetry} if it maps metric balls in $X$ to metric balls of the same radius
in~$Y$.
\end{Def}

This simple property turns out to be rather rigid at least for submetries between
Alexandrov spaces. And we present in the following some interesting results about
submetries relevant to this thesis.
We refer the reader to \cite{Lyt} for a detailed discussion.

First note that we can characterize submetries by looking at the distance function
of fibres (cf.\ \cite[Lem.~4.3]{Lyt}):

\begin{Lem}
A mapping $f\colon X\to Y$ between metric spaces is a submetry if and only if
for any subset $A$ (possibly a single point) of $Y$ the equality
$$
    d_{f^{-1}(A)}=d_A\circ f
$$
holds.
\end{Lem}

We call a point $p\in X$ \emph{near} to~$x\in X$ (with respect to $f$) if
$\Dist{x}{p}=\dist{F_x}{p}$ where $F_x$ is the fibre of $f$ passing through~$x$.
We denote the set of points near to~$x$ by $N_x$.

A geodesic $\gamma$ emanating at $x$ will be called \emph{horizontal} if its
image under $f$ is a geodesic of the same length. Thus a shortest curve is horizontal
if and only if its start and endpoint are near to each other.

It should be noticed that many topological and geometric properties are inherited
by the base space of a submetry (cf. \cite[Prop.~4.4]{Lyt}).
We present only a few:

\begin{Prop}
Let $f\colon X\to Y$ be a submetry between metric spaces.
Then $Y$ is complete or connected or is a length space or has
curvature bounded below by $\kappa$ or has dimension $\leq n$ if $X$ has the
respective property.
\end{Prop}

Finally, we mention the following factorization property of submetries, which is
an immediate consequence of the definition (cf.\ \cite[Lem.~4.1]{Lyt}):

\begin{Lem}\label{Lem:factorizing submetries}
    Let $X,Y,Z$ be metric spaces and $f: X \ra Y$, $g: Y \ra Z$ be maps between them.
    Suppose that $f$ and $h:= g \circ f$ are submetries then so is $g$.
\end{Lem}

\begin{proof}
Let $\ball{r}{y}$ be some metric ball in $Y$, which is the image under $f$ of
some ball $\ball{r}{x}$ in $X$ since $f$ is a submetry.
Then $g(\ball{r}{y})=h(\ball{r}{x})=\ball{r}{h(x)}$.
\end{proof}

\subsection{Lifting}

With Riemannian submersions $p\colon M\to N$ it is possible to lift geodesics in the
base~$N$ to horizontal geodesics in $M$. This follows easily from the conditions
posed on the differential of the submersion.

However, this can be shown in a purely geometrical way as is done e.g.\ in \cite{GB}.
Using essentially the same arguments we see that these lifts exist in the case of
submetries as well:

\begin{Lem}\label{Lem:Horizontal lift of shortest paths}
Let $f\colon X\to \bar{X}$ be a submetry between Alexandrov spaces and let
$\bar\gamma\colon[0,l]\to \bar{X}$ be a shortest path of length $l$ between
two points $\bar{p}$ and $\bar{q}$.
\begin{enumerate}[(a)]
\item Let $p\in f^{-1}(\bar{p})$ then there exists a \emph{horizontal lift}~$\gamma$
of $\bar{\gamma}$ to $p$, i.e.\ a shortest path $\gamma\colon[0,l]\to X$ of the
same length such that $\gamma(0)=p$ and $f\circ\gamma=\bar{\gamma}$.
\item If $\bar\gamma$ can be extended beyond $\bar{p}$ as a shortest path then
the horizontal lift is unique.
\end{enumerate}
\end{Lem}

\begin{proof}
Assume for now that $\bar\gamma$ can be extended beyond $\bar{p}$.

\begin{enumerate}[(a)]
\item Since $f$ is a submetry
$\dist{f^{-1}(\bar{p})}{f^{-1}(\bar{q})} = \Dist{\bar{p}}{\bar{q}}$ and since
$f^{-1}(\bar{p})$ and $f^{-1}(\bar{q})$ are closed there is a
point~$q\in f^{-1}(\bar{q})$ such that $\Dist{p}{q}=\Dist{\bar{p}}{\bar{q}}$,
i.e.\ $q$ is near to $p$.

Let $\gamma\colon[0,l']\to X$ be a shortest path connecting $p$ and $q$.
Then $\length{\gamma}=l'=\Dist{p}{q}=\Dist{\bar{p}}{\bar{q}}=l$ and consequently
$f\circ\gamma$ is a curve of length at most $l$ connecting $\bar{p}$ and $\bar{q}$.
Hence it is a shortest curve.
Remember that since $\bar{\gamma}$ is extendible it is the unique shortest path
connecting those two points and so has to agree with $f\circ\gamma$.

\item Suppose there are two different lifts $\gamma_1$ and $\gamma_2$ to $p$.

Let $\bar\alpha\colon[-\eps,l]\to\bar{X}$ be an extension as a shortest path
of~$\bar{\gamma}$ and let $\bar{r}$ be the point~$\bar\alpha(-\eps)$.

We can now lift $\bar\alpha|_{[-\eps,0]}$ to $p$. Let us call this lift $\beta$
and its starting point~$r$. Then $r$ is near to $p$ so
$$
    \Dist{\bar{r}}{\bar{q}}=
    \Dist{\bar{r}}{\bar{p}}+\Dist{\bar{p}}{\bar{q}}=
    \Dist{r}{p}+\Dist{p}{q} \geq
    \Dist{r}{q}\geq
    \Dist{\bar{r}}{\bar{q}}
$$
where the last inequality holds because $f$ does not increase distances.

So, continuing $\beta$ by either $\gamma_1$ or $\gamma_2$ yields a shortest path
between $r$ and $q$ which agrees with the other at least up to $p$. But then
the $\gamma_i$ have to agree as well since in Alexandrov spaces geodesics do not
branch.
\end{enumerate}
To show (a) in general just choose some point $\bar{x}$ in the interior
of~$\bar\gamma$, take $x\in f^{-1}(\bar{x})$ near to $p$ and lift~$\bar\gamma$
to~$x$. This lift then has $p$ as one endpoint.
\end{proof}

\begin{Rem}\label{Rem:Horizontal lift of geodesic}
Of course Lemma~\ref{Lem:Horizontal lift of shortest paths} also holds for
geodesics instead of shortest paths. Since geodesics are locally shortest we can
lift these shortest paths and use the fact that the lifts at interior points of the
geodesic are unique.

Note that there is an even stronger lifting property
(Proposition~\ref{Prop:Horizontal lift of quasigeodesics}) if $X$ is a manifold.
\end{Rem}

\subsection{Differentials}\label{Sec:Differentials}

Several results in this work are based on examining the differential of a submetry.
So let us explain what we mean by differentiability and the differential
of a map between Alexandrov spaces.

\begin{Rem*}
The material presented in this section is mostly due to \cite{Lyt}. But since it
is nonstandard material we include it here and present it in a way more
suitable for the needs of this thesis.
\end{Rem*}

In \cite[p.44]{BGP} a Lipschitz function $f\colon X\to\R$ on a finite dimensional
Alexandrov space is said to be differentiable if its restriction to any geodesic
is differentiable (with respect to arc length) from the right.

This is generalized in \cite[Sect.~3]{Lyt} to Lipschitz maps $f\colon X\to Y$
between finite dimensional Alexandrov spaces.

\begin{Rem*}
In the following we will be using ultralimits. We refer the reader
to~\cite[Sect.~2.4]{KL} for a concise definition of ultralimits.
In short this concept allows us to coherently choose for any sequence $(x_j)$
in a compact space one of its limit points. This limit point is called the
ultralimit $\lim_\omega x_j$ of $(x_j)$ and depends on the particular choice
of the  nonprincipal ultrafilter $\omega$ on the integers.
\end{Rem*}

Using this \cite{KL} considers sequences of pointed metric spaces $(X_j,x_j)$ and defines
their ultralimits $\lim_\omega(X_j,x_j)$ as the set $X_\infty$ consisting of all
sequences $(y_j)$ with $y_j\in X_j$ such that $d_j(y_j,x_j)$ is uniformly bounded.
Then $x\in X_\infty$ is defined as $(x_j)$ and we get a pseudometric $d((y_j),(z_j))$
which is defined as the ultralimit $\lim_\omega d_j(y_j,z_j)$. After identifying
points $y,z\in X_\infty$ for which $d(y,z)=0$ this turns $(X:=X_\infty/_{(d=0)},x)$
into a pointed metric space.

\begin{Rem*}
If $(X_j,x_j)$ is a sequence of proper spaces converging in the pointed
Gromov-Hausdorff topology towards the proper space $(X,x)$ then for any $\omega$
the ultralimit $\lim_\omega(X_j,x_j)$ is isometric to $(X,x)$.
\end{Rem*}

The ultralimit approach has the advantage that we can extend this notion naturally
to maps between metric spaces:
Let $f_j \colon (X_j,x_j)\to(Y_j,y_j)$ be a sequence of Lipschitz maps with uniform
Lipschitz constant then the ultralimit $f:=\lim_\omega f_j$ is given by
$f((z_j))=(f_j(z_j))$.

Now let us look in particular at the tangent cone of a finite dimensional Alexandrov
space $X$: The tangent space $T_xX$ at $x$  is the pointed Gromov-Hausdorff limit
of the scaled spaces $(\frac{1}{r_j}X,x)$ for any positive sequence $(r_j)$
tending to zero. By $\lambda X$ we mean the space $X$ with the scaled metric
$\lambda \cdot d$.

\begin{Rem*}
The tangent space $T_xX$ defined in this way is isometric to the metric cone~$\cone\Sigma_x$ over the space of directions at~$x$
(cf. \cite[Sect.~10.9]{BBI}).
\end{Rem*}

Based on this \cite{Lyt} makes the following definition:

\begin{Def}
Let $f\colon X \to Y$ be a Lipschitz map between finite dimensional Alexandrov
spaces.  We consider for any positive sequence~$(r_j)$ tending to zero the
ultralimit~$\lim_\omega f_j$ of the sequence
$f_j:=f\colon(\frac{1}{r_j}X,x)\to(\frac{1}{r_j}Y,f(x))$.

We say $f$ is \emph{differentiable} at $x\in X$ if $\lim_\omega f_j$ does not
depend on the choice of~$(r_j)$ and call the resulting Lipschitz
map~$\diff[x]{f}\colon T_xX\to T_{f(x)}Y$ the \emph{differential} of $f$ at $x$.
\end{Def}

In detail $\diff[x]{f}$ is given in the following way: Let $p\in X$ be close to
$x$ and let $\gamma$ be a shortest path connecting $x$ to $p$ with direction $\xi$
at $x$.
Then considering that $(\frac{1}{r_j}X,x)$ converges to $T_xX$ we see that
$(\gamma(r_j\cdot|xp|))$ converges towards $|xp|\cdot\xi$ and consequently
$(f(\gamma(r_j\cdot|xp|)))$ tends to some $\eta$ in $T_{f(x)}Y$. If $\eta$ is
independent of $(r_j)$ then $\diff[x]{f}(|xp|\cdot\xi)=\eta$.

Note that by this property $\diff[x]{f}$ is \emph{homogeneous}, i.e.
$\diff[x]{f}(t\xi)=t\diff[x]{f}(\xi)$ for any nonnegative $t$.

\newpage
\subsubsection*{Application to Submetries}

\begin{enumerate}
\item By \cite[Prop.~3.7]{Lyt} $f\colon X\to Y$ is differentiable at
$x\in X$ if and only if for any $y\in Y$ with $y\neq f(x)$ the function
$d_y \circ f$ is differentiable, thus reducing the question of differentiability
to the case treated by \cite{BGP}.
\item From \cite[Lem.~4.3]{Lyt} we know that $f\colon X\to Y$ is a submetry if
and only if $d_{f^{-1}(y)}=d_y \circ f$ for any point $y$ in $Y$.
Since for any closed $A\subset X$ the function $d_A$ is differentiable outside~$A$
(cf.~\cite[p.44]{BGP}) this implies that submetries are differentiable.
\item If $f_j\colon(X_j,x_j)\to(Y_j,y_j)$ is a sequence of submetries then its
ultralimit is a submetry as well.
This is an immediate consequence of the definition of ultralimits
and shows that the differential of a submetry is itself a submetry between the
tangent spaces.

Moreover, the fibres of $f_j$ converge to the fibres of $f$
(cf.~\cite[Lem.a~4.6]{Lyt}).
\end{enumerate}

Thus the study of the differential of a submetry reduces to the study of homogeneous
submetries $f\colon\cone\Sigma\to\cone{S}$ between cones or simply to submetries
$f\colon\Sigma\to S$ where $\Sigma$ and $S$ have curvature $\geq1$.

We give some more results from \cite{Lyt} for this setting:

\begin{Prop}\label{Prop:vertical and horizonal vectors}
Let $\Sigma$ and $S$ be finite dimensional Alexandrov spaces of curvature $\geq1$
and let $f\colon\cone{\Sigma}\to\cone{S}$ be a homogeneous submetry.
Then the following assertions hold:

\begin{enumerate}[(a)]
\item The preimage $f^{-1}(0)$ of the apex is the cone over some totally convex
set $V\subset\Sigma$. The directions in $V$ are called \emph{vertical}.

\item Let $H$ be the polar set of $V$ with respect to $\Sigma$.
Then $\cone{H}$ consists just of the \emph{horizontal} vectors of $f$, i.e.\ those
$h\in\cone{\Sigma}$ such that $|f(h)|=|h|$.

\item For any~$x\in \cone{\Sigma}\setminus(\cone{V}\cup\cone{H})$ there are
unique $v\in\cone{V}$ and $h\in\cone{H}$ such that $x=h+v$, $\ip{h}{v}=0$ and
$f(x)=f(h)$.

\item The restriction $f\colon\cone{H}\to\cone{S}$ is a submetry.
\end{enumerate}

\end{Prop}

The proof for Proposition~\ref{Prop:vertical and horizonal vectors} can be found
in \cite[Prop.~6.4, Lem.~6.5, Cor.~6.10]{Lyt}. We give a detailed proof of part
(c) since this result will be essential later on.

\begin{proof}
First note that since $H$ is polar to $V$ there may be at most one shortest curve
in $\Sigma$ connecting $H$ and $V$ and passing through $\xi=\frac{x}{|x|}$.
Otherwise we could
combine two such geodesics in such a way as to produce a branch point.
So the notation $h+v$ is well defined.

Let $y=f(x)$ and let $c$ be the geodesic ray in $\cone{S}$ emanating at 0 (i.e.~$c(0)=0$)
and passing through $y$.
There is a unique horizontal lift $\gamma$  of $c$ through $x$ since $y$ lies in
the interior of $c$. Let $\tilde{\gamma}$ be the ray parallel to $\gamma$ and
emanating at 0, i.e.\ $\tilde{\gamma}(t)+\gamma(0)=\gamma(t)$.

We define $v:=\gamma(0)$, so $v$ is contained in $\cone{V}$ because
$f(\gamma(0))=c(0)=0$. Thus $\gamma(t)=\tilde{\gamma}(t)+v$ and so
$$
    f(\tilde{\gamma}(t)+v)=f(\gamma(t))=tf(\gamma(1)).
$$
Now as $f$ is 1-Lipschitz we get
\begin{equation}\label{Eq:vertizontal splitting}
    \big|f(\gamma(t))\,f(\tilde{\gamma}(t))\big| \leq
    \big|\gamma(t)\,\tilde{\gamma}(t)\big|=|v|
\end{equation}
but on the other hand using that $f$ is homogeneous and $\tilde{\gamma}$ is a ray
we get
$$
    \big|f(\gamma(t))\,f(\tilde{\gamma}(t))\big| =
    \big|(tc(1))\,f(t\tilde{\gamma}(1))\big| =
    t\big|c(1)\,f(\tilde{\gamma}(1))\big|
$$
for arbitrarily large $t$.
Using (\ref{Eq:vertizontal splitting}) this implies
$f(\gamma(t))=f(\tilde{\gamma}(t))$ for all $t\geq0$.

In particular choosing $t_0$ such that $\gamma(t_0)=x$ we define
$h:=\tilde{\gamma}(t_0)$. Then $h\in\cone{H}$ and $f(h)=f(x)$.

Finally, by construction $\gamma$ is perpendicular to the geodesic
ray~$\Set{tv}{t\geq0}$ and hence so is $\tilde\gamma$, i.e.\ $\ip{h}{v}=0$.
\end{proof}

\begin{Rem*}
Let $f\colon X\to Y$ be a submetry between Alexandrov spaces and consider
$\diff[x]{f}\colon\cone{\Sigma}_x\to\cone{\Sigma}_{f(x)}$.
The cone $\cone{V}_x$ is the tangent cone at $x$ of the fibre of $f$ containing $x$
and $\cone{H}_x$ is the tangent cone at $x$ of the set $N_x$ of points near to $x$
(cf. \cite[Chap.~5]{Lyt}).
\end{Rem*}

\section{Equidistant Foliations}

\begin{Def}
An \emph{equidistant foliation} of $\R^n$ is a partition $\Fol$ into complete, smooth,
connected, properly embedded submanifolds of $\R^n$ such that for any two
leaves~$F,G\in\Fol$ and $p\in F$ the distance $d_G(p)$ does not depend on the
choice of~$p\in F$.
Moreover, we demand the foliation to be smooth, i.e.\ any vector tangent to a leaf
can be locally extended to a vector field that is everywhere tangent to the leaves
of~$\Fol$.

The space $\base=\R^n/\Fol$ of the leaves of $\Fol$ bears the natural metric
$d_\base(F,G)=\operatorname{dist}_{\R^n}(F,G)$ and the canonical projection $\pi\colon\R^n\to\base$
is a submetry. The leaves of $\Fol$ are then the fibres of $\pi$.
\end{Def}

\begin{Rem}\label{Rem:Molino}
Note that this definition is a special case of that of a
\emph{singular Riemannian foliation} as given by \cite{Mol}:
A partition $\fol{L}$ of a Riemannian manifold into connected immersed
submanifolds such that
\begin{enumerate}[(a)]
\item any vector tangent to a leaf can be locally extended to a vector field
tangent to the leaves of $\fol{L}$, and
\item the foliation is \emph{transnormal}, i.e.\ every geodesic that is perpendicular
at one point to a leaf remains perpendicular to every leaf it meets.
\end{enumerate}
Note that transnormality characterizes local equidistance of the leaves --- and
indeed global equidistance if the leaves are properly embedded.

Concerning condition (a) observe that Lemma~\ref{Lem:Proj is submersion} already
implies that we may extend any vector tangent to a leaf to a local vector field
everywhere tangent to the leaves. However, this vector field need not --- a priori ---
be smooth at singular leaves.
\end{Rem}

It is, however, quite reasonable to stick to our more restrictive definition
as the additional structure we gain is very useful.
For example the submetry $\pi$ and the base space $\base$ have
some nice properties (cf. \cite[Prop.~12.8--12.11]{Lyt}):

\begin{Prop}
\begin{enumerate}[(a)]
\item Let $p$ be any point in $\R^n$. Then the set $N_p$ of points near to $p$
is convex.
\item Let $F$ be the leaf passing through $p$.
Then any direction perpendicular to $T_pF$ is horizontal,
and there is a positive number~$\eps$ such that for any direction $\xi_p\in\normal_pF$
there is a horizontal geodesic of length at least $\eps$ starting in the direction
of~$\xi_p$.

Consequently, at $\bar{p}:=\pi(p)$, for any $\bar{\xi}\in\Sigma_{\bar{p}}\base$
there is a geodesic in $\base$ emanating at $p$ of length at least $\eps$ with
direction $\bar{\xi}$.
\end{enumerate}
\end{Prop}

Moreover, we get from Chapter~13 of \cite{Lyt}:

\begin{Prop}
The set of regular points in $\base$ is a smooth Riemannian mani\-fold over which
$\pi$ is a smooth Riemannian submersion.
\end{Prop}

We call the fibres over regular points of $\base$ the \emph{regular leaves}
of~$\Fol$.

We introduce some notation commonly used when dealing with Riemannian submersions:

We denote the \emph{vertical space} $T_pF$ at $p\in F$ by $\Ver[p]$ and the
\emph{horizontal space} $\normal_pF$ by $\Hor[p]$. Note that $\Ver$ and $\Hor$
are (at least locally) spanned by smooth vector fields
(see Lemma~\ref{Lem:Proj is submersion}).
We denote the set of vertical and horizontal vector fields by $\VerF$ and $\HorF$
respectively.

Let $\kov{}{}$ be the standard Levi-Civita connection on $\R^n$, and
$\kovV{}{}$ and $\kovH{}{}$ its projections to $\Ver$ and $\Hor$ respectively.

The \emph{shape operator} $\shape{}{}$ of $F\in\Fol$ is as usual the 1-form
on $\Hor[F]$ with values in the symmetric endomorphisms  of $\Ver[F]$ that is
dual to the second fundamental form $\sff$ of $F$:
$$
    \shape{X}{V} = -\kovV{V}{X}, \qquad X\in\HorF, V\in\VerF.
$$

The \emph{integrability tensor} or \emph{O'Neill tensor} $\oneill{}{}$ is
the skew symmetric 2-form on $\Hor$ with values in $\Ver$, given by
$$
    \oneill{X}{Y}=\frac{1}{2}\ver{\lie{X}{Y}}=\kovV{X}{Y}, \qquad X,Y\in\HorF.
$$

A vector field $\xi$ on the regular part of $\Fol$ which is everywhere horizontal
and for which $\diff{\pi}\xi$ is a well defined vector field on the regular part of~$\base$
is called \emph{basic horizontal} or \emph{Bott-parallel}. We denote the set
of Bott-parallel vector fields by~$\BottF$.

Observe that on the regular part of $\Fol$ we have
$$
    \lie{\BottF}{\VerF} \subset \VerF
$$
and as a consequence
$$
    \kovH{V}{\xi}=\kovH{\xi}{V}=-\oneillAd{\xi}{V}, \qquad V\in\VerF, \xi\in\BottF
$$
where $\oneillAd{\xi}{}$ is the pointwise adjoint of $\oneill{\xi}{}$.

\begin{Rem}\label{Rem:Oneills Formula}
As a consequence of O'Neill's formula
(using the constant curvature
of~$\R^n$) the O'Neill vector fields $\oneill{\xi}{\eta}$ for $\xi,\eta\in\BottF$
have constant norm along the regular leaves of $\Fol$.
\end{Rem}

\subsubsection*{Lifting through singular leaves}

We are frequently in a situation where we want to lift a curve that is the projection
of a geodesic which at least starts horizontally. This means the start of the
projected curve is a geodesic but the whole curve may not be due to the fact that
there may be points in the base, such as the boundary, beyond which a geodesic
cannot be extended.

Such projections of geodesics which start horizontally are
\emph{quasigeodesics} (see for example \cite{PP94} for a concise definition and further
properties of quasigeodesics).
We only mention a few key properties (cf.\ \cite[Sect.~12.4]{Lyt}):

\begin{Prop}\label{Fact:Quasigeodesics}
Let $X$ be an Alexandrov space.
\begin{enumerate}[(a)]
\item For any $x\in X$ and $\xi\in\Sigma_x$ there is a quasigeodesic $\bar\gamma$
emanating from $x$ with direction $\xi$.

\item If there is a shortest curve $\gamma$ of length $l$ starting at $x$ with the
same direction $\xi$ then $\bar\gamma$ agrees with $\gamma$ up to length $l$.
\item If $X$ is the base space of a submetry $f\colon M\to X$ from a Riemannian
manifold then any quasigeodesic in~$X$ defined on a bounded
(not necessarily compact)
interval consists
of finitely many geodesic pieces.

\item Let $\gamma$ be a geodesic in $M$ starting horizontally, then $f\circ\gamma$
is a quasigeodesic in~$X$.
\end{enumerate}
\end{Prop}

This enables us to prove:

\begin{Prop}\label{Prop:Horizontal lift of quasigeodesics}
Let $f\colon M\to X$ be a submetry between a Riemannian manifold and an
Alexandrov space and let $\gamma\colon[0,l]\to M$ be a geodesic such that the
restriction of $\gamma$ to $[0,\eps]$ for some $\eps>0$ is horizontal.

Then for any $p'$ in the same fibre as $p:=\gamma(0)$ it is possible to lift~$f\circ\gamma$
as a geodesic to $p'$ and this lift is unique if the lift of
$f\circ\gamma|_{[0,\eps]}$ is unique.
\end{Prop}

First we show:

\begin{Lem}\label{Lem:submetries and antipodes}
Let $B$ be a connected Alexandrov space and $f,g\colon\S{n}\to B$ a~submetry with
$f(p)=g(p)$ for some point $p\in\S{n}$. Then $f(-p)=g(-p)$.
\end{Lem}

\begin{proof}
We use induction over the dimension $n$ of the sphere.
For $n=0$ there is nothing to show since $B$ has to be a single point.

So suppose our claim holds for $\S{k}$ with $k=0,\ldots, n-1$.
Let $v,w$ be unit vectors in $T_p\S{n}$, horizontal with respect to $f$ and $g$
respectively, such that $\diff[p]{f}(v)=\diff[p]{g}(w)$.
Denote by $\gamma_v$ and $\gamma_w$ the geodesics starting at $p$ with direction
$v$ and $w$ respectively.

We show that $f\circ\gamma_v=g\circ\gamma_w$. Then $f$ and $g$ agree at
$\gamma_v(\pi)=\gamma_w(\pi)=-p$.

Note that up to some maximal time $t_0$ the curves $f\circ\gamma_v$ and $g\circ\gamma_w$
are geodesics in $X$ starting at the same point in the same direction; hence they
agree at the beginning, up to the point
$\bar{q}:=f\circ\gamma_v(t_0)=g\circ\gamma_w(t_0)$.
Denote by $q_1$ and $q_2$ the
points $\gamma_v(t_0)$ and $\gamma_w(t_0)$ respectively and define
$$
    \tilde{v}:=\Ddt\gamma_v(t_0-t), \qquad \tilde{w}:=\Ddt\gamma_w(t_0-t).
$$

We can then identify the space of directions $\S{n-1}$ at $q_1$ with that at $q_2$
setting $\tilde{v}=\tilde{w}$.
Then $\diff[q_1]{f}, \diff[q_2]{g}\colon\S{n-1}\to\Sigma_{\bar{q}}B$ are submetries
agreeing at a point and hence, by induction, at its antipode.

Now remember that $\gamma_v$ and $\gamma_w$ are both quasigeodesics in $X$ consisting
of finitely many geodesic segments. Applying the above argument successively to
each of these segments finishes our prove.

Note that the only problematic case, i.e.\ $\Sigma_{\bar{q}}B$ not being connected,
can arise only when $n=1$ with $\Sigma_{\bar{q}}B=\S{0}$. But then $f\circ\gamma$
can be extended beyond $\bar{q}$, so $\bar{q}$ is not a hinge point of $f\circ\gamma$.
\end{proof}

\begin{proof}[Proof of Proposition~\ref{Prop:Horizontal lift of quasigeodesics}]
We only need to check what happens at the hinge points of the quasigeodesic~$f\circ\gamma$.

Let $t_0$ be the first time $\gamma$ meets a singular fibre of $f$.
Let $\gamma'$ be the horizontal lift to $p'$ of $f\circ\gamma|_{[0,t_0]}$.

Identifying the spaces of horizontal directions at $q:=\gamma(t_0)$ and
$q'=\gamma'(t_0)$ we get that $\diff[q]{f}, \diff[q']{f}\colon\S{k}\to T_{f(q)}X$
agree on the direction from which $\gamma$ and $\gamma'$ arrive and hence on their
respective antipodes.

This allows us to continue $\gamma'$ smoothly by a lift of the next geodesic segment
in $f\circ\gamma$. Repeating this for the remaining hinge points finishes the proof.
\end{proof}

\begin{Rem}\label{Rem:Vertical Spaces}
Define $F+\xi$ to be $\Set{p+\xi_p}{p\in F}$. If $F$ is a \emph{regular} leaf in
$\Fol$ and $\xi$ is Bott-parallel along $F$ Proposition~\ref{Prop:Horizontal lift of quasigeodesics}
implies that $F+\xi$ is a leaf of $\Fol$ and the smooth map $p\mapsto p+\xi$
between these leaves is surjective.
Note that it is bijective and hence a diffeomorphism, if $F+\xi$ is regular.

In particular the tangent space $T_{p+\xi}(F+\xi)$ is given by $\Set{v+\kov{v}{\xi}}{v\in T_pF}$.
\end{Rem}

Even if $F+\xi$ is singular the map $p\mapsto p+\xi$ is at least a submersion:

\begin{Lem}\label{Lem:Proj is submersion}
Let $F\in\Fol$ be regular. Then the map $P\colon F\to G:=F+\xi$ with $P(x)=x+\xi_x$
is a surjective submersion.
\end{Lem}

\begin{proof}
Observe that we can extend $\xi$ to be a Bott-parallel normal field in a
neighborhood of $F$ such that $F'+\xi=G$ for all leaves $F'$ in that neighborhood.
Using this we can also extend $P\colon p\mapsto p+\xi_p$ to the same neighborhood
renaming $P|_F\colon F \to G$ to $\tilde{P}$.

Note that for any point $p$ the differential $\diff[p]{P}$ is just the orthogonal projection
onto~$\Ver[p+\xi_p]$.

Now assume there is a point $p\in F$ such that
$\diff[p]{\tilde{P}}\colon\Ver[p]\to\Ver[q]$, with $q:=p+\xi_p$, is not surjective.
We show that $\diff{\tilde{P}}$ is nowhere surjective along $F$.

So take some $v\in\Ver[q]$ perpendicular to the image of $\diff[p]{\tilde{P}}$.
Then $v$, or rather its parallel translate to $p$, is contained in~$\normal_pF$
since $\ip{v}{x}=\ip{v}{\diff[p]{P}x}$ for any vector $x$ with base point $p$.
Let $\eta$ be the extension of $v$ to a Bott-parallel normal field along~$F$.

We get
$$
    \diff{P}\eta = \eta + \kov{\eta}{\xi} =
        \lt(\eta + \kovH{\eta}{\xi}\rt) + \oneill{\eta}{\xi}
$$
and $\kovH{\eta}{\xi}$ is again a Bott-parallel normal field along $F$.
By Remark~\ref{Rem:Oneills Formula} the norm of $\diff{P}\eta$ is constant along~$F$,
which implies that $\diff{P}\eta=\eta$ for any point in $F$
since $\diff{P}$ is an orthogonal projection at every point.

Hence, the differential of $\tilde{P}$ is nowhere surjective. But since
$\tilde{P}\colon F\to G$ is a surjective map its singular values should be
a set of measure zero in $F$ by Sard's Theorem.
\end{proof}

Using Proposition~\ref{Prop:Horizontal lift of quasigeodesics} we can prove the
following rigidity result for the regular leaves of $\Fol$
(based on the idea of \cite[Lem.~6.1]{HLO} for the case of Riemannian submersions).

\begin{Prop}\label{Prop:Principal curvatures are constant}
Let $\Fol$ be an equidistant foliation of $\R^n$ and $\pi\colon \R^n\to \base$
the corresponding submetry.
Then for any regular leaf~$F$ the principal curvatures in the direction
of Bott-parallel~$\xi$ are constant along $F$.
\end{Prop}

\begin{proof}
Let $\lambda\neq0$ be an eigenvalue of $\shape{\xi}{}$ at $p\in F$ and $||\xi_p||=0$.
Let $v\in T_pF$ with $||v||=1$ be a corresponding eigenvector.

\begin{figure}[h]
    \begin{center}
\includegraphics{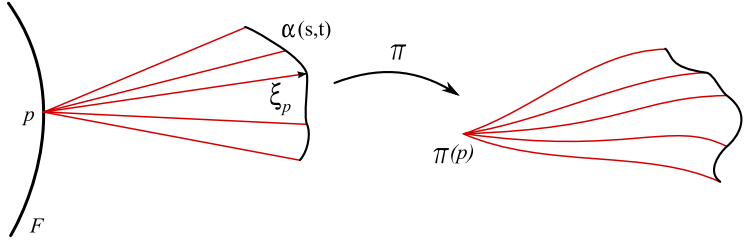}
    \caption[Main curvatures are constant.]
        {Main curvatures are constant along regular leaves.}
    \label{Fig:Main curvatures}
    \end{center}
\end{figure}

Consider the geodesic $\gamma_p(t)=t\xi_p$ and its horizontal variation
$$
    \alpha_p(s,t)=t\xi_p-st(\oneillAd{\xi}{v})_p
$$
yielding the Jacobi field
$$
    J_p(t):=\ddt[s]\alpha_p(s,t)=-t\left(\oneillAd{\xi_p}{v_p}\right)_{\gamma(t)}
$$
along $\gamma_p$.
Denote the leaf passing through $\gamma(t)$ by $F_t$ and remember that
$T_{\gamma(t)}F_t$ is spanned by $\set{v_i+t\kov{v_i}{\xi}}$ if $\set{v_i}$ is a basis
of $T_pF$.

In particular this implies for $t=1/\lambda$ that
$$
    J_p\lt(\frac{1}{\lambda}\rt)=
    -\frac{1}{\lambda}\oneillAd{\xi_p}{v_p}=
    v_p+\frac{1}{\lambda}\kov{v_p}{\xi}
$$
is vertical at $\gamma(\frac{1}{\lambda})$.

Now $\bar\alpha(s,t)=\pi\circ\alpha_p(s,t)$ is a variation of $\bar\gamma=\pi\circ\gamma_p$
by quasigeodesics and we can lift this variation to any point $q\in F$.
Thus we get the $\alpha_q(s,t)=q+t\xi_q-ts\eta$ where $\eta$ is the Bott parallel
continuation along $F$ of $\oneillAd{\xi_p}{v_p}$.

Note that the image of $\oneillAd{\xi}{}$ is equal to the image
of~$\oneillAd{\xi}{\oneill{\xi}{}}$ which is Bott parallel for $\xi\in\BottF$
(cf. Remark~\ref{Rem:Onormal is Bott-parallel}).
Hence there exists $w_q\in T_qF$ such that $\eta_q=\oneillAd{\xi_q}{w_q}$.

As a consequence of this lifting
$J_q\lt(\frac{1}{\lambda}\rt)=-\frac{1}{\lambda}\oneillAd{\xi_q}{w_q}$
is vertical at $\gamma_q(\frac{1}{\lambda})$ which means that there is a $v_q\in T_qF$
such that
$$
    J_q\lt(\frac{1}{\lambda}\rt) =
    v_q+\kov{v_q}{\xi} =
        \left(I-\frac{1}{\lambda}\shape{\xi_q}{}\right)v_q -\frac{1}{\lambda}\oneill{\xi_q}{v_q}.
$$
In particular $\left(I-\frac{1}{\lambda}\shape{\xi_q}{}\right)v_q$ vanishes,
which proves our claim.
Note that by continuity of the principal curvatures their multiplicities are
constant along~$F$ as well.
\end{proof}

\begin{Rem*}
A generalization of this result to singular Riemannian foliations has recently
been proved by Alexandrino and T\"oben (cf.\ \cite{AT}).
\end{Rem*}

\chapter{Existence of an Affine Leaf}\label{Chap:Existence}

Gromoll and Walschap show in \cite{GW:1} that a regular equidistant foliation
always has an affine leaf. To be more precise they show that the space of leaves
has a soul, which is a point, and that the leaf corresponding to the soul is an
affine space.

In Section~\ref{Sec:A soul construction} we show that it is possible to perform the
same soul construction for singular foliations as well and in
Section~\ref{Sec:Submetries onto compact Alexandrov spaces} we prove that
the soul in the singular case also has to be a point.
The approach used in the latter case is completely different to \cite{GW:1} since
their argument uses the spectral sequence for the homology of the fibration,
which does not work at all in the singular setting.

Thus we get:

\begin{Thm}\label{Thm:fibre over soul is totally geodesic}
    Let $\Fol$ be an equidistant foliation of $\R^n$ with $\pi\colon\R^n\to\base$
    the corresponding submetry.
    Then $\base$ has a soul $S$ which is a single point and the fibre over $S$
    is an affine subspace of $\R^n$.

    In short, $\Fol$ contains a leaf which is an affine subspace (possibly a single
    point) of $\R^n$.
\end{Thm}

\section{A Soul Construction}\label{Sec:A soul construction}

We will first use the Cheeger-Gromoll soul construction
(cf.~\cite{CE})
to arrive at a totally convex, compact subset of $\base$ without boundary.

We will, however, concentrate on lifting this construction to $\R^n$ since we
are more interested in $\pi^{-1}(S)$ than in the soul $S$ itself.

\sep

Remember that a \emph{ray} $\gamma$ in a length space is a unit
speed geodesic defined on~$[0,\infty)$ such that any restriction
$\gamma|_{[0,T]}$ is a shortest path. By a ray in $\R^{n}$ we
will mean throughout this section a horizontal one (with respect to $\Fol$).
The following lemma ensures the existence of rays.

\begin{Lem}\label{Lem:existence of ray}
    For any point $p$ in a locally compact, complete, noncompact
    length space $X$ there is a ray $\gamma$ starting at $p$.
\end{Lem}

\begin{proof}
    Since $X$ is not compact it cannot be bounded
    (cf.\ the Hopf-Rinow-Cohn-Vossen Theorem \cite[Thm.~2.5.28]{BBI}).
    So let $(p_n)$ be a sequence in $X$ with
    $\Dist{p}{p_n}$ tending to infinity.
    Consider the sequence $(\gamma_n)$ of shortest paths,
    connecting $p$ to $p_n$ and denote
    by $\gamma_n^T$ their restriction to $[0,T]$.
    By the compactness of $\overline{\ball{T}{p}}$ an
    Arzela-Ascoli type argument (cf. \cite[Thm 2.5.14]{BBI})
    yields the uniform convergence of a subsequence of  $(\gamma_n^T)$
    towards some curve $\gamma^T$.
    However, in a length space, the limit of a sequence of shortest
    paths is itself a shortest path
    (cf. \cite[Prop.~2.5.17]{BBI}).

    By increasing $T$ and passing on to subsequences we arrive at a
    curve $\gamma\colon\R_0^+\to X$ starting at $p$
    and the restriction of $\gamma$ to any $[0,T]$ is a shortest path.
\end{proof}

Let $\gamma$ be a ray starting at some point $p_0$ of
$\base$. We define $B_\gamma$ to be the horosphere $\union_{t>0}\ball{t}{\gamma(t)}$ and
$C_\gamma := \base \setminus B_\gamma$.
Finally let $C$ be the intersection of all $C_\gamma$ where $\gamma$
ranges over all rays starting in $p_0$.

\begin{Rem*}
    It is easy to check that $C$ is totally convex by simply using the
    same proof as in the manifold case (cf. \cite[pp.\ 135f]{CE}).
    The essential ingredient there is Toponogov's Theorem, which
    holds for Alexandrov spaces as well (cf. \cite[p.~360]{BBI}).
\end{Rem*}

\begin{Rem*}
    Note that $C$ is nonempty since it contains $p_0$ and closed
    since the $B_\gamma$ are all open.
    Clearly $C$ is also compact. If it were not, we could find a
    ray starting at $p_0$ and lying in $C$ by the argument used in
    the proof of Lemma~\ref{Lem:existence of ray} using the fact
    that $C$ is closed.
    But by definition of $C$ no point of this ray --- apart from $p_0$ ---
    is contained in $C$.
\end{Rem*}

We will now pass on to the lift of this construction.
For any lift $\tilde{\gamma}$ of a ray
$\gamma$ starting in $p_0$ we define $B_{\tilde{\gamma}}\subset\R^n$
in analogy to $B_\gamma\subset\base$.
Note that the $B_{\tilde{\gamma}}$ are
open halfspaces of $\R^{n}$.

Denote by $\tilde{B}_\gamma$ the
union of the $B_{\tilde{\gamma}}$ where $\tilde{\gamma}$ ranges
over all lifts of $\gamma$ along $F_0$, and by $\tilde{C}_\gamma$ its complement.
Finally let $\tilde{C}$ be the intersection of the sets~$\tilde{C}_\gamma$.

Obviously the latter are closed and convex being the intersection
of closed halfspaces and hence so is $\tilde{C}$.

\begin{Prop}
    The set $\tilde{C}$ is the preimage of $C$.
\end{Prop}

\begin{proof}
    Let $q$ be any point in $\R^n\setminus\tilde{C}$. That means $q$ is contained in
    a ball $q\in\ball{t}{\tilde{\gamma}(t)}$, where $\tilde\gamma$ is a horizonal
    ray emanating from $F_0$.
    But since $\pi$ is a submetry this implies
    $\pi(q)\in\ball{t}{\gamma(t)}$, with $\gamma=\pi\circ\tilde\gamma$,
    so $q$ cannot lie in $\pi^{-1}(C)$.

    On the other hand, consider any $q\in\R^n\setminus\pi^{-1}(C)$.
    Then $\pi(q)$ must lie in some $\ball{t}{\gamma(t)}$
    for a ray $\gamma$ starting at $p_0$.

    Now take a lift $\tilde\gamma$ of $\gamma$ such that $\tilde\gamma(t)$ is
    near to $q$, i.e.\
$$
    \Dist{q}{\tilde\gamma(t)}=\dist{q}{\pi^{-1}(\gamma(t))} = \Dist{\pi(q)}{\gamma(t)}.
$$
    Thus $q\in\ball{t}{\tilde{\gamma}(t)}$ which implies that $q$ is not contained in
    $\tilde{C}$.
\end{proof}

\begin{Rem*}
Thus $\tilde{C}$ is \emph{foliated} by $\Fol$, i.e.\ any leaf of $\Fol$
intesecting~$\tilde{C}$ is contained in~$\tilde{C}$.
In particular $\tilde{C}$ it is nonempty.
\end{Rem*}

In fact this is also true of its boundary
but, since $\tilde{C}$ may have empty interior in $\R^{n}$ we
have to find the right notion of ``boundary'' first.

Let $V^m$ be the unique affine subspace of minimal dimension $m$
such that $\tilde{C}$ is contained in $V$. We will denote the
interior and boundary of a set $X\subset V$ with respect to $V$ by
$\inn[V]{X}$ and $\rand[V]{X}$ respectively.

The convexity of $\tilde{C}$ implies that $\inn[V]{\smash{\tilde{C}}}$ is nonempty:
By definition of $V$ we may choose $m+1$ points $q_0,\ldots,q_m\in\tilde{C}$ such
that the vectors $q_1-q_0,\ldots,q_m-q_0$ are linearly independent. But then the
convex hull of these points has nonempty interior in $V$ and is contained in~$\tilde{C}$.

\begin{Rem*}
    Thus it makes perfect sense to call $m$ the \emph{dimension} of
    $\tilde{C}$.
\end{Rem*}

\begin{Lem}\label{Lem:border consists of fibres}
    Let $A$ be a closed, convex subset of $\R^{n}$ foliated by~$\Fol$.
    Moreover, let $V$ be
    the the minimal affine subspace of $\R^{n}$ containing $A$.
    Then $\rand[V]{A}$ --- if it is non\-empty --- is also foliated by~$\Fol$.

\end{Lem}

\begin{proof}
We have to make sure that fibres of $\pi$ that contain a boundary point of $A$
are themselves completely contained in the boundary of $A$.

\parpic{\includegraphics[height=5cm]{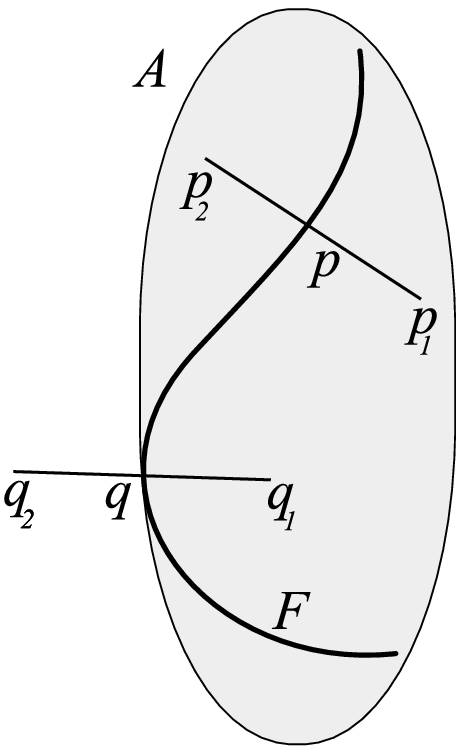}}
    So, suppose there is a leaf $F\subset A$ and two points $p,q\in F$ such that
    $p\in\inn[V]{A}$ and $q\in\rand[V]{A}$.

    By convexity of $A$ and since $F$ is smooth there is a geodesic $\gamma$ in
    $V$ passing through $q$, perpendicular to $F$ such that $\gamma(0)=q$ and
    $q_1:=\gamma(\eps)\in\inn[V]{A}$ and $q_2:=\gamma(-\eps)\notin A$ for ~$\eps>0$
    sufficiently small.
    We know that the line segment $[q_1q_2]$ is mapped by $\pi$
    onto a quasigeodesic in $\base$, which by Proposition~\ref{Prop:Horizontal lift of quasigeodesics}
    can be lifted to a geodesic $\gamma'$ passing through $p$.
    Denote by $p_1$ and $p_2$ the points $\gamma'(\eps)$ and $\gamma'(-\eps)$
    respectively.

    Since $[qq_1]$ is contained in $A$ so is $[pp_1]$ as $A$ is foliated by~$\Fol$.
    Hence $\gamma'$ is a line segment in $V$ and so for small $\eps$ the point~$p_2$
    is contained in $A$, which is a contradiction because $q_2$ and $p_2$ lie in the
    same leaf.
\end{proof}

Using this last result we can now continue the construction
recursively until we end up with a compact set in $\base$ the
preimage of which is an affine subspace of~$\R^{n}$.
To be more precise we set $\tilde{C}(1):=\tilde{C}$ and  construct $\tilde{C}(n+1)$
from $\tilde{C}(n)$ in the following way:

We will show inductively that $\tilde{C}(n)$ is again closed, convex and
foliated by~$\Fol$. Denote its dimension by $m(n)$
and write $\rand{\tilde{C}(n)}$ for its boundary with respect to the
$m(n)$-dimensional affine subspace containing it. If this boundary
is nonempty let $\tilde{C}(n+1)$ be the set of those points in
$\tilde{C}(n)$ whose distance from $\rand{\tilde{C}(n)}$ is
maximal.

More formally: For $p$ in $\tilde{C}(n)$ define $\rho_n(p)$ to be the distance
function $d_{\rand{\tilde{C}(n)}}(p)$ relative to $\rand{\tilde{C}(n)}$ and let
$R(n)$ be the maximum of $\rho_n$ on $\tilde{C}(n)$. Then $\tilde{C}(n+1)$ is the
$R(n)$-level set of $\rho_n$.

\begin{Rem*}
    Note the equality
$$
    \rho_n = d_{\rand{\tilde{C}(n)}}=d_{\pi(\rand{\tilde{C}(n)})}\circ\pi.
$$
    Since $\tilde{C}(n)$ is closed and foliated by~$\Fol$ we get that
    $\pi(\tilde{C}(n))$ is closed and thus compact being a subset of $C$.
    Hence, $\rho_n$ does indeed have a maximum, which is positive since
    $\tilde{C}(n)$ has nonempty interior.
\end{Rem*}

\begin{Prop}
    For any $n\in\N$ the set $\tilde{C}(n+1)$ (if defined) is closed, convex and
    foliated by~$\Fol$.
    Moreover, if $\rand{\tilde{C}(n+1)}$ is nonempty, then it too is foliated by~$\Fol$.
    Finally, the dimension of $\tilde{C}(n+1)$ is strictly
    less than that of $\tilde{C}(n)$.
\end{Prop}

\begin{proof}
    Obviously, $\tilde{C}(n+1)$ is closed as it is a level set of
    $\rho_{n}$.
\parpic{\includegraphics[width=4.7cm]{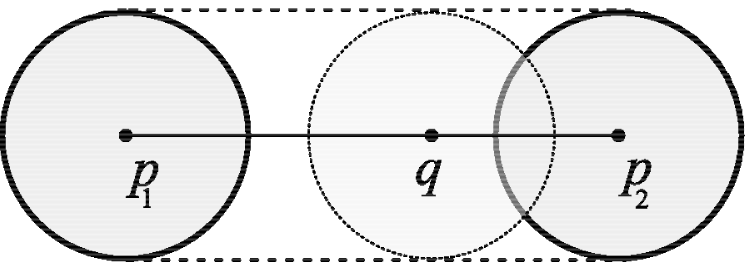}}
    To show its convexity assume $p_1,p_2$ to lie in $\tilde{C}(n+1)$.
    By definition, $\ball[V]{R(n)}{p_i}$ is then contained in
    $\tilde{C}(n)$, where $V$ is the minimal affine subspace
    containing $\tilde{C}(n)$. By the latter's convexity, the
    convex hull of the two balls is also contained in
    $\tilde{C}(n)$ and hence also the balls
    $\ball[V]{R(n)}{q}$ where $q$ is any point on the line
    segment $[p_1p_2]$. Thus, $[p_1p_2]$ is contained in
    $\tilde{C}(n+1)$.

    We now show that $\tilde{C}(n+1)$ is foliated by~$\Fol$.
    We begin by showing this property for the auxiliary set
$$
    \hat{C}(n+1):=\Set{p\in\R^{n}}{\dist{p}{\rand{\tilde{C}(n)}}=R(n)}.
$$
    Now
$$
    \dist{p}{\rand{\tilde{C}(n)}} = \min_F (d_F(p)),
$$
    where the minimum is taken over all leaves $F$ in
    $\rand{\tilde{C}(n)}$.
    As we have observed before, due to $\pi$ being a submetry we get
    $d_F(p) = \Dist{\pi(p)}{\pi(F)}$, which is constant along the leaf
    through~$p$. But this also holds for the minimum over all
    leaves $F$ in $\rand{\tilde{C}(n)}$, so for $p\in \hat{C}(n+1)$
    the whole leaf through $p$ is contained in this set.

    Observe that $\tilde{C}(n+1) = \hat{C}(n+1) \cap \tilde{C}(n)$
    and the intersection of two sets foliated by~$\Fol$ is also foliated.

    Then $\rand{\tilde{C}(n+1)}$ being foliated by~$\Fol$ is an immediate consequence
    of Lemma~\ref{Lem:border consists of fibres}.

    Obviously $m(n+1) \leq m(n)$, so assume equality holds.
    Then $\tilde{C}(n+1)$ has interior points
    with respect to the minimal affine subspace $V$ containing
    $\tilde{C}(n)$. Thus, $\tilde{C}(n+1)$ contains some ball
    $\ball[V]{\eps}{p}$. But clearly there are points in this ball
    that are closer to the boundary of $\tilde{C}(n)$ than $p$,
    which is a contradiction.
\end{proof}

This implies that our recursive construction terminates at some
$n$ --- at the very latest, when $\tilde{C}(n)$ is a point. The
final $\tilde{C}(n)$ then is a closed, convex subset of $\R^n$ without
boundary, i.e.\ an affine subspace, foliated by~$\Fol$. So,
restricting ourselves to this subspace, we have a submetry from a
Euclidean space onto $\pi(\tilde{C}(n))$.

\begin{Rem*}
    In $\R^n$ convexity and total convexity are the same. Hence, the set~$\pi(\tilde{C}(n))$,
    i.e.\ the \emph{soul} of $\base$
    is a totally convex subset of $\base$ (since we can lift geodesics)
    and so is again an Alexandrov space of nonnegative curvature.

    Observe that indeed any nonnegatively curved finite dimensional Alexandrov
    space~$X$ that is complete and unbouned has a soul,
    which can be obtained by the same construction as above. The only ingredient
    still needed in that construction is the fact that $C(n)$ is convex, which
    follows from the fact that the distance to $\rand{C}(n-1)$ is concave.
    This was proven
    (together with an Alexandrov space version of the soul theorem) by Perelman
    in 1991. A proof of the above mentioned concavity result can be found in
    \cite[Thm.~1.1(3B)]{AB} (as far as we know Perelman's result still exists only
    as a preprint).
\end{Rem*}

\section{Submetries onto Compact Alexandrov Spaces}
\label{Sec:Submetries onto compact Alexandrov spaces}

Let $\Fol$ be an equidistant foliation of $\R^n$ and assume the space of leaves~$\base$
to be compact.
We will show that $\base$ has to be a point.

Assume, for now, that $\base$ is not a point. Since $\base$ is
compact it has finite diameter $\diam{\base} > 0$. So for any
leaf $F$ of $\Fol$ the closed $\diam{\base}$-tube
$$
    \tube{\diam{\base}}{F}:=
        \Set{p\in\R^n}{d_F(p)\leq\diam{\base}}
$$
around $F$ is $\R^{n}$.

Consider a regular leaf $F$. Let $\xi_p$ be a unit normal
vector in $\Hor[p]$, $p\in F$, and denote by $\xi$ its Bott-parallel
continuation along $F$. We denote by $F_t$ the leaf $F+t\xi$
through $p_t:=p+t\xi_p$.
Note that if $F_t$ is regular then $\xi_t$ with $\xi_t(q+t\xi_q):=\xi(q)$ is Bott-parallel
along $F_t$.

\begin{Rem*}
Using Proposition~\ref{Fact:Quasigeodesics}
we can always make sure that
$F_t$ is regular by passing from $t$ to $t+\eps$, if necessary, for sufficiently
small $\eps>0$.
\end{Rem*}

We will now express $\shape{\xi_t}{}$ and $\oneillAd{\xi_t}{}$ on
$F_t$ in terms of $\shape{\xi}{}$ and $\oneillAd{\xi}{}$ on $F$:

Let $\gamma$ be a smooth curve on $F$ with $\gamma(0)= p$, $\dot{\gamma}(0)=v$
and denote by $\gamma_t$ its Bott-parallel translate
$\gamma_t(s):= \gamma(s)+t\xi(\gamma(s))$.
Recall from Remark~\ref{Rem:Vertical Spaces} that $\gamma_t$ is a smooth curve
on $F_t$  and we get
$\gamma_t(0)=p_t$ and  $\dot{\gamma}_t(0)=v+t\kov{v}{\xi}=:v_t$.

Using this we calculate
\begin{equation*}
    \shape{\xi_t}{v_t} = - \ver{\lt(\kov{v_t}{\xi_t}\rt)}
    = - \ver{\left(\ddt[s] \xi_t(\gamma_t(s))\right)}
    = - \ver{\lt(\kov{v}{\xi}\rt)}
\end{equation*}
and
\begin{equation*}
    \oneillAd{\xi_t}{v_t} = - \hor{\lt(\kov{v_t}{\xi_t}\rt)}
    = - \hor{\lt(\ddt[s] \xi_t(\gamma_t(s))\rt)}
    = - \hor{\lt(\kov{v}{\xi}\rt)}
\end{equation*}
with the vertical and horizontal parts taken with respect to $F_t$. Since
\begin{equation}\label{Eq:norm(v_t)}
    \begin{split}
    \norm{v_t}^2 &=
    \norm{v-t\shape{\xi}{v}-t\oneillAd{\xi}{v}}^2 \\
            &= \norm{(I-t\shape{\xi}{})v_t}^2 + t^2\norm{\oneillAd{\xi}{v_t}}^2\\
            &= \norm{v}^2 - 2t\ip{v}{\shape{\xi}{v}} + t^2\norm{\shape{\xi}{v}}^2
                + t^2\norm{\oneillAd{\xi}{v}}^2
    \end{split}
\end{equation}
we see that both
$\shape{\xi_t}{\frac{v_t}{\norm{v_t}}}$ and
$\oneillAd{\xi_t}{\frac{v_t}{\norm{v_t}}}$ tend to zero as $t$
goes to infinity.

Hence the leaves $F_t$ become more ``flat'' as $t$ increases. To formalize
this we introduce the following notation:

For any leaf $G\in\Fol$ let $\ball[G]{R}{p}$ be the intrinsic metric ball in $G$
around $p$ of radius $R$.
Furthermore, we will
denote by $\plane{p}{\xi}$ the hyperplane through $p$ with normal
vector $\xi$ and by $\plane[\eps]{p}{\xi}$  the $\eps$-tube
around $\plane{p}{\xi}$, i.e.\
$\plane[\eps]{p}{\xi}=\Set{x\in\R^n}{\,|\ip{x-p}{\xi}|<\eps}$.

\begin{Prop}\label{Prop:fibre is contained in thick hyperplane}
    Let $R$ and $\delta$ be positive, then for sufficiently
    large $t$ the closed intrinsic balls $\closure{\ball[F_t]{R}{p_t}}$ are
    contained in $\plane[\delta]{p_t}{\xi_t}$.
\end{Prop}

\begin{figure}[h]
    \begin{center}
    \includegraphics[width=8cm]{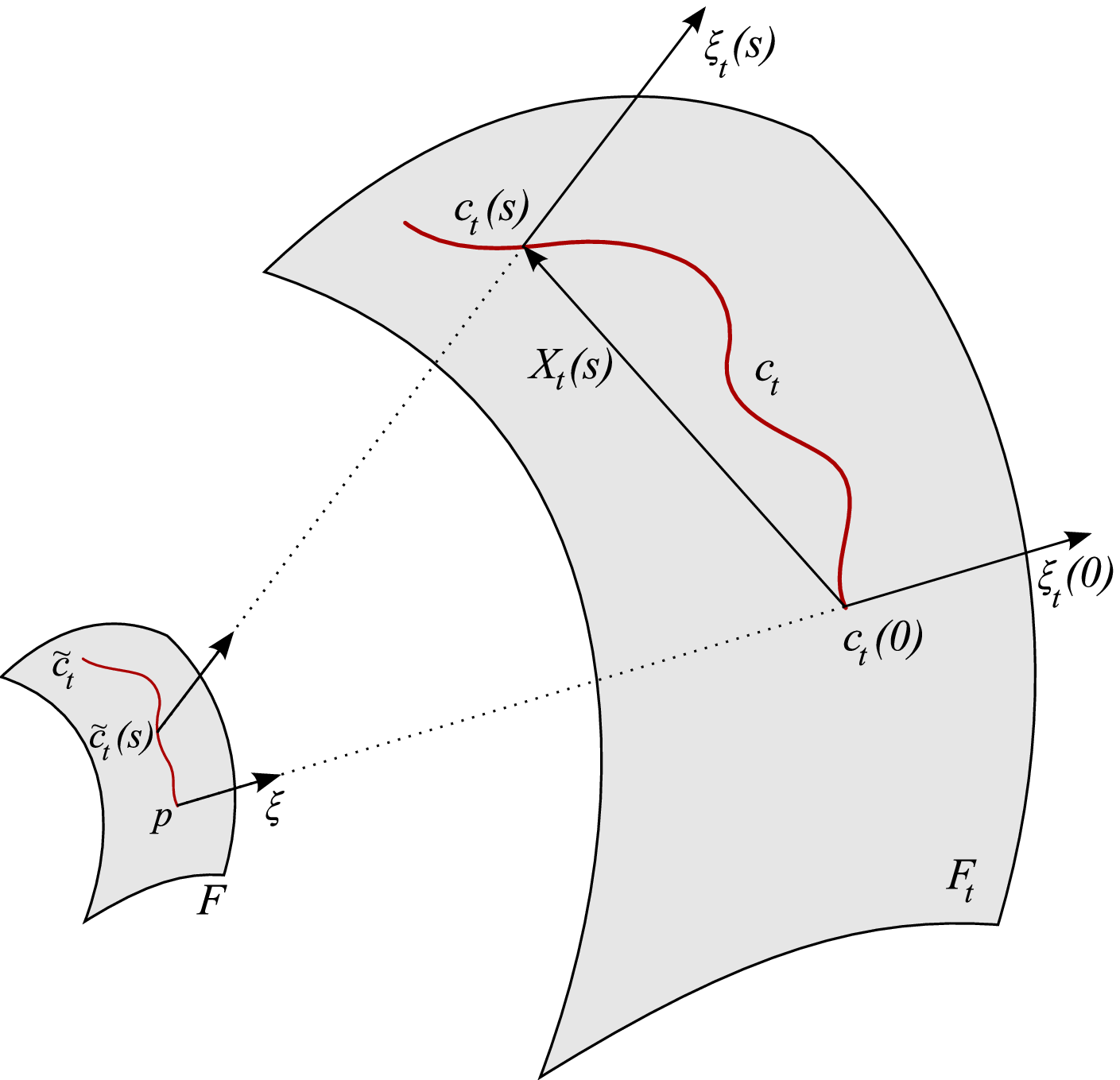}
    \caption[Illustrating the proof of Proposition~\ref{Prop:fibre is contained in thick hyperplane}]
        {In the proof of Proposition~\ref{Prop:fibre is contained in thick hyperplane}
        we can pull back the construction on $F_t$ to the original leaf $F$.}
    \end{center}
\end{figure}

\begin{proof}
    Let $c_t$ be a curve parameterized by arclength on $F_t$ starting in $p_t$.
    We define $X_t(s):=c_t(s)-p_t$
    and $\xi_t(s):=\xi_t(c_t(s))$.
    So, we only have to check that for sufficiently large $t$ the estimate
$$
    |\ip{X_t(s)}{\xi_t(0)}| < \delta
$$
    holds for all $s \leq R$.

    Observe first that we can pull back this construction to $F$ via the map
    $q\mapsto q+t\xi$. So, there is a curve $\tilde{c}_t$ on $F$ such that
    $c_t(s)=\tilde{c}_t(s)+\tilde{\xi}_t(s)$ where we define $\tilde{\xi}_t(s)$
    to be $\xi(\tilde{c}_t(s))$.

    {\bf Part 1.}
    The main step is to show that
    $\dot{\xi}_t(s)=\dot{\tilde{\xi}}_t(s)=\kov{\dot{\tilde{c}}_t(s)}{\xi}$
    tends uniformly to zero as $t$ goes to infinity.
    We first show this convergence pointwise:

    For any fixed $s\in[0,R]$ we apply
    Equation~(\ref{Eq:norm(v_t)}) to our situation:
\begin{equation}
\label{Eq:1=norm(v_t)}
    1=\norm{\dot{c}_t(s)}^2=\norm{(I-t\shape{\xi}{})w_t}^2+t^2\norm{\oneillAd{\xi}{w_t}}^2
\end{equation}
    where we have used $w_t$ as a shorthand for $\dot{\tilde{c}}_t(s)$.
    Obviously this implies that $\oneillAd{\xi}{w_t}$ tends to zero as $t$ goes to infinity.

    On the other hand the eigenvalues of $(I-t^2\shape{\xi})$ are $1-t^2\lambda_i$ where
    the $\lambda_i$ are the eigenvalues of $\shape{\xi}{}$. So, for any $\lambda_i\neq0$
    we get $1-t^2\lambda_i\to\pm\infty$ as $t\to\infty$ and hence the
    projection~$(w_t)_i$ of $w_t$ to the eigenspace of $\shape{\xi}{}$ corresponding
    to $\lambda_i$ tends to zero as $t$ goes to infinity.
    Note that this argument only works because the eigenvalues of $\shape{\xi}{}$
    are constant along $F$.

\begin{Rem}\label{Rem:kov(xi) is bounded}
    Observe that $\kov{}{\xi}$ is uniformly bounded on $F$, i.e.\ there is a
    constant $C$ such that $\norm{\kov{v}{\xi}}\leq C\norm{v}$.
    This is obviously true pointwise. Consider then
$$
    \norm{\kov{v}{\xi}}=\norm{\shape{\xi}{v}}+\norm{\oneillAd{\xi}{v}}.
$$
    The first term is bounded by $(\max\set{|\lambda_i|})\cdot\norm{v}$ and the $\lambda_i$
    are constant along~$F$.
    For the second term consider any $\eta\in\BottF$ and observe that
$$
    \ip{\oneillAd{\xi}{v}}{\eta} = \ip{v}{\oneill{\xi}{\eta}}
        \leq\norm{\oneill{\xi}{\eta}} \norm{v}
$$
    and $\norm{\oneill{\xi}{\eta}}$ is again constant along $F$.
\end{Rem}

Now suppose $\lambda_0=0$ then our conclusions from Equation~(\ref{Eq:1=norm(v_t)})
imply
$$
    \norm{\dot{\tilde{\xi}}_t(s)}=\norm{\kov{w_t}{\xi}}=
        \norm{
            \kov{\lt(\sum_{i\neq 0}(w_t)_i\rt)}{\xi}
            -\oneillAd{\xi}{(w_t)_0}}
        \leq \sum_{i\neq 0} \norm{\kov{(w_t)_i}{\xi}} + \norm{\oneillAd{\xi}{(w_t)_0}}
$$
and the last term tends to zero as $t\to\infty$ as we have seen. The remaining
terms tend to zero as well because of Remark~\ref{Rem:kov(xi) is bounded}.

But this also implies uniform convergence $\norm{\dot{\tilde{\xi}}_t(s)}\to 0$
since $\norm{\dot{\tilde{\xi}}_t(s)}$ is defined on the \emph{compact} interval $[0,R]$.
In particular we can choose $t$ large enough such that $\norm{\dot{\xi}_t(s)}<\frac{\eps}{R}$
uniformly in $s$.

\newpage
{\bf Part 2.}
We return to proving the assertion of the proposition:

Writing $\xi_t(s) = \int_{0}^{s} \dot{\xi}_t(\sigma) d\sigma +\xi_t(0)$ we get
$$
    \ip{\xi_t(s)}{\xi_t(0)} = 1 + \int_0^s\ip{\dot{\xi}_t(\sigma)}{\xi_t(0)} d\sigma
$$
and the modulus of the integrand is bounded by $\frac{\eps}{R}$.
Hence $\ip{\xi_t(s)}{\xi_t(0)}$ is contained in the interval $(1-\eps,1+\eps)$.

As a consequence we get the estimate
$$
    \norm{\xi_t(s)-\xi_t(0)}^2 =
    \norm{\xi_t(s)}^2 + \norm{\xi_t(0)}^2 - 2\ip{\xi_t(s)}{\xi_t(0)} < 2\eps
$$
since $\xi_t(s)$ is a unit vector for any $s$.
Moreover we can write
$$
    \ip{\xi_t(s)}{X_t(s)} = \int_{0}^{s} \lt(\Abl{}{\sigma} \ip{\xi_t(\sigma)}{X_t(\sigma)} \rt) d\sigma
$$
since $X_t(0)=0$ and
\begin{equation*}
    \lt|\Abl{}{s} \ip{\xi_t(s)}{X_t(s)}\rt| =
    \lt|\ip{\dot{\xi}_t(s)}{X_t(s)}\rt|< \lt\|\dot{\xi}_t(s)\rt\| \cdot R
    < \eps
\end{equation*}
since $\dot{X}_t(s)=\dot{\gamma}_t(s) \perp \xi_t(s)$.

So, $|\ip{\xi_t(s)}{X_t(s)}| < \eps \cdot R$.
Hence, we can finally show
\begin{equation*}
    |\ip{X_t(s)}{\xi_t(0)}| =
    |\ip{X_t(s)}{\xi_t(s) +(\xi_t(0) - \xi_t(s))}| <
    \eps R + \sqrt{2\eps}R.
\end{equation*}
Choosing $\eps$ sufficiently small proves our claim.
\end{proof}

Note that since $\xi_t$ is a Bott-parallel normal field along $F_t$
the assertion of Proposition~\ref{Prop:fibre is contained in thick hyperplane}
holds for every point of $F_t$.

Now, consider a sequence $t_n$ with $t_n \to \infty$ and denote by
$F_n$ the leaf $F_{t_n}$. By compactness of the base $\base$ we
may assume $\pi(F_n)$ to converge in $\base$. We will call the
fibre over this limit $\tilde{F}$. The compactness of $\base$ also
implies that the closed ball $\closure{\ball{\diam{\base}}{p}}$
meets all leaves. Choose now a sequence $p_n$ in $\closure{\ball{\diam{\base}}{p}}$
with $p_n\in F_n$. Remember that $\xi_n:=\xi_{t_n}(p_n)$ is a unit vector
for any $n$.
By passing on to subsequences we may assume that $p_n$ converges
towards some point $\tilde{p}\in\tilde{F}$ and $\xi_n(p_n) \ra
\tilde{\xi}(\tilde{p})$ for some unit vector $\tilde{\xi}$ with base
point~$\tilde{p}$. We do not care if $\tilde{\xi}$ is contained
in~$\normal_{\tilde{p}}\tilde{F}$.

\begin{Prop}\label{Prop:limit fibre is contained in hyperplane}
    The limit leaf $\tilde{F}$ is contained in the hyperplane
    $\plane{\tilde{p}}{\smash{\tilde{\xi}}(\tilde{p})}$.
\end{Prop}

\begin{figure}[h]
    \begin{center}
    \includegraphics[height=8cm]{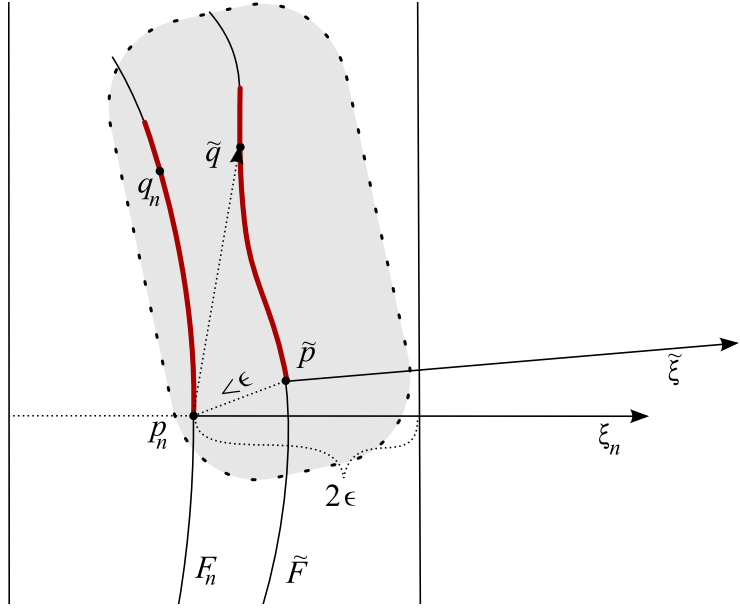}
    \caption[Illustrating the proof of Proposition~\ref{Prop:limit fibre is contained in hyperplane}]
        {The curve $\tilde{\gamma}$ from the proof of
        Proposition~\ref{Prop:limit fibre is contained in hyperplane}
        is contained in the blown up hyperplane~$\plane[2\eps]{p_n}{\xi_n}$.}
    \end{center}
\end{figure}

\begin{proof}
Let $\tilde{\gamma}\colon[0,R]\to \tilde{F}$ be a simple curve parameterized by
arclength starting at~$\tilde{p}$.
By Lemma~\ref{Lem:Proj is submersion} we may extend the velocity $\dot{\tilde{\gamma}}$
to a vertical vector field $V$ in some neighborhood of the image of $\tilde{\gamma}$.
Since the latter is compact we may choose this neighborhood to be some compact
tube around the image of $\tilde{\gamma}$.

Choose some point $\tilde{q}:=\tilde\gamma(t_0)$ lying on $\tilde\gamma$ and let
$\gamma_n$ be the integral curve of~$V$ starting at $p_n$. Using standard theory of
ordinary differential equations
we see that choosing $p_n$ sufficiently close to $\tilde{p}$ implies that
$q_n:=\gamma_n(t_0)$ is arbitrarily close to $\tilde{q}$ and also the length
of~$\gamma_n|_{[0,t_0]}$ is arbitrarily close to $t_0$, in particular it is less
than $2R$, say.

Let then $0<\eps<R$ and choose $n$ to be sufficiently large such that the following
inequalities hold:
$$
        \ball[F_n]{2R}{p_n} \subset \plane[\eps]{p_n}{\xi_n},
        \qquad
        \norm{p_n - \tilde{p}} < \eps,
        \qquad
        \norm{\tilde{\xi} - \xi_n} < \eps
$$
and increase $n$ even further if necessary such that the aforementioned properties
$$
    \norm{q_n-\tilde{q}} < \eps,
    \qquad
    \length{\gamma_n|_{[0,t_0]}} < 2R
$$
also hold.

This implies that $\tilde{q}$ is contained in the blown up
hyperplane~$\plane[2\eps]{p_n}{\xi_n}$:
$$
    |\ip{\tilde{q}-p_n}{\xi_n}| =
    |\ip{(\tilde{q}-q_n)+(q_n-p_n)}{\xi_n}| <
    2\eps,
$$
which in turn shows that $\tilde{q}$ lies in the hyperplane~$\plane{\tilde{p}}{\smash{\tilde\xi}}$
because
\begin{align*}
    \lt|\ip{\tilde{q}-\tilde{p}}{\tilde{\xi}}\rt|
        & =
            \lt|\ip{(\tilde{q}-p_n) + (p_n - \tilde{p})}
                    {\xi_n + (\tilde{\xi} - \xi_n)}\rt|
            \\
        & \leq
            \lt|
                \ip{\tilde{q}-p_n}{\xi_n}
            +
                \ip{\tilde{q}-p_n}{\tilde{\xi} - \xi_n}
            +
                \ip{p_n - \tilde{p}}{\xi_n}
            +
                \ip{p_n - \tilde{p}}{\tilde{\xi} - \xi_n}
            \rt| \\
        & <
            2\eps + (\eps + R)\eps + \eps + \eps^2
\end{align*}
for arbitrarily small $\eps>0$.

Since this holds for all $\tilde{q}\in\ball[\tilde{F}]{R}{\tilde{p}}$ and indeed
for any radius $R$ it follows that the whole leaf $\tilde{F}$ is contained
in the hyperplane~$\plane{\tilde{p}}{\smash{\tilde\xi}}$.
\end{proof}

Now $\tilde{F}$ being contained in a hyperplane means that it cannot be
$\diam{\base}$-close to every point in $\R^n$.
So $\base$ must be a point.
Thus we have shown:

\begin{Thm}\label{Thm:Submetries with compact base}
Let $\Fol$ be an equidistant foliation of $\R^n$ and suppose the space of leaves
$\base$ to be compact. Then $\base$ is a single point.
\end{Thm}

\begin{Rem*}
Observe that the leaves of $\Fol$ being connected, as we assumed in definition
of an equidistant foliation, is essential for this assertion to hold.
A simple counterexample to the theorem, dropping connectedness, is given by the
covering
$$
    f\colon\R\to\S{1},\qquad t\mapsto e^{it}.
$$
\end{Rem*}

We denote the affine leaf of $\Fol$ by $\affineleaf$ and for the rest of this thesis
we assume that $\affineleaf=\R^k\times\set{0}\subset\R^{k+n}$.

\begin{Rem*}
We end this chapter by observing that due to $\Fol$ being equidistant the affine
leaf~$\affineleaf$ of course is the most singular leaf of $\Fol$, i.e.\
the dimension of $\affineleaf$ is smallest.

Note also that we may assume $\affineleaf$ to be unique. For assume there is another
affine leaf, $\affineleaf'$ say, then $\affineleaf$ and $\affineleaf'$ are
parallel and the affine space $A$, spanned by them is foliated by leaves
of~$\Fol$ parallel to $\affineleaf$.
Observe that $A\cong\R^{k+n}$ is spanned by $\affineleaf$ and a line $l=p_0+\spann{v}$
meeting $\affineleaf$ and $\affineleaf'$ perpendicularly.

\parpic{\includegraphics[height=5cm]{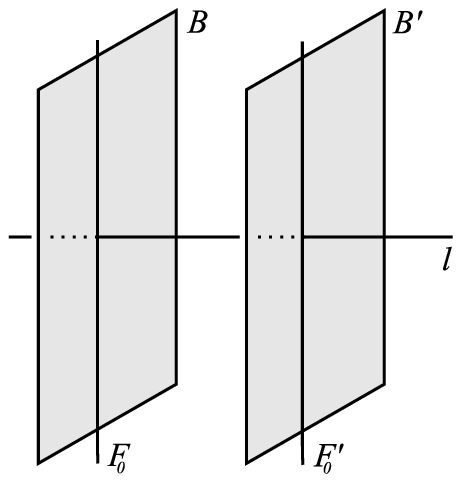}}
Consequently, for each point $p\in l$ the orthogonal complement $p+\spann{v}^\perp$
to the line~$l$ is invariant under $\Fol$, i.e.\ each leaf meeting that space is contained in it.
Since $\Fol$ is equidistant, the restrictions of $\Fol$ to any two such perpendicular complements $B$
and $B'$ of~$l$ differ only by a parallel translation along $l$.
Hence, $\Fol$ is the product of the induced foliation of $B$ and the discrete
foliation of the line~$l$.
\end{Rem*}

\chapter{The Induced Foliation in the Horizontal Layers}\label{Chap:Horizontal}

The existence of the affine leaf $\affinefibre$ leaves us in the special situation
that $\Fol$ together with the horizontal distribution along $\affinefibre$ induces
a further, refined foliation $\IndFol$ of $\R^{n+k}$ by intersecting the leaves
of $\Fol$ with the normal spaces of $\affineleaf$.

We first look at the homogeneous case and show that $\Fol$ is given by the orbits
of $G\times\R^k$ with $G$ compact and $\R^k$ acting on $\R^{k+n}$ by generalized
screw motions around the axis $\affineleaf$.
In particular we conclude that the induced foliation $\IndFol$ is equidistant.

In the remainder of this chapter we examine how much of this rather nice structure
of~$\IndFol$ can be recovered in the general case.

\sep

For any $p\in\affineleaf$ we denote the affine space $p+\Hor[p]$ by $\layer{p}$
and call it the \emph{horizontal layer} through $p$.

\begin{Def}\label{Def:Induced_Fibration}
For any $p\in F_0$ we will denote by $\IndFol[p]$ the foliation of $\layer{p}$
induced by $\Fol$, i.e.\
\begin{equation*}
    \IndFol[p]:=\Set{F\cap\layer{p}}{F\in\Fol}.
\end{equation*}
Consequently, the union $\IndFol$ over all $\IndFol[p]$, where $p$ is in $F_0$,
is a foliation of $\R^{n+k}$.
We denote the leaf $F\cap\layer{p}$ of $\IndFol[p]$ by $\IndLeaf{F}{p}$.
\end{Def}

Note that we have to make sure that the $\layer{p}$ intersect the leaves of~$\Fol$
as transversally as possible.

Let us first introduce some tools and notation used throughout this chapter.

\subsection*{Projections onto the affine leaf}

Let $\pos$ be the vector field on $\R^{k+n}$ indicating the position relative to
$\affinefibre$, i.e.\ for $x=(x_1,x_2)\in\R^{k+n}$ we set $\pos_x:=(0,-x_2)$.
Obviously, the shortest path from a point $x$ to $\affinefibre$ is given by
$t\mapsto x+t\pos_x$, hence the restriction of $\pos$ to the regular part of $\Fol$
is a Bott-parallel horizontal field.

\begin{Def}
Let $\proj\colon\R^{k+n}\ra\affinefibre$ be the orthogonal projection onto
the affine leaf $\affineleaf$.
We denote by $\projV$ and $\projH$ the restriction of $\diff{\proj}$ to the
vertical and horizontal distributions respectively.
\end{Def}

We can easily describe these projections using $\pos$ since $\proj{x}=x+\pos_x$.
Consequently, its derivative is given by
\begin{equation}
    \diff{\proj}X=X+\kov{X}{\pos},\label{Eq:Def Proj}
\end{equation}
for any vector $X$.

\begin{Lem}\label{Lem:IP of Projections}
Let $F$ be a regular leaf and $\xi,\eta$ two Bott-parallel vector fields on~$F$.
Then $\ip{\projH{\xi}}{\projH{\eta}}$ is constant along $F$.
\end{Lem}

\begin{proof}
This follows immediately from the proof of Lemma~\ref{Lem:Proj is submersion}
and Remark~\ref{Rem:Onormal is Bott-parallel}.
\end{proof}

By Lemma~\ref{Lem:Proj is submersion} the projection $\projV$ is surjective
at any regular point of the foliation $\Fol$, which
enables us to lift any tangent vector field on $\affinefibre$ to one on the
regular leaves of $\Fol$.

\begin{Def}\label{Def:Vertical Lift}
Let $v$ be a vector in $T_p\affinefibre$ and let $x$ be a point in a regular
leaf~$F\in\Fol$ such that $\proj{x}=p$.
We will call the unique vector $\Lift[x]{v} \in (\ker{\projV_x})^\perp \subset T_xF$
such that $\projV{\Lift[x]{v}}=v$ the \emph{vertical lift} of $v$ to $x$.
\end{Def}


\sep
After this digression we show that $\IndFol$ is indeed a smooth foliation.

\begin{Lem}\label{IndLeafs are submanifolds}
For any $p\in\affinefibre$ the leaves of $\IndFol[p]$ are complete smooth
submanifolds of $\layer{p}$.
\end{Lem}

\begin{proof}
Let us first look at a \emph{regular} leaf $F$ of $\Fol$.
By Lemma~\ref{Lem:Proj is submersion} every $p\in\affinefibre$ is a regular
value of the orthogonal projection $\proj|_{F}\colon F\to\affineleaf$ so the
preimage $\IndLeaf{F}{p}$ of $p$ is a smooth submanifold of $F$.

To deal with the \emph{singular} leaves of $\Fol$ note that we will show
in Proposition~\ref{Prop:Equidistance of each induced Foliation} that~$\IndFol[p]$
is equidistant. To be more precise, for any $p\in\affinefibre$
the restriction to $\layer{p}$ of $\Deriv{\pi}{p}$ is a submetry and its fibres
are the leaves of $\IndFol[p]$.

But the regular fibres being smooth submanifolds
already implies the same property for the singular fibres
(cf.\ \cite[Prop.\ 13.5]{Lyt}).
\end{proof}

\section{The Homogeneous Case}

In order to understand the role of the induced foliation $\IndFol$ better
let us first consider the homogeneous case.
So, in this section we assume the fibres of $\Fol$ to be the orbits of a connected
Lie group $G\subset \Isom{\R^{k+n}}$
acting effectively on $\R^{k+n}$
with $\affinefibre = \R^k \times \{0\}$
being its most singular orbit.

Obviously, for any $p\in \affinefibre$ the foliation $\IndFol[p]$ is then given
by the orbits of the slice representation of $G_p$. Hence, each $\IndFol[p]$ is
equidistant and since the isotropy groups along a fibre are conjugate
any two $\IndFol[p]$ and $\IndFol[q]$ are isometric to each other.
Moreover, we show:

\begin{Thm}\label{Thm:Induced foliation is equidistant in homogeneous case}
In the homogeneous case the induced foliation $\IndFol$ is equidistant.
\end{Thm}

To achieve this we must take a closer look on how $G$ acts on $\R^k$ and $\R^n$
respectively.
Since $G$ leaves the affine space $\affinefibre$ invariant it must be a subgroup
of
$$
    \Isom{\R^k}\times\SO{n}=
    \Set{\groupelement{A}{B}{a}}{A\in\SO{k}, B\in\SO{n}, a\in\R^k}
$$
where any $g \in G$ acts on $\point{x}{y}\in\R^{k}\times\R^n$ via
$$
    \groupelement{A}{B}{a}.\point{x}{y}=\point{Ax+a}{By}.
$$

\begin{Rem*}
Consider the two natural projections
\begin{align*}
    P_1\colon G \ra &\Isom{\R^k} \quad\text{and}\\
    P_2\colon G \ra &\SO{n},
\end{align*}
both of which are continuous group homomorphisms.
Note that $P_i(G)$ may not be a closed group.
We will use the following notation:

We denote the kernel of $P_i$ by $N_i$.
For any subgroup $H$ of $G$ we will use $\hat{H}$ and $\tilde{H}$ for its image
under the projections $P_1$ and $P_2$ respectively.
\end{Rem*}

We start by proving a reducibility result.

\begin{Lem}
Either $N_2$ is trivial or $\Fol$ splits off a Euclidean factor.
\end{Lem}

\begin{proof}
Assume that $N_2$ is not trivial.
Observe that since
$$
    N_2=\Set{\groupelement{A}{E}{a}}{(A,a)\in P_1(G)}
$$
the projection $P_1|_{N_2}\colon N_2 \to \hat{N_2}$ is an isomorphism.

Consider the action of $\hat{N_2}$ on $\R^k$.
By Theorem~\ref{Thm:fibre over soul is totally geodesic} one of the orbits of
this action is an affine space $\mathcal{A}$, which we may assume without loss
of generality to pass through the origin.

Remember that $\hat{G}$ acts transitively on $\R^{k}$.
Let $x$ be an arbitrary point in $\R^k$ and let $g\in \hat{G}$ be such that
$g.0=x$.
Since $N_2$ is a normal subgroup of $G$ we get $\hat{N_2}\lhd\hat{G}$.
Thus the $\hat{N_2}$-orbit passing through $x$, given by
$$
    \hat{N_2}.x=\hat{N_2}.g.0=g.\hat{N_2}.0=g.\mathcal{A},
$$
is also an affine space, which we denote by $\mathcal{A}_x$.
By the equidistance of the orbits of $\hat{N_2}$ all these affine spaces
$\mathcal{A}_x$ must be parallel.

Remember that $N_2$ acts trivially on $\R^n$. So for any $\point{x}{y}\in\R^{k+n}$
the $N_2$-orbit through $\point{x}{y}$ is just the affine space
$\point{x}{y}+\mathcal{A}\times\{0\}$. Hence, $\Fol$ splits off the Euclidean
factor $\mathcal{A}\times\{0\}$.

Suppose that $\mathcal{A}=\{0\}$. Then $N_2$ acts trivially on $\R^k$ since
$\hat{N_2}$ does. So $N_2$ is trivial as we assumed the action of $G$ to be
effective.
\end{proof}

\begin{Rem}\label{Rem:P_2 is isomorphism}
In the following we will concentrate on the
case of $P_2$ being an isomorphism by passing on to the reduced foliation if
necessary.
\end{Rem}

\begin{Lem}\label{Lem:Isotropy acts trivially on F_0}
The isotropy group $G_0$ is equal to $N_1$ and the projection $\hat{G}$ of $G$
is abelian.
\end{Lem}

\begin{proof}
According to Remark \ref{Rem:P_2 is isomorphism} we have
$G \cong \tilde{G}=P_2(G)$ and
since $\tilde{G}$ is contained in the compact Lie group $\SO{n}$ we get the
following decomposition for the Lie algebra $\Lie{g}$ of $G$:
\begin{equation}\label{Eq:Lie algebra decomposition}
    \Lie{g} = \Lie{z} \oplus \Lie{g}',
\end{equation}
where $\Lie{z}$ is its center and $\Lie{g}'$ is semisimple.
\begin{Rem*}
Note that a priori we only get the decomposition
$$
    \Lie{g}= \rad\Lie{g}\oplus\Lie{h},
$$
where $\rad\Lie{g}$ is the solvable radical of $\Lie{g}$ and $\Lie{h}$ is
semisimple (cf.\ \cite[Thm.~3.8.1]{Var}).
Let then $R$ be the connected Lie group corresponding to $\rad\Lie{g}$ and consider
its image under $P_2$.

Obviously $P_2(R)$ is solvable, i.e.\ there is a
chain $\{1\}=:G_0 \lhd \ldots \lhd G_n:=P_2(R)$ of normal subgroups
such that subsequent quotients $G_{i+1}/G_i$ are abelian.
But clearly by continuity of the group operations the property of being a normal
subgroup is preserved if we take closures and the subsequent
quotients of the $\closure{G_i}$ remain abelian as well.

Hence, $\closure{P_2(R)}$ is solvable. But as a compact Lie group this
can only be the case if it is abelian.
Now $\closure{P_2(G)}$ is contained in the normalizer of $\closure{P_2(R)}$ and
acts on $\closure{P_2(R)}$ by conjugation.
But since $\closure{P_2(R)}$ is a torus its automorphism group is discrete.
Also note that $\closure{P_2(G)}$ is connected, so $\closure{P_2(G)}$ is in fact
contained in the centralizer of $\closure{P_2(R)}$.

In particular $P_2(R)$ lies in the center of $P_2(G)$ and hence $R$ lies in the
center of $G$ because $P_2$ is a group isomorphism.
The reverse inclusion follows from the definition of $R$.
\end{Rem*}

Let $G'$ be the unique connected Lie subgroup of $G$ corresponding to the
Lie algebra~$\Lie{g}'$.
Note, that the decomposition (\ref{Eq:Lie algebra decomposition})
implies $G'$ to be a normal subgroup of $G$.

Observe that $G'$ is a semisimple subgroup of
$\Isom{\R^{k+n}}=\R^{k+n}\ltimes\SO{k+n}$ and
consider the natural projection~$P\colon\Isom{\R^{k+n}}\to\SO{k+n}$.
Note that $P$ is a Lie group homomorphism.

Now the Lie algebra $\Lie{g}'$ decomposes into a sum of simple Lie algebras $\Lie{g}'_i$.
Each~$\Lie{g}'_i$ is either mapped to zero or to an isomorphic image of $\Lie{g}'_i$.
But $\diff{P}(\Lie{g}'_i)=0$ means that the corresponding connected Lie group $G'_i$
consists only of translations of $\R^{k+n}$ and hence is solvable, which contradicts
$G'$ being semisimple.

So $\Lie{g}'$ is isomorphic to a subalgebra of $\Lie{so}(n)$ and
hence $G'$ is compact.
In particular it has a fixed point $\point{x}{y}\in\R^{k+n}$.
Consequently, $\hat{G'}$ leaves $x\in\R^k$ invariant.

\begin{Rem}\label{Rem:Normal subgroup acts trivially}
Let $G$ be any group acting transitively on some space $X$ and suppose $H$
to be a normal subgroup of $G$ which is contained in the isotropy group $G_x$ of
some point $x\in X$. Then $H$ acts trivially on $X$.

To see this observe that the $H$-orbit passing through some $y\in X$ is given by
$$
    H.y=H.g.x=g.H.x=g.x=y,
$$
for some $g\in G$ since $G$ acts transitively.
\end{Rem}

This implies that $\hat{G'}$ acts trivially on $\R^k$, i.e.\ $G'$ is contained
in $N_1$.
So $N_1$ has Lie algebra $\Lie{t}\oplus\Lie{g}'$ for some subalgebra
$\Lie{t}$ of $\Lie{z}$.
Since $\hat{G}=P_1(G)\cong G/N_1$ and $G/N_1$ has Lie algebra
$(\Lie{z}\oplus\Lie{g}')/(\Lie{t}\oplus\Lie{g}'$) it follows that
$\hat{G}$ is abelian.

Thus $\hat{G}_0$ is a normal subgroup of $\hat{G}$ and Remark~\ref{Rem:Normal subgroup acts trivially}
implies $G_0 \subset N_1$, which finishes the proof.
\end{proof}

In particular this means that the isotropy group $G_p$ does not depend on the
choice of $p\in\affinefibre$, hence, the induced foliations $\IndFol[p]$ are
equal up to parallel transport along $\affinefibre$, which proves
Theorem~\ref{Thm:Induced foliation is equidistant in homogeneous case}.

Also, this provides a convenient way to describe the action of $G$ on $\R^{k+n}$:

\begin{Prop}\label{Prop:Phi goes to Centralizer}
If $\Fol$ is irreducible there exists a Lie group homomorphism
$\Phi\colon\R^k\to\Centr{\tilde{G}_0}$
into the centralizer of $\tilde{G}_0$ relative to $\SO{n}$ such that the orbits
of~$G$ are of the form
\begin{equation}\label{Eq:Orbits of G}
    G.(x,y)=\Set{(x+v,\Phi(v).\tilde{G}_0.y)}{v\in\R^k}.
\end{equation}
\end{Prop}

Remember that $G_0$ acts trivially on $\R^k$, thus $\tilde{G}_0$ is just the trivial
embedding of $G_0$ into $\Isom{\R^n}$ and hence a Lie group.

Let us first show that $\hat{G}$ and thus $G$ act on $\R^k$ by translations.
Since the action of $\hat{G}$ on $\R^k$ has trivial isotropy and $\hat{G}$ is
abelian it suffices to prove:

\begin{Lem}
Let $H$ be an abelian group acting simply transitively on $\R^m$ by Isometries.
Then $H$ acts by translations.
\end{Lem}

\begin{proof}
Remember that $H$ may be viewed as a subgroup of
$$
    \Isom{\R^m}=\Set{(A,a)}{A\in\O{m},a\in\R^m}
$$
with the group multiplication given by $(A,a)\circ(B,b)=(AB,a+Ab)$.

Since $H$ acts simply transitively any $h=(A,a)\in H$ is uniquely determined by
its translational part, i.e.\ there is a group homomorphism
$\vphi\colon\R^m\to\O{m}$ such that any $h\in H$ is of the form $h=(\vphi(a),a)$
for some $a\in\R^m$.

Define $V_0:=\ker{\vphi}$ and $V_1:=V_0^\perp$. Observe that since $H$ is abelian
the dimension of its image under $\vphi$ is at most the rank of $O(m)$ which is
strictly less than $m$ so $V_0$ has positive dimension.

Assume $V_1$ to be non-trivial. Let $v\in V_1$ with $v\neq0$ and $w\in V_0$.
The group $H$ being abelian then implies
$$
    (\vphi(v)\vphi(w),v+\vphi(v)w) = (\vphi(w)\vphi(v),w+\vphi(w)v).
$$
In particular, using $w\in\ker\vphi$, this means $v+\vphi(v)w=w+v$ for all
$w\in V_0$.
So, $\vphi(H)$ acts trivially on $V_0$ and thus the image of $\vphi$ is contained
in $\O{V_1}$.

This yields the group homomorphism $\vphi|_{V_1}\colon V_1\to\O{V_1}$ which by
the above rank argument must have a non-trivial kernel. But this contradicts
$\vphi|_{V_1}$ being injective so $V_1$ must be trivial.
\end{proof}

As a consequence of this and because $G_0$ is a normal subgroup of $G$ any element
of $G$ is uniquely determined by a translation on $\R^k$ up to multiplication
with $G_0$.

This yields a homomorphism $\phi\colon\R^k\to\SO{n}$ such that the orbits of $G$
are of the form described in~(\ref{Eq:Orbits of G}).

\begin{Rem*}
Note that the image of $\phi$ has to be contained in $\Norm{\tilde{G}_0}$ since
$G_0$ is a normal subgroup of $G$. But it need not, in general, be contained
in $\Centr{\tilde{G}_0}$.

In fact, the map $\Phi$ we construct in the following may lead to a different group
action, which, however, is orbit equivalent to that of $G$.
\end{Rem*}

\begin{proof}[Proof of Proposition~\ref{Prop:Phi goes to Centralizer}]

Let us first take a look at $\phi$ at the level of Lie algebras:
$\diff[0]{\phi}\colon\R^k\to\Lie{n}$ maps the abelian
Lie algebra $\R^k$ into the normalizer $\Lie{n}:= \Lie{norm}(\tilde{\Lie{g}}_0)$
(relative to $\Lie{so}(n)$) of the Lie algebra $\tilde{\Lie{g}}_0$ of $\tilde{G}_0$.

\parpic[l]{
\setlength{\unitlength}{1.2cm}
\begin{picture}(3,3.2)(-1,-1.2)
    \put(-1,0){\line(1,0){2.5}}
    \put(0,-1){\line(0,1){2.5}}
    \put(-.5,-1){\line(1,2){1.2}}
    \put(-2.5,.8){$\Lie{n}:=\Lie{norm}(\tilde{\Lie{g}}_0)$}
    \put(1.6,0){$\tilde{\Lie{g}}_0$}
    \put(-.2,1.7){$(\tilde{\Lie{g}}_0)^\perp$}
    \put(.6,1){$\diff[0]{\phi}\R^k$}
\end{picture}
}
Consider the natural projection
$P\colon \Norm{\tilde{G}_0} \to \Norm{\tilde{G}_0}/\tilde{G}_0$
and its derivative
$\diff[e]{P}\colon \Lie{n} \to \Lie{n} / \tilde{\Lie{g}}_0$.
We may assume that $\diff[e]{P}\circ\diff[0]{\phi}$ is injective for otherwise
$\Fol$ would split off the kernel of $\diff[e]{P}\circ\diff[0]{\phi}$ as a
Euclidean factor (cf.\ also Section~\ref{Sec:Noncompact Case}).

Now $\Lie{n} / \tilde{\Lie{g}}_0$ is canonically isomorphic
to $(\tilde{\Lie{g}}_0)^\perp$, with the orthogonal complement taken with respect to
the Killing form in~$\Lie{n}$.
And since $\tilde{\Lie{g}}_0$ is an ideal of $\Lie{n}$ so is~$(\tilde{\Lie{g}}_0)^\perp$
(cf.~\cite[Chap.~6]{Helg}).
Hence, $(\tilde{\Lie{g}}_0)^\perp$ is contained in the centralizer of
$\tilde{\Lie{g}}_0$, in fact
$\Lie{centr}(\tilde{\Lie{g}}_0)=(\tilde{\Lie{g}}_0)^\perp\oplus\Lie{z}(\tilde{\Lie{g}}_0)$,
where $\Lie{z}(\tilde{\Lie{g}}_0)$ is the center of $\tilde{\Lie{g}}_0$.

Thus there is a Lie algebra homomorphism
$\tilde{\Phi}\colon\R^k\to\Lie{centr}(\tilde{\Lie{g}}_0)$ such that the following
diagram commutes:
$$
    \xymatrix{
        &\R^k \ar[r]^{\diff[0]{\phi}} \ar@{.>}[d]_{\tilde{\Phi}}
            & \Lie{n} \ar[d]^{\diff[e]{P}} \\
        \Lie{centr}(\tilde{\Lie{g}}_0) \ar@{}[r]|-{\supset}&(\tilde{\Lie{g}}_0)^\perp
            & \Lie{n}/\tilde{\Lie{g}}_0 \ar[l]_\cong
}
$$

And since $\R^k$ is simply connected we can lift $\tilde{\Phi}$ to a Lie group
homomorphism~$\Phi$ from $\R^k$ to the connected component of $\Centr{\tilde{G}_0}$.

By this construction we get for any $v\in\R^k$ that $\Phi(v)$ is equal to $\phi(v)$
up to multiplication by some element in $\tilde{G}_0$, which implies that the orbits
of $G$ may indeed be written in the form~(\ref{Eq:Orbits of G}).
\end{proof}

\section[The Induced Foliation in each Horizontal Layer]%
{The Induced Foliation in each Horizontal Layer is Equidistant}%
\label{Sec:IndFol in each Layer}

We have seen that in the homogeneous case each of the induced foliations~$\IndFol[p]$
is equidistant.
This holds in general for equidistant foliations of $\R^{k+n}$:

\begin{Prop}\label{Prop:Equidistance of each induced Foliation}
For any point $p$ in the affine leaf $\affineleaf$ the induced foliation $\IndFol[p]$
of the horizontal Layer $\layer{p}$ is equidistant.
\end{Prop}

\begin{proof}
Let $\bar{\pi}$ be the restriction $\diff[p]{\pi}\colon\Hor[p]\to T_{\pi(p)}\base$ of the
differential $\diff[p]{\pi}$ to the horizontal space at $p\in\affineleaf$.
By Proposition~\ref{Prop:vertical and horizonal vectors} we know that $\bar{\pi}$ is
a submetry, hence its fibres are equidistant.

Now $\Hor[p]$ is isometric to $L_p$ via the normal exponential map at $p$.
And from Section~\ref{Sec:Differentials} we know that $\bar{\pi}$ maps a
vector~$h$ to~$\bar{h}$ if and only if $\pi$ maps the geodesic with starting direction
$h$ to that with starting direction $\bar{h}$.
So the exponential map commutes for horizontal directions with the differential
of the submetry.
Hence, the fibration of $\Hor[p]$ by $\bar{\pi}$ is isometric to the induced
foliation~$\IndFol[p]$.
\end{proof}

\section[The Induced Foliations in distinct Horizontal Layers]
    {The Induced Foliations in distinct Horizontal Layers are Isometric}

We have seen that each individual induced foliation~$\IndFol[p]$
is equidistant. In this section we will examine how these foliations change if we
move along the affine leaf~$\affineleaf$.

\begin{Rem}\label{Rem:Homogeneity assumption}
Unless we say otherwise we will from now on assume each of the
induced foliations $\IndFol[p]$ to be homogeneous.
In particular, for each $p\in\affineleaf$
the leaves of $\IndFol[p]$ are the orbits of some connected closed subgroup $G$
of $\SO{n}$, where $G$ acts on $\R^n$ via the restriction of the standard
representation of $\SO{n}$.

\end{Rem}

Note that for any fixed $p$ this group $G$ need not be unique.
Therefore we will pass on to the maximal group that has the same orbits.

\begin{Def}
Let $G\subset\SO{n}$ be a closed connected Lie group acting on $\R^n$ via the
restriction of the standard representation of $\SO{n}$. Then
\begin{equation*}
    \maxgroup{G}:=\Set{g\in \SO{n}}{g(Gx)=G.(gx),\forall x\in\R^n}_0
\end{equation*}
is the maximal connected Lie subgroup of $\SO{n}$ having the same orbits as $G$.
\end{Def}

By definition $\maxgroup{G}$ leaves the orbits of $G$ invariant and acts
transitively on them, since $G\subset\maxgroup{G}$. A straightforward calculation
shows that $\maxgroup{G}$ is indeed a Lie group.

We will denote the maximal connected subgroup of
$\SO{n}$ whose orbits are the leaves of $\IndFol[p]$ by $\layergroup{p}$.
This notation already suggests that $\layergroup{p}$ is just the isotropy group
of~$p$ in case $\Fol$ is homogeneous.

\begin{Prop}\label{Prop:Isometry of_fibrations}
For any $p,q\in F_0$ the induced foliations $\IndFol[p]$ and $\IndFol[q]$ are
isometric to each other.
\end{Prop}

We first show that $\IndFol[p]$ and $\IndFol[q]$ are diffeomorphic to each other.
Next we prove that $\IndFol[p]\to\IndFol[q]$ in a suitable way as $p\to q$.
We conclude then that $\layergroup{p}$ and $\layergroup{q}$ have to be in the same
conjugation class and $\layergroup{p}\to\layergroup{q}$ as $p$ tends to $q$.

\begin{Lem}\label{Lem:F_p and F_q are diffeomorphic}
Let $p$ and $q$ be any two points in $\affinefibre$ then $\IndFol[p]$
and $\IndFol[q]$ are diffeomorphic to each other.
\end{Lem}

\begin{proof}
Consider a parallel vector field $V$ on $\affinefibre$ such that $p+V_p=q$
and its vertical lift $\Lift{V}$ to $\R^{k+n}$ as introduced in Definition~\ref{Def:Vertical Lift}.
By construction the flow~$\phi_t$ of $\Lift{V}$ maps horizontal layers
onto each other preserving the leaves of $\Fol$, in particular $p$ is mapped to
$q$ for $t=1$. This yields the desired diffeomorphism.
\end{proof}

As we have seen in the previous section the induced foliations $\IndFol[p]$ are
equidistant, so it makes sense to contemplate the restriction of this foliation
to the unit sphere in $\layer{p}$ based at $p$ even if we drop the homogeneity
assumption made in Remark~\ref{Rem:Homogeneity assumption}.
We will denote this restriction by $\IndFolS[p]$.

\begin{Lem}\label{Lem:Horizontal leaf convergence is uniform}
Let $(p_j)$ be a sequence in $\affinefibre$ with $p_j\ra p\in\affinefibre$.
Then $\IndFolS[p_j]$ converges uniformly in Hausdorff distance towards
$\IndFolS[p]$.
\end{Lem}

\begin{Rem*}
By $\IndFolS[p_j]\HausConv\IndFolS[p]$ we mean the following. Let us identify
all horizontal layers~$\layer{p}$ by parallel translation along $\affinefibre$.
Thus we understand the $\IndFol[\star]$ to be foliations on the same euclidean
space $\R^n$.
Then $\IndFolS[p_j]$ tends to $\IndFolS[p]$ in the Hausdorff distance if and only if
for any leaf $F\in\IndFolS[p]$ there is a sequence of leaves $F_j\in\IndFolS[p_j]$
such that $F_j\HausConv F$ and this convergence is uniform in the leaves $F$.
\end{Rem*}

\begin{proof}
First we show that $\IndFolS[p_j]$ converges towards $\IndFolS[p]$ leafwise,
i.e.\ for any leaf $F$ in $\FolS$ the leaves $\IndLeaf{F}{p_j}$ tend towards
$\IndLeaf{F}{p}$ in Hausdorff distance.

For any $j\in\N$ consider the vertical lift $\gamma_{j,x}$ of the line segment
$pp_j$ through $x\in\IndLeaf{F}{p}$, which gives us the estimate
\begin{equation*}
    \dHaus{\IndLeaf{F}{p}}{\IndLeaf{F}{p_j}}
        \leq \max_{x\in\IndLeaf{F}{p}} \length{\gamma_{j,x}}.
\end{equation*}
Using the lifting map $\lift\colon \R^{k+n}\times T\affineleaf \to T\R^{k+n}$
we can express the length of $\gamma_{j,x}$ via
\begin{equation*}
    \length{\gamma_{j,x}}=\int_0^1 \norm{\Lift[\gamma_{j,x}(t)]{p_j-p}} dt.
\end{equation*}
By construction $\lift$ is linear in its second argument.
So $\Lift[x]{p_j-p}$ tends to zero for fixed $x$ as $j$ tends to
infinity.
Since $\lift$ is continuous this convergence is uniform
in $K\times\S{n-1}$, where $K$ is any compact neighbourhood of $p$
in $\affinefibre$.
Hence, $\IndLeaf{F}{p_j}\HausConv\IndLeaf{F}{p}$ and this convergence is uniform
in the choice of $F\in\FolS$.
\end{proof}

\begin{Rem*}
Whereas the homogeneity assumption for the foliations $\IndFol[p]$ was not necessary
for the previous two lemmas it is essential for the following arguments.
\end{Rem*}

Choose any biinvariant metric on $G$ and
consider the space $\subgrp{G}$ of closed subgroups of $G$
equipped with the Hausdorff metric. Compactness of $G$ implies
that $\subgrp{G}$ is compact as well. To see this consider the following:

\begin{Rem}\label{Rem:Hausdorff}
For any compact metric space $X$ the set $\hausdorffspace{X}$
of all closed subsets of $X$ equipped with the
Hausdorff distance is compact (cf.\ \cite[Thm.~7.3.8, p.~253]{BBI}).

Suppose $A_j\ra A$ in $\hausdorffspace{X}$ then $A$ is the set of all limits of
all sequences $(a_j)\in X$ such that $a_j\in A_j$ (cf.\ \cite[p.~253]{BBI}).
\end{Rem}

Now, suppose $(H_j)\in\subgrp{G}$ converges to $H\in\hausdorffspace{G}$.
The previous remark then clearly implies that the 1-element of $G$ is in $H$.
And since all of the $H_j$ are groups so is $H$ by continuity of the group operations.
Thus, $\subgrp{G}$ is a closed subset of $\hausdorffspace{G}$ and so is compact
as well.

\begin{Lem}\label{Lem:Groups are close}
Let $G$ and $G_j,$ with $j\in\N$, be closed Lie subgroups of $\SO{n}$ and let
$\FolS_G$ and $\FolS_{G_j}$ be the foliations of $\S{n-1}$ by the orbits of
$G$ and $G_j$ respectively.
Assume the group actions to be the restrictions of the standard representation
of $\SO{n}$ and assume further that $G=\maxgroup{G}$ and $G_j=\maxgroup{G_j}$
for all $j$.
Then, the uniform convergence of $\FolS_{G_j}$ towards $\FolS_G$ in Hausdorff
distance implies $G_j\HausConv G$.
\end{Lem}

\begin{proof}
Since $\subgrp{\SO{n}}$ is compact we may assume for the moment
without loss of generality that $G_j$ converges
to some Lie subgroup $H\subset\SO{n}$.

The main part of this proof is to show that $H$ is contained in $G$, which is to
say that $H$ leaves the orbits of $G$ invariant.
Assume the contrary, i.e.\ there is an $h\in H$ and a point $x\in\S{n-1}$ such
that $hx\notin Gx$. By Remark~\ref{Rem:Hausdorff} we get a sequence
$g_j\in G_j$ tending to $h$. The uniform convergence of $\FolS_{G_j}$ towards
$\FolS_G$ then implies that the distance between $g_jx$ and $Gx$ tends to zero,
which contradicts our assumption.

Note that by an analogous argument $H$ acts transitively on the leaves of $\FolS_G$,
which implies $\maxgroup{H}=G$.
But since all $G_j$ are maximal so is their limit, hence $H=G$.

To finish the proof we drop the convergence assumption on $(G_j)$.
Since any subsequence of $(G_j)$ contains itself a convergent subsequence and the
limit of these is always $G$ we get that $G_j\HausConv G$.
\end{proof}

Hence the map
$$
    \affineleaf\to\subgrp{\SO{n}}, \qquad p\mapsto\layergroup{p}
$$
is continuous and we finish the proof of Proposition~\ref{Prop:Isometry of_fibrations}
by showing:

\begin{Lem}\label{Lem:Conjugacy classes are connected components}
The conjugacy classes in $G$ are the path connected components of~$\subgrp{G}$.
\end{Lem}

\begin{proof}
Let us denote the conjugacy classes in $G$ by $(K_\alpha)_{\alpha\in A}$.
Obviously $\subgrp{G}$ is the disjoint union of these $K_\alpha$.

First observe that $G$ contains only countably many conjugacy classes:
There are only finitely many semisimple Lie subgroups up to conjugation.
And the tori are characterized by the slope of their embedding in a maximal torus.
Since we only consider closed subgroups of $G$ this slope has to be rational.
Hence, up to conjugation, there are only countably many closed abelian subgroups of $G$.
Since any Lie subalgebra of $\Lie{g}$ is the sum of an abelian and a semisimple Lie algebra
this proves our first claim.

Now, consider the action of $G$ on $\subgrp{G}$ by conjugation:
$g.H = gHg^{-1}$, for $g\in G$ and $H\in\subgrp{G}$.
Obviously, this action is continuous. Moreover $G$ acts by isometries
since the Hausdorff metric on $\subgrp{G}$ is based on a biinvariant metric on~$G$.
Hence, each $G$-orbit in $\subgrp{G}$ is compact and path connected and the
orbits form an equidistant decomposition of $\subgrp{G}$.

On the other hand, suppose $K_1$ and $K_2$ to be in the same path connected component
of $\subgrp{G}$ and $\gamma$ a path connecting them. Clearly $t\mapsto \dist{\gamma(t)}{K_1}$
is continuous but takes only values in a countable set and, hence, is constant.
So, the two conjugacy classes are identical, proving the lemma's assertion.
\end{proof}

\begin{Rem}\label{Rem:Parametrization of Fol}
As a consequence we may describe the leaves of $\Fol$ in analogy to
Equation~(\ref{Eq:Orbits of G}) from Proposition~\ref{Prop:Phi goes to Centralizer}.
That is to say, we can find for any $x\in\R^k$ a smooth
map $\Psi_x\colon\R^k\to\SO{n}$ such that the leaf $F$
passing through $(x,y)\in\R^{k+n}$ is given by
$$
    F=\Set{(x+v,\Psi_x(v).\layergroup{x}.y)}{v\in\R^k}.
$$
We stress again that this map depends on $x\in\R^k$ but not on $y$.

We call $\Psi_x$ a \emph{screw motion map} although $\Psi_x$ need not, a priori,
be a group homomorphism.
However, we can of course choose $\Psi_x$ such that $\Psi_x(0)=\id$ holds.
\end{Rem}

\section{Equidistance of the Leaves in distinct Horizontal Layers}

We have seen that in the homogeneous case the induced foliations $\IndFol[p]$ are
the same for every point $p$ up to parallel translation along $\affineleaf$.
We show that in general  this property is characterized by the behavior of the
projections of Bott-parallel fields.

We first introduce some more notation.

\begin{Def}
We denote by $\vproj[p]\colon\R^{k+n}\to\layer{p}$ the orthogonal projection
onto the horizontal layer $\layer{p}$ and by $\vprojH[p]\colon\Hor\to\R^n$
the restriction of its differential to the horizontal distribution $\Hor$.
\end{Def}

We sometimes omit the index $p$ and write just $\vprojH$ if it is not important
which specific horizontal layer we are considering.

\begin{Def}
We call $\Fol$ \emph{horizontally full} if at every regular point $x$ of $\Fol$
the map $\projH\colon\Hor[x]\to T_{\proj{x}}\affineleaf$ is surjective.
\end{Def}

Let us now examine how the projections of Bott-parallel normal fields behave.
Our first result states that $\Fol$ and $\IndFol[p]$ are ``compatible'' via the
projection $\vproj[p]$.

\begin{Lem}\label{Lem:Projection to horizontal Layer of Boot-parallel fields}
Let $F$ be a regular leaf of $\Fol$ and $\xi$ a Bott-parallel normal field along~$F$.
For any $p\in\affineleaf$ consider the restriction of $\xi$ to the induced leaf
$\IndLeaf{F}{p}$.
Then the projection $\vprojH[p]{\xi}$ of $\xi$ to the horizontal layer $\layer{p}$
is Bott-parallel (with respect to~$\IndFol[p]$) along $\IndLeaf{F}{p}$.
\end{Lem}

\begin{proof}
We refer the reader to figure~\ref{Fig:Projection of Bott-parallel field} for an
illustration of the construction used in this proof.

\begin{figure}[h]
    \begin{center}
    \includegraphics[width=13.5cm]{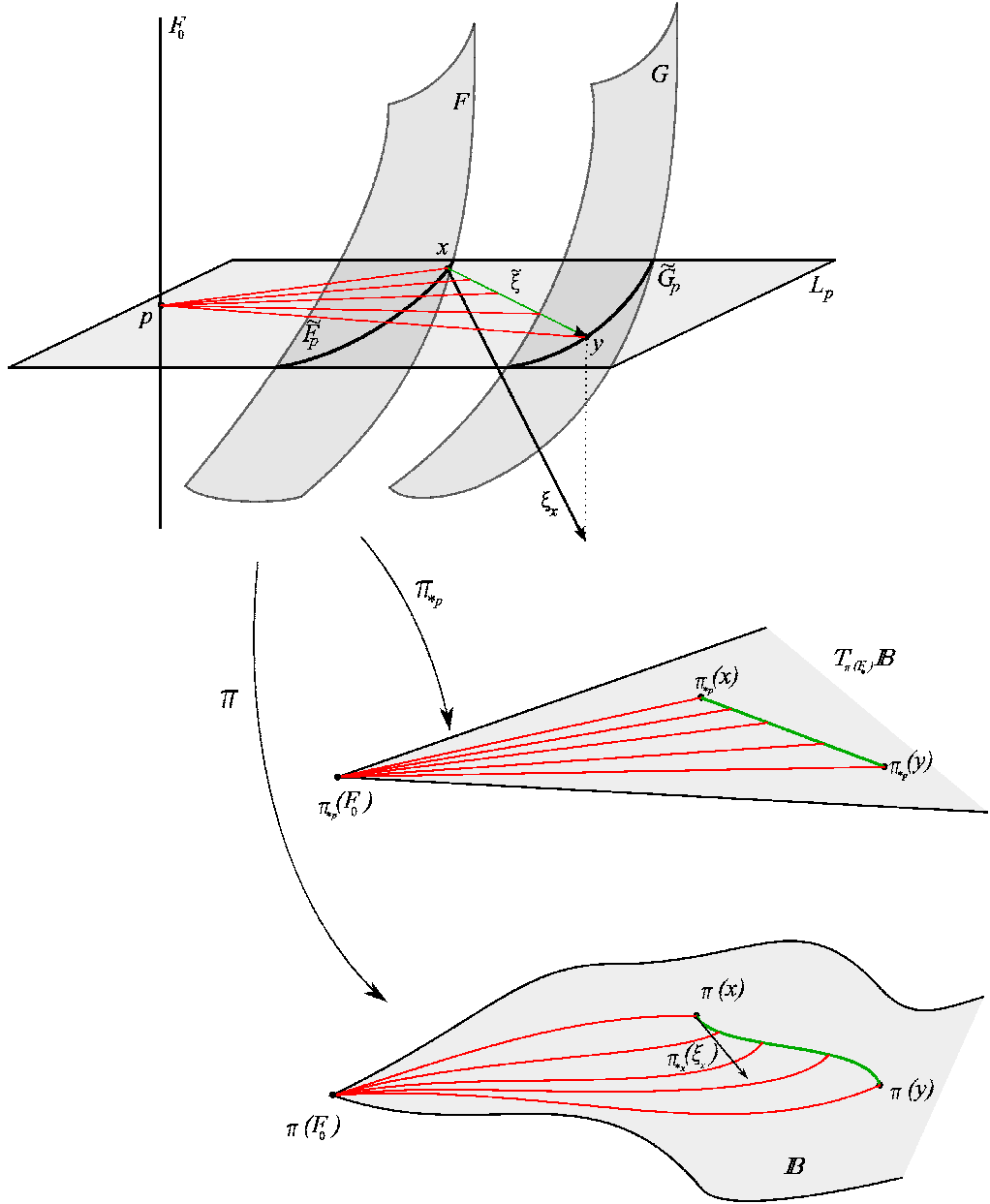}
    \caption[Projection of Bott-parallel fields are Bott-parallel.]
        {The projection of a Bott-parallel normal field to a horizontal
        layer is Bott-parallel with respect to the induced foliation in that layer.}
    \label{Fig:Projection of Bott-parallel field}
    \end{center}
\end{figure}

Choose an arbitrary point~$x\in\IndLeaf{F}{p}$.
Denote by $\tilde{\xi}_x$ the projection $\vprojH[p]{\xi_x}$ of $\xi_x$ and
by $\tilde{\xi}$ its Bott-parallel continuation (with respect to $\IndFol[p]$).
The leaf $\IndLeaf{F}{p}+\tilde{\xi}$ in $\IndFol[p]$ will be called $\IndLeaf{G}{p}$.

Consider the curve $\gamma\colon[0,1]\to\layer{p}$ with $\gamma(t)=x+t\tilde\xi_x$
and denote its endpoint by $y$. In the following we will examine the image of
$\gamma$ under both $\pi$ and $\diff[p]{\pi}$.

Since the image of $\gamma$ is a horizontal shortest path in $\IndFol[p]$ it
is mapped by $\diff[p]{\pi}$ to a shortest path in the tangent cone
$T_{\pi(\affinefibre)}\base$.

Note that in general this might only yield a quasi-geodesic in
$T_{\pi(\affinefibre)}\base$ but we get a proper geodesic if $\gamma$ is
sufficiently short. Since our argument works for arbitrary small $|\xi_x|>0$
this poses no problem.

On the other hand the variation given by the curves
$\alpha_t\colon s\mapsto p+s(\gamma(t)-p)$ are horizontal shortest paths with
respect to $\Fol$ so $\pi$ maps them to shortest paths
in $\base$. Again we may have to assume the image of $\gamma$ to be close to $p$,
which we can do without loss of generality since the assertion we want to prove
is left invariant by dilating radially from $\affinefibre$.

So the curve $\pi\circ\gamma$ is given by the endpoints of the shortest paths
$\pi\circ\alpha_t$:
$$
    \pi(\gamma(t)) = \pi(\alpha_t(1))
$$
and the starting direction of $\pi\circ\gamma$ is just $\diff[x]{\pi}(\tilde{\xi}_x)$.
By taking the horizontal part of $\tilde{\xi}_x$ with respect to $\Fol$, i.e.\ $\xi_x$,
and using Proposition~\ref{Prop:vertical and horizonal vectors}
we see that in fact the starting direction of $\pi\circ\gamma$ is given by
$\diff[x]{\pi}({\xi}_x)$.

As an aside we observe that we need not bother to check whether $\pi\circ\gamma$
has a well defined starting direction, since for small $|\xi_x|$ the image of
$\gamma$ lies within the regular part of $\Fol$ and here $\pi$ is given by a
Riemannian submersion.

Now if we choose another starting point on $\IndLeaf{F}{p}$, $x'$ say, and
construct a curve~$\gamma'$ in analogy to $\gamma$ using $\tilde{\xi}_{x'}$
we get $\diff[p]{\pi}\circ\gamma'=\diff[p]{\pi}\circ\gamma$ in
$T_{\pi(\affinefibre)}\base$.
Consequently, the variation $\diff[p]{\pi}\circ\alpha_t'$ does not depend on the
choice of $x'\in\IndLeaf{F}{p}$ and so neither does $\pi\circ\alpha_t'$ since
shortest paths in $\base$ are uniquely determined by their starting direction
and their length.

But this means that $\pi\circ\gamma'$ is independent of the choice of~$x'$ as well.
Hence the above argument implies that at any point $x'\in\IndLeaf{F}{p}$ it is
exactly the $\Fol$-Bott-parallel continuation of $\xi_x$ that projects onto the
$\IndFol[p]$-Bott-parallel continuation of $\tilde{\xi}_x$ via $\vprojH[p]$
thus proving our claim.
\end{proof}

\begin{Prop}\label{Prop:Characterisation of equidistance of IndFol}
If the induced foliation $\IndFol$ is equidistant then:
\begin{enumerate}
\item[($\ast$)] For any Bott-parallel vector field $\xi$ and any $p\in\affinefibre$
the projection $\projH{\xi}$ of $\xi$ to $\affinefibre$
is constant along any regular leaf $\IndLeaf{F}{p}$ of $\IndFol[p]$.
\end{enumerate}

Conversely, if ($\ast$) holds and if $\Fol$ is horizontally full
then $\IndFol$ is equidistant.
\end{Prop}

\begin{proof}
{\bfseries Part 1:}
We first assume $\IndFol$ to be equidistant.

\begin{figure}[t]
    \begin{center}
    \includegraphics[height=10cm]{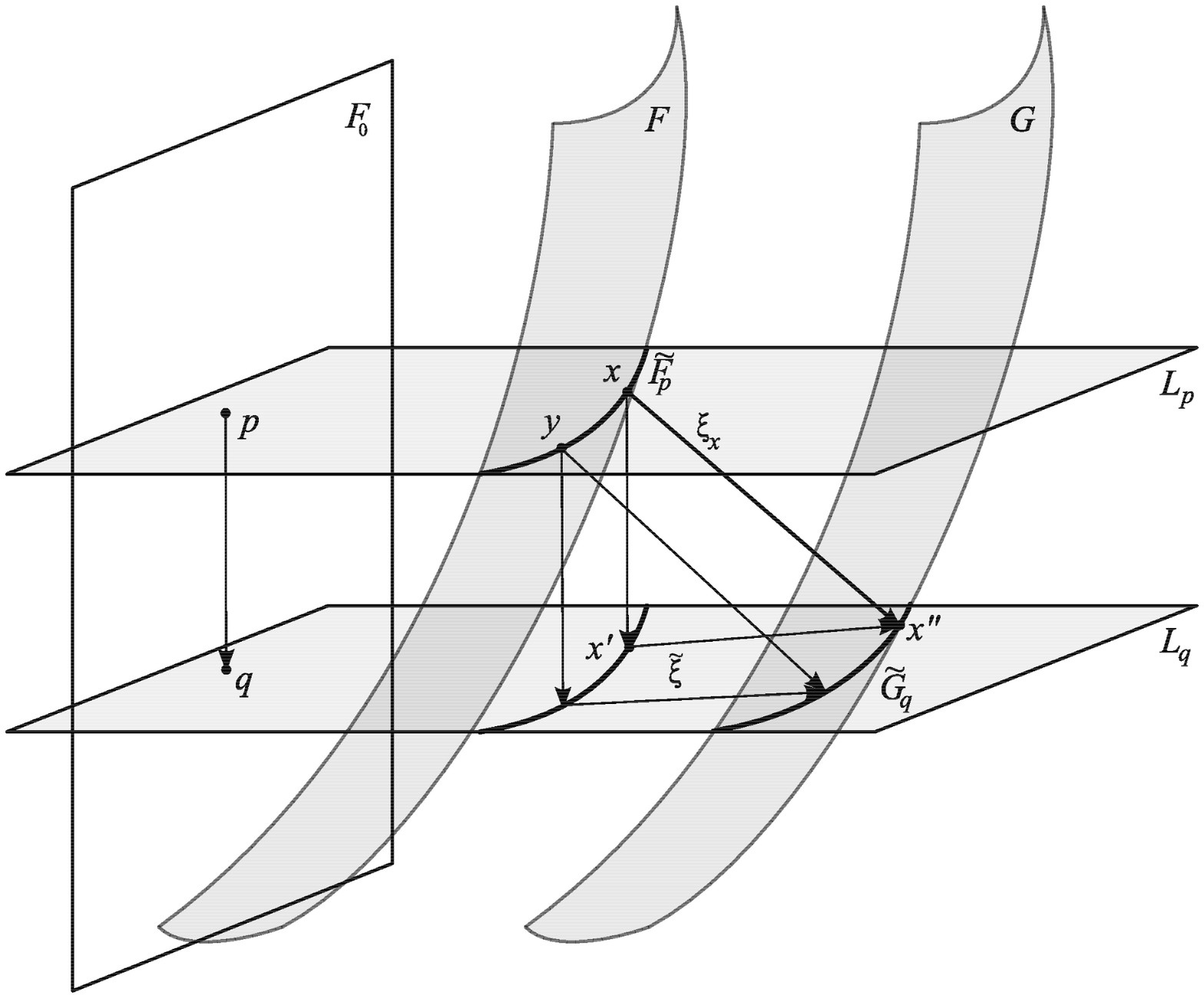}
    \caption[Equidistance of $\IndFol$ and the properties of the projections
            of $\xi$]
        {The equidistance of $\IndFol$ is equivalent to $\projH\xi$ being constant
        along $\IndLeaf{F}{p}$.}
    \end{center}
\end{figure}

Let $p$ be a point in $\affinefibre$ and $x\in F$ such that $\proj{x}=p$.
Choose any $\xi_x\in\normal_xF$ and define $q\in\affinefibre$ by
$q:=p+\projH\xi_x$.
We define $\tilde{\xi}_x:=\vprojH\xi_x$ and denote by $\tilde{\xi}$ its Bott
parallel (with respect to $\IndFol[p]$) continuation along $\IndLeaf{F}{p}$.

Then $\gamma_x \colon t \mapsto x+ t\xi_x$, for $t\in[0,1]$, is the
shortest path between $F$ and the leaf passing through $x+\xi_x$, which we will
denote by $G$.
Note that we may have to replace $\xi_x$ by $\eps\xi_x$ for $\gamma_x$ to be not
only \emph{locally} shortest, but the assertion of the lemma is invariant under
such a scaling of $\xi$.
Moreover, choosing $\eps$ sufficiently small guarantees the regularity of $G$.

Now for any point $y\in\IndLeaf{F}{p}$ we define $\xi_y:=\projH\xi_x+\tilde{\xi}_y$,
where we have identified vectors differing only by parallel transport in $\R^{k+n}$.

The equidistance of $\IndFol$ implies that both $x':=p+\projH\xi_x$ and
$y':=y+\projH\xi_x$ lie in the same leaf of $\IndFol[q]$. In particular
$x'':=x'+\tilde{\xi}_x=x+\xi_x$ and $y'':=y'+\tilde{\xi}_y=y+\xi_y$ both lie in
$\IndLeaf{G}{q}$ since $\tilde{\xi}$ is $\IndFol$-Bott parallel.

On the other hand, by definition $\xi$ has constant norm along $\IndLeaf{F}{p}$,
i.e.\ $|xx''|=|yy''|=\dist{F}{G}$. So, $\xi$ is the $\Fol$-Bott parallel
continuation of $\xi_x$ along $\IndLeaf{F}{p}$ and by construction $(\ast)$ holds.

{\bfseries Part 2:}
Let $p$ and $q$ be any two points in $\affinefibre$ and assume $(\ast)$ holds.
We will show that $\IndFol[p]$ and $\IndFol[q]$ are equidistant to each other.

Let $x\in F$ be a regular point of $\IndFol[p]$ and let $\xi$ be any Bott parallel
normal field along $F$.
Other than this, we will use the same notation as in Part~1.
Assertion ($\ast$) implies that $y+\xi_y$ lies in the same leaf $\IndLeaf{G}{q}$ of
$\IndFol[q]$ for any $y\in\IndLeaf{F}{p}$, in fact $\IndLeaf{F}{p}+\xi=\IndLeaf{G}{q}$.

On the other hand, assume $G=F+\xi$ to be regular. Then $\xi$ yields a Bott
parallel normal field on $G$ by defining $\zeta_{z+\xi_z}:=-\xi_z$ for $z\in G$.

Using assertion (b), we conclude that $\tilde\zeta=\vprojH\zeta$ is
$\IndFol[q]$-Bott parallel along $\IndLeaf{G}{q}$, i.e.\
$\IndLeaf{G}{q}+\tilde{\zeta}$ is some leaf $\IndLeaf{H}{q}$ in $\IndFol[q]$.
But by construction this is just the parallel translate by $\projH{\xi_x}$ of
$\IndLeaf{F}{p}$.

\begin{Rem*}
Note that $F+t\xi$ may be singular for certain values of $t$. However, this can
only happen for finitely many values of $t\in[0,1]$ (cf. Proposition~\ref{Fact:Quasigeodesics}).
So, the parallel translate of $\IndLeaf{F}{p}$ to $\layer{r}$ is a leaf in
$\IndFol[r]$ for almost all points $r$ lying on the line~$pq$.
By continuity of $\IndFol$
this holds indeed for \emph{all} $r$ in $pq$.
\end{Rem*}
\end{proof}

In general condition (*) appears hard to verify.
However, equidistance of $\IndFol$ follows if we prescribe certain dimensional
restrictions to the leaves of $\Fol$.

\begin{Cor}
If the affine leaf $\affineleaf$ is 1-dimensional the induced foliation $\IndFol$
is equidistant.
\end{Cor}

\begin{proof}
If $\Fol$ is horizontally full this is an immediate consequence of
Proposition~\ref{Prop:Characterisation of equidistance of IndFol}.

Otherwise, $\Hor[x]$ is everywhere perpendicular to $\affineleaf$, i.e.\ the leaves~$F$
of $\Fol$ are cylinders $\affineleaf\times\IndLeaf{F}{\star}$ and hence the assertion holds.
\end{proof}

\begin{Rem*}
Observe that if the regular leaves have codimension~2 horizonal fullness implies
$\affineleaf$ to be 1-dimensional and hence $\IndFol$ is equidistant as we have
seen.

Of course $\IndFol$ is equidistant if the regular leaves are hypersurfaces
and hence spherical cylinders around $\affineleaf$. Obviously, $\Fol$ cannot
be horizontally full in this case.
\end{Rem*}

\section{Isometries of the Induced Foliation}\label{Sec:Isometries}

We close this chapter with some observations on the group of isometries of the
induced foliation in each horizontal layer. Though interesting in themselves they
will become particularly important in the following chapters.

We are often interested in the objects related to the horizontal layer based at
a generic point in $\affineleaf$. Often these objects will be essentially
independent of the particular choice of base point and we will denote this generic
point by $\star$ and the objects based at this point by $\layer{\star}$,
$\IndFol[\star]$, etc.

The (effective) isometry group of $\IndFol[\star]$ is given by
\begin{equation}\label{Eq:Isom(IndFol)}
    \Isom{\IndFol[\star]}=\Norm{\IndFol[\star]}/\Centr{\IndFol[\star]},
\end{equation}
where the normalizer of $\IndFol[\star]$ consists of all $g\in\SO{n}$ leaving $\IndFol[\star]$
invariant while the centralizer of $\IndFol[\star]$ fixes each leaf of $\IndFol[\star]$.

If $\IndFol[\star]$ is homogeneous, i.e.\ given by the orbits of $\layergroup{\star}$,
then maximality of $\layergroup{\star}$ implies that $\Isom{\IndFol[\star]}$ is
simply $\Norm{\layergroup{\star}}/\layergroup{\star}$.
At least for irreducible $\IndFol[\star]$ we get some a priori information
about its isometry group.

\begin{Lem}\label{Lem:Isometry of irreducible foliation}
If the action of $\layergroup{\star}$ on $\layer{\star}$ is irreducible then the
the connected component of $\Isom{\IndFol[\star]}$
is contained in either $\set{\pm1}$, $\U{1}$ or $\Sp{1}$ depending on the type
of the $\layergroup{\star}$-action.
\end{Lem}

\begin{Rem}\label{Rem:G star perp}
Let $N:=\Norm{\layergroup{\star}}$ and denote by $\layergroup{\star}^\perp$ the
Lie subgroup $\exp(\Lie{g}_\star^\perp)$ of $N$ where $\Lie{g}_\star$ is the Lie
algebra of~$\layergroup{\star}$ and the orthogonal complement is taken with respect to
the Killing form on $N$ (cf.\ the proof of Proposition~\ref{Prop:Phi goes to Centralizer}).
Then $\layergroup{\star}^\perp$ is contained in the connected component
of the centralizer of $\layergroup{\star}$ and it is isomorphic to $\IsomNull{\IndFol[\star]}$
(cf.\ the proof of Proposition~\ref{Prop:Phi goes to Centralizer}).

\end{Rem}

\begin{proof}
Obviously $\layergroup{\star}^\perp$ acts on $\layer{\star}$ as a group of
$\layergroup{\star}$-invariant endomorphisms.
Since the $\layergroup{\star}$-action on $\layer{\star}$ is irreducible
Schur's Lemma implies that these endomorphisms are either zero or invertible.
Thus they form an associative division algebra over $\R$, namely $\R$, $\C$ or
$\Quaternions$, depending on the type of the representation
(cf.~\cite[Chap.II]{BtD}, in particular Thm.~(6.7)).

As $\layergroup{\star}^\perp$ acts by isometries it is contained in the respective
group of units. Hence, $\layergroup{\star}^\perp$ is either $\set{\pm1}$, $\U{1}$
or $\Sp{1}$.
\end{proof}

\begin{Rem*}
Let $H$ denote $\set{\pm1}$, $\U{1}$ or $\Sp{1}$ depending on the type of the
representation. Note that the isometric $H$-action on $\IndFol[\star]$ need not
be effective. For example consider the standard representation of $\U{n}$ on $\C^{n}$,
which is obviously of complex type but the isometry group of the orbit foliation
is trivial.

So $\IsomNull{\IndFol[\star]}$ may be much smaller
than~$H$. But at least the lemma provides an upper bound on $\IsomNull{\IndFol[\star]}$.
\end{Rem*}

The main reason for our interest in the isometries of $\IndFol[\star]$ is the
description of the leaves of $\Fol$ by the screw motion maps $\Psi_x$ as
introduced in Remark~\ref{Rem:Parametrization of Fol}.
From this description it is clear that the induced foliation $\IndFol$ is equidistant
if and only if the image of $\Psi_x$ is contained in the normalizer of
$\layergroup{x}$, for one and thus for any $x\in\affineleaf$.

So, assuming $\IndFol$ to be equidistant, Equation~(\ref{Eq:Isom(IndFol)})
implies even more, since it is not really $\Psi_x\colon\R^k\to\Norm{\IndFol[\star]}$
we are interested in but rather the induced
map~$\bar{\Psi}_x\colon\R^k\to\Isom{\IndFol[\star]}$.
As a consequence we get
a rather stronger result than that in Remark~\ref{Rem:Parametrization of Fol}:

\begin{Lem}\label{Lem:Parametrization into centralizer}
Let $\IndFol$ be equidistant and $\IndFol[\star]$ homogeneous.
Then for any $x\in\R^k$ there is a smooth map $\Psi_x\colon\R^k\to\layergroup{\star}^\perp$
such that the leaf $F$ passing through $(x,y)\in\R^{k+n}$ is given by
\begin{equation}\label{Eq:screw motion map}
    F=\Set{(x+v,\Psi_x(v).\layergroup{\star}.y)}{v\in\R^k}.
\end{equation}
\end{Lem}

\begin{proof}
As said above, Remark~\ref{Rem:Parametrization of Fol} yields a smooth map
$\psi_x\colon\R^k\to\Norm{\layergroup{\star}}$ satisfying~(\ref{Eq:screw motion map}).
Also the image of $\psi_x$ is contained in the connected component
of~$\Norm{\layergroup{\star}}$ as $\R^k$ is connected.

Let $P$ be the canonical projection
$\Norm{\layergroup{\star}}\to\Norm{\layergroup{\star}}/\layergroup{\star}$.
According to Remark~\ref{Rem:G star perp} there is a
Lie group isomorphism
$$
    \varphi\colon(\Norm{\layergroup{\star}}/\layergroup{\star})_0=
    \IsomNull{\IndFol[\star]}
    \to\layergroup{\star}^\perp
$$
such that any $h\in\NormNull{\layergroup{\star}}$ differs from $\varphi(P(h))$ only by
multiplication with some element of $\layergroup{\star}$.
And $\layergroup{\star}^\perp$ commutes with $\layergroup{\star}$.

In particular, setting $\Psi_x:=\varphi\circ P\circ\psi_x$ gives us the desired map
since $\psi_x$ and~$\Psi_x$ describe the same foliation $\Fol$.
\end{proof}

Finally observe that $\Fol$ is homogeneous if and only if
$\Psi_x\colon\R^k\to\layergroup{\star}^\perp$ is a Lie group homomorphism that
is independent of the base point~$x$.

\chapter{Reducibility of Equidistant Foliations}\label{Chap:Reducibility}

This chapter deals with two different notions of reducibility. The concept we start
with, the existence of invariant subspaces, is well known from representation theory
and we show that fullness of regular leaves characterizes irreducibility even in
the inhomogeneous case.
We then examine reducibility in the sense that the foliation splits as a product
and examine how this is linked to the notion of horizontal fullness we introduced
in the last chapter.

\section{Invariant Subspaces}

It is a well known fact
that a homogeneous foliation of Euclidean
space containing a non-full leaf is reducible.
To be more precise, suppose $G$ to be a Lie group acting on $\R^n$ by isometries.
Let $F$ be a $G$-orbit such that the minimal affine subspace $V$ containing $F$
has dimension strictly less than $n$. Then $V$ is invariant under the action of $G$.
This follows, using minimality of $V$, from the fact that the action of $G$ is
affine.

An analogous result holds for equidistant foliations:

\begin{Prop}\label{Prop:equidistant irreducibility}
    Let $\Fol$ be an equidistant foliation of $\R^n$ and let $F$ be
    a regular leaf.
    If $F$ is not full the minimal affine space $V$ containing $F$ consists of
    leaves of~$\Fol$, i.e.\ all leaves intersecting~$V$ are contained in $V$.
\end{Prop}

To prove this proposition we show that there is a Bott-parallel subbundle of~$\normal F$
such that at any point $x\in F$ the affine space $x+T_xF+\normal_xF$ is equal to $V$.
We achieve this by studying the following tensor.

\begin{Def}
Let $\onormal{}{}\colon\Hor\times\Hor\to\Hor$ be the tensor on the regular part
of $\Fol$ given by
\begin{equation}
    \onormal{X}{Y}:=\oneillAd{X}{\oneill{X}{Y}},
\end{equation}
where $\oneill{}{}$ and $\oneillAd{}{}$ are the O'Neill-tensor of $\Fol$ and its
pointwise adjoint.
\end{Def}

\begin{Rem}\label{Rem:Onormal is Bott-parallel}
Note that $\onormal{}{}$ is Bott-parallel, i.e.\ for $\xi,\eta\in\BottF$ the
vectorfield $\onormal{\xi}{\eta}$ is Bott-parallel as well. To see this let
$\xi,\eta,\zeta$ be Bott-parallel and observe that
$$
    \ip{\onormal{\xi}{\eta}}{\zeta}=\ip{\oneill{\xi}{\eta}}{\oneill{\xi}{\zeta}},
$$
which is constant along the regular leaves of $\Fol$ (cf.\ \cite[p.\ 145]{GG:2}).

Observe also that the image of any linear map $A$ is equal to the image of
$A^\ast A$, where $A^\ast$ is the adjoint of $A$.
In particular $\image{\oneillAd{\xi}{}}=\image{\onormal{\xi}{}}$ is Bott-parallel
if $\xi$ is.
\end{Rem}

\begin{Def}
The \emph{$k$-th osculating space} of $F$ at $p$ is the space $\osc{k}_pF$
spanned by the first $k$ derivatives of curves $\gamma\colon(-\eps,\eps)\ra F$
with $\gamma(0)=p$.

The \emph{$k$-th normal space} of $F$ at $p$ is the orthogonal complement
$\normal[k]_pF$ of $\osc{k}_pF$ in $\osc{k+1}_pF$.

We will use the notation
$$
    \minnormal[k]_pF:   = \bigoplus_{i=1}^k \normal[k]_pF
                        = \osc{k+1}_pF\cap\normal_pF
$$
for the direct sum over the first $k$ normal spaces and denote the sum over all
$\normal[k]_pF$ by $\minnormal_pF$.
\end{Def}

\begin{Rem}\label{Rem:Non-full leaf}
Note that the dimension of these spaces may depend on the point~$p$, hence,
in general, they do not form bundles over $F$.
However, if they do
then $F$ is contained in the affine space $p+ T_pF+\minnormal_pF=p+\osc{\infty}_pF$
for any $p\in F$
and this space is minimal in that respect.
(cf. \cite[Sect.~2.5 and p.~213]{BCO}).
\end{Rem}

\begin{Lem}\label{Lem:Ocsulating bundle is Bott-parallel}
Let $F\in\Fol$ be a regular leaf. Then for any $k$ the space
$\minnormal[k]_pF$ forms a Bott-parallel bundle over $F$.
\end{Lem}

\begin{proof}
We will prove this lemma by induction over $k$.

First we show that the first normal spaces $\normal[1]_pF$ are Bott-parallel,
in particular their dimensions are constant along $F$.
Let $x\in\normal_pF$ be a vector in the orthogonal complement of $\normal[1]_pF$,
i.e.\
$$
    0=\ip{x}{\Sff{v}{w}}=\ip{\shape{x}{v}}{w}
$$
for all $v,w\in T_pF$.

Let $X$ be the Bott-parallel continuation of $x$. Then $\shape{X}=0$, which is to
say $X$ is orthogonal to $\normal[1]F$, along $F$
by Proposition~\ref{Prop:Principal curvatures are constant}.
Hence, $(\normal[1]F)^\perp$ is Bott-parallel and consequently so is $\normal[1]F$.

Suppose $\minnormal[k]F$ to be a Bott-parallel bundle over $F$ and let
 $\xi_1,\ldots,\xi_{m}\in\BottF$ be an orthonormal frame of $\minnormal[k]F$.

Now $\minnormal[k+1]_pF$
can be viewed as the sum of $\minnormal[k]_pF$ and the space spanned by the
horizontal part $\kovH{v}{X}$ of the covariant derivatives at $p$ of vector
fields~$X\in\schnitt{\minnormal[k]F}$ in directions $v\in T_pF$.
In fact, writing such a vector field $X$ as a $C^\infty$ linear combination
$\sum_i f_i \xi_i$ of the $\xi_i$, it is easily seen that $\minnormal[k+1]F$
is spanned at each point by the $\xi_i$ and the horizontal part of their covariant
derivatives.

Remember that for Bott-parallel normal fields $\xi_i$ the equality
$$
    \kovH{v}{\xi_i} = -\oneillAd{\xi_i}{v}
$$
holds.
Remark \ref{Rem:Onormal is Bott-parallel} then implies that $\minnormal[k+1]F$
is a Bott-parallel bundle over $F$, which proves our claim.
\end{proof}

Now, Proposition~\ref{Prop:equidistant irreducibility} is a simple corollary of
Lemma \ref{Lem:Ocsulating bundle is Bott-parallel}:
Let $F$ be a non-full regular leaf of $\Fol$ and $V$ the minimal affine space
containing it. By Remark \ref{Rem:Non-full leaf}, $V=p+T_pF+\minnormal_pF$ for
any $p\in F$.

Take any point $q\in V\setminus F$, let $F'$ be the leaf passing through $q$ and
denote by $p_0$ the point in $F$ minimizing the distance to $q$.
Consider the Bott-parallel continu\-ation~$\xi$ of $q-p\in \normal_pF$ along $F$.
By Lemma \ref{Lem:Ocsulating bundle is Bott-parallel}, the horizontal geodesic
$t\mapsto p+t\xi_p$ is contained in $p+\minnormal_pF$ for any $p$, hence,
$F'=\Set{p+\xi_p}{p\in F}$ is a subset of $V$.

\section{The Non-compact Case}\label{Sec:Noncompact Case}

The results of the previous section make no assumptions on the affine leaf being
compact or not.
In order to deal with the stronger reducibility concept of $\Fol$ being a
product we now concentrate on the non-compact case.

\begin{Def}
For any leaf $F\in\Fol$ and points $x\in F$ and $p=\proj{x}\in\affineleaf$
let $\pDistN[F]{p,x}$ and $\pDist[F]{p,x}$ be the subspaces of $T_p\affineleaf$ defined by
$$
    \pDistN[F]{p,x}=\projH(\normal_xF), \qquad \pDist[F]{p,x}=(\pDistN[F]{p,x})^\perp.
$$
We call $F$ \emph{\nice} if for all $p\in\affineleaf$ the space $\pDistN[F]{p,x}$
(and hence $\pDist[F]{p,x}$) only depends on $p$ but not on $x\in\IndLeaf{F}{p}$.
The foliation $\Fol$ is called \emph{\nice} if all regular leaves are \nice.
\end{Def}

If $F$ is \nice we omit the index $x$.
By Lemma~\ref{Lem:IP of Projections} the dimension of~$\pDistN[F]{p}$ does not depend
on $p\in\affineleaf$ so $\pDistN[F]{}$ and $\pDist[F]{}$ are well defined distributions
on~$\affineleaf$.
Also we frequently omit the index $F$ and write just $\pDistN{}$ and $\pDist{}$ if
it is clear from the context which leaf the distributions are associated with.

\begin{Rem*}
Observe that Proposition~\ref{Prop:Characterisation of equidistance of IndFol}
implies that $\Fol$ is \nice if $\IndFol$ is equidistant. In particular
the regular leaves of a homogeneous foliation $\Fol$ are \nice.
Finally, $\Fol$ is \nice if it is horizontally full.
\end{Rem*}

We will show that there is a connection between $\Fol$ not being horizontally full
and $\Fol$ being reducible in the sense that it splits off a Euclidean factor.
By the latter we mean that there is an orthogonal vector space decomposition
$\R^{k+n}=V \oplus W$ and an equidistant Foliation $\Fol'$ of $V$ such that
$\Fol=\Set{F'\times W}{F'\in\Fol'}$.

Let us first list some properties of the distribution $\pDist{}$ beginning with
an auxiliary lemma:

\begin{Lem}\label{Lem:Projections and pDist}
Let $F$ be a leaf of $\Fol$ (not necessarily \nice), $x$ a point in $F$
and $p=\proj{x}\in\affineleaf$.
Identifying $\R^{k+n}$ with its tangent space at any point a vector $v\in\R^{k+n}$
is contained in $T_xF$ \emph{and} $T_p\affineleaf$ if and only if $v\in\pDist[F]{p,x}$.
\end{Lem}

\begin{proof}
The vector $v$ is contained in both $T_xF$ and $T_p\affineleaf$ if and only if
$\projV{v}=v$ (ignoring the base point).
From elementary linear algebra we know that
if $P$ is any orthogonal projection then
$$
    \ip{Pv}{Pw}=\ip{v}{Pw}=\ip{Pv}{w}, \quad \forall v,w.
$$
The rest follows taking $w\in\normal_xF$.
\end{proof}

If $F$ is \nice this implies that $\pDist[F]{}$ lifts to $F$ by parallel translation.

\begin{Prop}
Let $F$ be a \nice regular leaf of $\Fol$. Then $\pDist[F]{}$ is integrable.
Moreover if $M_p$ is an integral manifold passing through $p\in\affineleaf$ and
$x\in\IndLeaf{F}{p}$ then the parallel translate $M_p+(x-p)$ of $M_p$ to $x$ is
contained in $F$.
\end{Prop}

\begin{proof}
First note that by Lemma~\ref{Lem:Projections and pDist} we can lift $\pDist{}$
to $F$ just by parallel translating it, i.e.\ the distribution $\pDistLift{}$
defined by $\pDistLift{x}:=\pDist{\proj{x}}$ is tangent to $F$.

Hence, if $X, Y$ are tangent vector fields on $\affineleaf$ with values in $\pDist{}$
their vertical lifts $\bar{X}=\lift{X}$ and $\bar{Y}=\lift{Y}$ to $F$
(see Definition~\ref{Def:Vertical Lift}) take values in $\pDistLift{}$.
Obviously the Lie brackets $\lie{X}{Y}$ and $\lie{\bar X}{\bar Y}$ are tangent to
$\affineleaf$ and $F$ respectively. Now $X,\bar{X}$ and $Y,\bar{Y}$ differ just
by parallel translation, which yields an equality of Lie brackets:
$$
    \lie{\bar X}{\bar Y}_x = \lie{X}{Y}_{\proj{x}}
$$
up to parallel transport.
Lemma~\ref{Lem:Projections and pDist} then implies that $\lie{X}{Y}$ can only have
values in $\pDist{}$, so the latter is integrable. The rest follows immediately.
\end{proof}

As mentioned above, we now examine the connections between horizontal fullness
and reducibility of $\Fol$.

\begin{Prop}
Let $\Fol$ be horizontally full, then $\Fol$ does not split off a Euclidean factor.
\end{Prop}

\begin{proof}
Since the linear space $W$ is contained in $T_xF$ for all $x\in F$ and $F\in\Fol$
Lemma~\ref{Lem:Projections and pDist} implies that $W$ is a subspace of $\pDist[F]{p}$
for all $p\in\affineleaf$.
But since $\Fol$ is horizontally full $\pDist[F]{}$, and hence $W$, is trivial.
\end{proof}

Now, the natural question is whether the converse holds as well.
At least for homogeneous foliations we can show that $\Fol$ is reducible if it is
not horizontally full.

\subsection{Homogeneous Foliations}

Let $\Fol$ be homogeneous, $G$ the Lie group acting on $\R^{k+n}$ such that the
leaves of $\Fol$ are the orbits of $G$.
Remember from Proposition~\ref{Prop:Phi goes to Centralizer} that we can describe~$\Fol$
by giving the isotropy group $G_\star$ and a Lie group homomorphism $\Phi\colon\R^k\to\Centr{G_\star}$
from the affine leaf $\affineleaf\cong\R^k$ to the centralizer of $G_\star$ in $\SO{n}$.
We may assume that $G_\star$ is the maximal connected subgroup of
$\SO{n}$ with the given orbits.

\begin{Lem}\label{Lem:Ker(Phi) splits off}
The distribution $\ker\diff{\Phi}$ on $\affinefibre$ is parallel and $\Fol$ splits
off the Euclidean factor $\ker\diff[0]{\Phi}$. In particular $\Fol$ splits if
$\dim\affineleaf>\rk{\Isom{\IndFol[\star]}}$.
\end{Lem}

\begin{proof}
We start by proving that the distribution $\ker\diff{\Phi}$ is $G$-equivariant.
Observe that the velocity field of a curve $\gamma$ in $\affineleaf$ is everywhere
tangent to $\ker\diff{\Phi}$ if and only if $\Phi(\gamma(t)).x=x$ for
all~$x\in\R^n$ and all~$t$.
That is to say that $\gamma$ can be lifted into any leaf of $\Fol$ by parallel
transport, which implies
\begin{equation}\label{Eq:ker Phi}
    \ker\diff[p]{\Phi} = \bigcap_{F\in\Fol}\pDist[F]{p}, \quad \forall p\in\affineleaf,
\end{equation}
where the inclusion of the right hand side in the left follows from
Proposition~\ref{Prop:Phi goes to Centralizer}.

Now for any $F\in\Fol$ the distribution $\pDist[F]{p}$ is $G$-equivariant
($\proj{x}=x+\pos_x$ and $\pos$ is $G$-equivariant) and hence so is $\ker\diff{\Phi}$.

Consequently, $\ker\diff{\Phi}$ is parallel since $G$ acts on $\affineleaf$ by
translations. Thus $\Fol$ splits off the Euclidean factor $\ker\diff[0]{\Phi}$.
\end{proof}

\begin{Rem}\label{Rem:F_0 is 1-dimensional}
Assume $\IndFol[\star]$ to be irreducible, which is to say that the action of~$G_\star$
is irreducible.
By Lemma~\ref{Lem:Isometry of irreducible foliation} the rank of
$\Isom{\IndFol[\star]}$ is at most 1 and hence so is $\rk{\Phi(\R^k)}$
(cf.\ Lemma~\ref{Lem:Parametrization into centralizer}).

So, if $\Fol$ does not split Lemma~\ref{Lem:Ker(Phi) splits off} asserts that
the affine leaf $\affineleaf$ can be at most 1-dimensional.
\end{Rem}

\begin{Prop}\label{Prop:Homogeneous nonfull splits off}
If  $\Fol$ is homogeneous and not horizontally full then $\Fol$ splits off the
Euclidean factor $\affineleaf$ or $\IndFol[\star]$ is reducible.
\end{Prop}

\begin{Rem*}
N.b.\ the assertion does not hold if $\IndFol[0]$ is reducible.
To illustrate this consider the homogeneous foliation of $\R^4$ given by
$$
    \Fol = \Set{\left(t,
        \left(\begin{smallmatrix}
            \cos t & \sin t\\
            -\sin t & \cos t\\
            && 1
        \end{smallmatrix}\right)
        x
        \right)}{t\in\R, x\in\R^3}.
$$
Let $F$ be the leaf passing through $(0,0,0,1)$ then $\pDistN[F]{}$ is trivial
while $\Fol$ does not split off a Euclidean factor.
\end{Rem*}

\begin{proof}[Proof of Proposition~\ref{Prop:Homogeneous nonfull splits off}]
Assume $\IndFol[\star]$ to be irreducible. By Remark~\ref{Rem:F_0 is 1-dimensional}
we may also assume $\affineleaf$ to be 1-dimensional so  $H:=\Phi(\affineleaf)$
is trivial or isomorphic to $\S{1}$.
Let us assume the latter since in the former case we are already finished.

Let $F\in\Fol$ be a regular leaf that is not horizontally full.
This means that $\pDistN[F]{}$ is trivial
and $F$ is a cylinder $F=\affineleaf\times\IndLeaf{F}{\star}$.
Thus, $\IndLeaf{F}{\star}$ is invariant under the action of $H$.

Observe first that we may assume $H$ to act trivially on $\IndLeaf{F}{\star}$
since by Proposition~\ref{Prop:Phi goes to Centralizer} we can choose $\Phi$ such
that its image is contained in $\Centr{G_\star}$.

Irreducibility of $\IndFol[\star]$ implies that any regular leaf, in
particular~$\IndLeaf{F}{\star}$, is full.
Since $H\cong\S{1}$ the horizontal layer $\layer{\star}$ splits into an orthogonal
sum of 1- or 2-dimensional $H$-modules. We only have to consider the latter
since the action on the 1-dimensional modules is of course trivial.
But $\IndLeaf{F}{\star}$ being full means that for any $H$-module $V$ we can find a
point $x\in\IndLeaf{F}{\star}$ such that the $V$-component of $x$ is nonzero.
Since $H$ fixes $\IndLeaf{F}{\star}$ pointwise the action of $H$ on $V$ must be trivial.

Thus $H$ acts trivially on $\layer{\star}$, which means that \emph{all} the leaves of
$\Fol$ are cylinders splitting off the Euclidean factor $\affineleaf$.
\end{proof}

\subsection{The General Case}

We show that a somewhat weaker analogue to Proposition~\ref{Prop:Homogeneous nonfull splits off}
holds even if we drop the homogeneity assumption for $\Fol$.
But let us first generalize some of the findings of the previous section.

The key ingredient for the results in the previous section was describing~$\Fol$ via
the Lie group homomorphism~$\Phi\colon\affineleaf\to\NormNull{G_\star}$.

Remember that by Remark~\ref{Rem:Parametrization of Fol} we can describe any
equidistant foliation $\Fol$ of~$\R^{k+n}$ in a way similar to this as long as
$\IndFol[\star]$ is homogeneous. This result is refined by
Lemma~\ref{Lem:Parametrization into centralizer} for equidistant $\IndFol$.

As noted before, the screw motion map $\Psi_a$, $a\in\R^k\cong\affineleaf$,
need not be a Lie group homomorphism.
However,
we can still use it as a tool to examine reducibility of~$\Fol$.

We first introduce a further distribution on $\affineleaf$, which is motivated by
Equation~(\ref{Eq:ker Phi}).

\begin{Def}
Let $\pDist[\Psi_a]{}$ be the distribution on $\affineleaf$ given by
$$
    \pDist[\Psi_a]{p}:=\ker\left(\diff[p]{(\Psi_a)}\right), \qquad p\in\affineleaf.
$$
\end{Def}

The connection to (\ref{Eq:ker Phi}) becomes clear in the next lemma:

\begin{Lem}\label{Lem:EPhi lifts}
Let $a$ be an arbitrary point in $\R^k$.
Then for any $p\in\affineleaf$ the space~$\pDist[\Psi_a]{p}$ can be vertically
lifted to any leaf in $\Fol$ by parallel translation to some $x\in\layer{p}$,
i.e.~we have the inclusion
\begin{equation}\label{Eq:EpPsi}
    \pDist[\Psi_a]{p} \subset \bigcap_{F\in\Fol,\ x\in\IndLeaf{F}{p}}\pDist[F]{p,x}.
\end{equation}
\end{Lem}

\begin{proof}
Let $\gamma\colon(-1,1)\to\affineleaf$ be a smooth curve such that
its derivative $\dot{\gamma}(0)$ is tangent to $\pDist[\Psi_a]{\gamma(0)}$,
i.e. $\Ddt \Psi_a(\gamma(t))=0$.

Let $F$ be an arbitrary leaf in $\Fol$ and $x\in\IndLeaf{F}{\gamma(0)}$.
Describing $F$ in accordance with Remark~\ref{Rem:Parametrization of Fol},
choose $b\in\R^n\simeq\layer{a}$ such that
$$
    x=\big(\gamma(0),\ \Psi_a(\gamma(0)-a).b\big).
$$
Here we have identified $\gamma(t)$ with just its first $k$ coordinates (since
the last $n$ coordinates vanish anyway).

Now, consider the \emph{lifted} curve $\bar\gamma\colon(-1,1)\to F$ given by
$$
    \bar\gamma(t)=\big(\gamma(t),\ \Psi_a(\gamma(t)-a).b\big).
$$
Looking at its derivative, we obviously get
$\dot{\bar{\gamma}}(0)=\big(\dot\gamma(0), 0\big)$ which is just $\dot\gamma(0)$,
abusing notation again.
Hence, Lemma~\ref{Lem:Projections and pDist} implies that $\dot\gamma(0)$ is contained
in~$\pDist[F]{\gamma(0),x}$.
\end{proof}

\begin{Rem*}
\emph{For the remainder of this section we assume $\IndFol$ to be equidistant and $\IndFol[\star]$
to be homogeneous.}
Then by Lemma~\ref{Lem:Parametrization into centralizer} we can choose
$\Psi_a$ such that its image is contained in $\layergroup{\star}^\perp$ and thus
equality holds in (\ref{Eq:EpPsi}).
\end{Rem*}

An immediate consequence is the following:

\begin{Cor}
If $\Isom{\IndFol[\star]}$ is discrete $\Fol$ splits off the Euclidean
factor~$\affineleaf$.
\end{Cor}

Remember that an essential point in the proof of Proposition~\ref{Prop:Homogeneous nonfull splits off}
was to assume that $\affineleaf$ is at most 1-dimensional.
We show that --- provided $\IndFol$ is equidistant and $\IndFol[\star]$ is given by the
orbits of an irreducible representation of complex type ---
$\Fol$~splits if $\affineleaf$ has dimension larger than 1:

\begin{Lem}\label{Lem:Fol splits if dim(F_0) is greater than 1}
Assume $\Isom{\IndFol[\star]}$ to be 1-dimensional.
Then either the affine leaf $\affineleaf$ of $\Fol$ is at most 1-dimensional or
$\Fol$ splits off a Euclidean factor.
\end{Lem}

\begin{proof}
By Lemma~\ref{Lem:Parametrization into centralizer} we may assume the image of $\Psi_a$ to be
contained in the 1-dimensional Lie group $\layergroup{\star}^\perp$.
So for any $p\in\affineleaf$ the kernel of $\diff[p]{(\Psi_a)}$ is either a
hyperplane or all of $T_p\affineleaf$.

If the latter holds at any $p_0\in\affineleaf$ Lemma~\ref{Lem:EPhi lifts} clearly
implies that $\pDist[F]{{p_0}}=T_{p_0}\affineleaf$ for all $F\in\Fol$.
Since the dimension of $\pDist[F]{p}$ is independent of $p\in\affineleaf$ it follows
that $\Fol$ splits off the whole affine leaf $\affineleaf$.

So let us assume $\ker\left(\diff[p]{(\Psi_a)}\right)$ to be a hyperplane at
every point, which means $\Psi_a$ has only regular values. Consequently the level
sets of $\Psi_a$ are regular hypersurfaces of $\affineleaf$. We show that their
connected components form the leaves of an equidistant foliation of $\affineleaf$.
We achieve this by showing that this foliation is transnormal, i.e.\ geodesics
meeting any leaf perpendicularly meet all leaves perpendicularly
(cf.\ Remark~\ref{Rem:Molino}).

Let $p$ be any point in $\affineleaf$ and $\xi\in T_p\affineleaf$ perpendicular
to $\pDist[\Psi_a]{p}$. By Lemma~\ref{Lem:EPhi lifts} there is some leaf $F\in\Fol$
such that $\xi\in\pDistN[F]{p}$. Let then $\bar\xi\in\normal_xF$ be such that
$\projH{\bar\xi}=\xi$ with $x\in\IndLeaf{F}{p}$.
Then $\bar\gamma(t):=x+t\bar\xi$ meets $F$ perpendicularly and stays perpendicular to
all leaves of $\Fol$ it meets. Hence, its projection $\gamma\colon t\mapsto p+t\xi$
to $\affineleaf$ stays perpendicular to the distribution $\pDist[\Psi_a]{}$ since
$$
    \dot{\gamma}(t)=\projH(\dot{\bar\gamma}(t))\in\pDistN[F_t]{\gamma(t)},
$$
where $F_t$ is the leaf passing through $\bar\gamma(t)$.

Now the only equidistant foliation of Euclidean space by hypersurfaces is given by
parallel hyperplanes and lifting these to all leaves of $\Fol$ we see that $\Fol$
splits if the dimension of $\affineleaf$ is greater than 1.
\end{proof}

Assume $\IndFol[\star]$ to be given by the orbits of an irreducible representation.
If the representation is of \emph{real} type $\Fol$ splits off $\affineleaf$ since
the isometry group of $\IndFol[\star]$ is discrete.
If it is of \emph{complex} type and $\affineleaf$ has dimension greater than 1
then $\Fol$ splits, as we have just shown.

\begin{Rem*}
Note, that we cannot use the proof of Lemma~\ref{Lem:Fol splits if dim(F_0) is greater than 1}
if the representation is of \emph{quaternionic} type:

In the worst case $\IsomNull{\IndFol[\star]}=\Sp{1}$.
Assume $\Psi_a$ to have only regular values and its fibres to be
equidistant.
Let $\fol{G}$ be the foliation of $\affineleaf$ given by the fibres of $\Psi_a$
and let $\bar{\fol{G}}$ be the refinement of $\fol{G}$ given by the connected
components of its leaves.
Then $\Psi_a$ factorizes in the following way

$$
\xymatrix{
    \affineleaf \ar[r]^{\bar\Psi_a} \ar[dr]_{\Psi_a}
                                & \affineleaf/\bar{\fol{G}} \ar[d]^{p} \\
                                & \affineleaf/{\fol{G}}=\Sp{1}
}
$$
where $\bar\Psi_a\colon\affineleaf\to\affineleaf/\bar{\fol{G}}$ and
$p\colon\affineleaf/\bar{\fol{G}}\to\Sp{1}$ are the canonical projections.

Both $\fol{G}$ and $\bar{\fol{G}}$ are equidistant so $\Psi_a$ and $\bar\Psi_a$
are submetries if we take the induced metrics on the respective quotients.
Then $p$ is a submetry as well (cf.~Lem\-ma~\ref{Lem:factorizing submetries}).

Observe that $p$ has to be a covering map
because the fibres of $\bar\Psi_a$ are all regular and $p$ must be discrete
(cf.~\cite[Thm.~10.1]{Lyt}).
So $\affineleaf/\bar{\fol{G}}$ must be $\Sp{1}$ since $\Sp{1}\simeq\S{3}$ is simply
connected.
But on the other hand Theorem~\ref{Thm:Submetries with compact base} implies that
$\affineleaf/\bar{\fol{G}}$ cannot be compact.
Hence our assumption was wrong.
\end{Rem*}

We close with the generalized version of Proposition~\ref{Prop:Homogeneous nonfull splits off}:

\begin{Prop}\label{Prop:Reducibility}
If $\Fol$ is not horizontally full and $\IndFol[\star]$ is given by the orbits
of an irreducible representation of complex type then $\Fol$ splits off a Euclidean
factor.
\end{Prop}

\begin{proof}
In analogy to the proof of Proposition~\ref{Prop:Homogeneous nonfull splits off}
we choose a regular not horizontally full leaf $F\in\Fol$.
Then $\layergroup{\star}^\perp$ and hence the image of $\Psi_a$ leaves $\IndLeaf{F}{\star}$
invariant, even pointwise by Lemma~\ref{Lem:Parametrization into centralizer}.
The rest is exactly the same as in the proof of Proposition~\ref{Prop:Homogeneous nonfull splits off}
replacing $H$ with $\IsomNull{\IndFol[\star]}$.
\end{proof}

\chapter{Homogeneity Results}\label{Chap:Homogeneity}

In this chapter we finally address homogeneity of $\Fol$.
First, we consider the quotient $\intermediatespace=\R^{k+n}/\IndFol$ and show
that --- provided $\IndFol$ is equidistant --- the image of $\Fol$ under the
natural projection is an equidistant foliation of $\intermediatespace$.
Moreover, this new foliation is described by the same screw motion map as the
original one.
Reversing this construction we show how to construct new \emph{inhomogeneous}
equidistant foliations of Euclidean space.

We conclude with a homogeneity result for $\Fol$ if $\IndFol[\star]$ is homogeneous
and if $\Isom{\IndFol[\star]}$ fulfills certain conditions, e.g. if it is
sufficiently small.

\sep

\emph{Throughout this chapter we will assume $\IndFol$ to be equidistant.}

\section{Factorizing the Submetry}

In this section we will show that the submetry $\pi$ factorizes
into a composition $\pi_2\circ\pi_1$ such that both $\pi_i$ are submetries again.
This yields a foliation $\InterFol$ of the intermediate space
$\intermediatespace:=\pi_1(\R^{k+n})$ given by the fibres of $\pi_2$.
We construct the factorization of $\pi$ in such a way that the leaves of $\InterFol$
are exactly the images of the leaves of $\Fol$ under $\pi_1$.

It turns out that $\InterFol$ is more regular than $\Fol$ in the
sense that the leaves of $\InterFol$ are all of the same dimension.
This regularity of $\InterFol$ will be the key ingredient of our study of $\Fol$
during the following sections.
It is, however, bought at the expense of $\intermediatespace$ only
being an Alexandrov space albeit of a rather nice type.

\sep

In order to construct the map $\pi_1$ consider the following:
Let $\basedir$ denote the space of directions of $\base$ at the point
$\pi(\affinefibre)$, then $\cone{\basedir}$ is the tangent cone $T_{\pi(\affineleaf)}\base$.
Consider the map $\bar{\pi}:\layer{0}\to\cone{\basedir}$,  where $\bar{\pi}$
is the restriction of $\diff[0]{\pi}$ to the horizontal layer $\layer{0}$,
identifying $\layer{0}$ with $\Hor[0]$.
As we have seen in Section~\ref{Sec:IndFol in each Layer}, $\bar{\pi}$ is just
the canonical projection from $\layer{0}$ to $\layer{0}/\IndFol[0]\cong\cone{\basedir}$.

\begin{Def}
We set
$$
    \pi_1 \colon \R^{k+n} \cong \affinefibre \times \layer{0} \to \affineleaf \times \cone{\basedir},
    \qquad
    \pi_1:= \ID{\affinefibre} \times \bar{\pi}
$$
and $\intermediatespace:=\affineleaf \times \cone{\basedir}$. We define the map
$\pi_2\colon\intermediatespace\to\base$ by
$$
    \pi_2(\bar{x}):=\pi\circ\pi_1^{-1}(\bar{x}).
$$
\end{Def}

\begin{Rem*}
Observe that $\pi_1$
is a submetry since its components $\ID{\affinefibre}$ and $\bar{\pi}$ are.
Moreover, the fibres of $\pi_1$ are the leaves of $\IndFol$ because the latter
is equidistant. So, $\pi_1$ is just the canonical projection $\R^{k+n}\to\R^{k+n}/\IndFol$.

Since $\IndFol$ is a subfoliation of $\Fol$ the map $\pi_2$ is well defined and
by Lemma~\ref{Lem:factorizing submetries} it is a submetry.
\end{Rem*}

So, the fibres of $\pi_2$ define an equidistant foliation $\InterFol$ of $\interspace$, which
by the remark above is given by the images of the leaves of $\Fol$, i.e.\
$$
    \InterFol=\Set{\pi_2^{-1}(x)}{x\in\base} = \Set{\pi_1(F)}{F\in\Fol}.
$$

\section{New Examples from Old}\label{Sec:New from Old}

We now study $\intermediatespace$ and its foliation $\InterFol$ in order to better
understand $\Fol$.

As the essential information about $\intermediatespace$ is contained in the
structure of $\basedir$ understanding $\Isom{\basedir}$ appears to be essential.
In  Section~\ref{Sec:Isometries} we have already discussed the isometry group of
the induced foliation $\IndFol[\star]$.
Now, remember that $\Isom{\IndFol[\star]}$ acts effectively and by isometries on
$\cone{\basedir}=\layer{\star}/\IndFol[\star]$ and hence on $\basedir$
as the action fixes the apex of the cone.
However, it is possible for the space of leaves to have more isometries
than the foliation.

\begin{Rem*}
The subgroup $\Isom{\IndFol[\star]}\subset\Isom{\basedir}$ consists exactly
of the isometries of~$\basedir$ that may be liftet to $\IndFol[\star]$.
\end{Rem*}

For example consider an isoparametric hypersurface in a sphere and
the foliation created by its parallel surfaces
(cf.~\cite[Sect.~8.4]{PT} and \cite{FKM}).
Such a foliation always has two
focal manifolds, hence the space of leaves is a compact interval with the reflection
at the midpoint being the only nontrivial isometry. But this reflection cannot always
be lifted to an isometry of the foliation since the two focal manifolds may have
different dimension.

It is not even clear whether the connected components of the two isometry groups
are the same.
Nevertheless we will see that understanding the action of $\IsomNull{\IndFol[\star]}$
is quite sufficient in order to understand $\InterFol$.

But first we mention a splitting result (cf.\ \cite[Prop.~12.14]{Lyt}) for the
submetry $\bar\pi\colon\R^{n}\to\cone{\basedir}$:

\begin{Prop}
If $\diam{\basedir} > \frac{\pi}{2}$ then $\cone{\basedir}$ splits as
$\cone{\basedir}=\R^l\times\cone{\basedir'}$ with
$\diam{\basedir'}\leq\frac{\pi}{2}$.
Moreover
$\bar\pi\colon \R^l\times\R^{n-l}\to\R^l\times\cone{\basedir'}$ splits as
$\bar\pi=\ID{\R^l}\times\bar\pi'$ and $\bar\pi'$ is a submetry.
\end{Prop}

In particular if $\basedir$ has diameter greater than $\pi/2$, $\IndFol[\star]$ is
reducible.

\sep

Assuming $\IndFol[\star]$ to be homogeneous Section~\ref{Sec:Isometries} shows that
$\Fol$ is completely described by two data: the group $\layergroup{\star}$ acting
on $\layer{\star}$ and a smooth map (or rather a set of maps)
$\Psi_x\colon \R^k\cong\affineleaf\to\layergroup{\star}^\perp\cong\IsomNull{\IndFol[\star]}$.
Thus the foliation $\InterFol$ is completely described by $\Psi_x$ interpreting
it as a map into $\IsomNull{\IndFol[\star]}\subset\IsomNull{\basedir}$:
\begin{equation}\label{Eq:InterFol}
        \InterFol=\Set{
                \Set{(x+v,\Psi_x(v).a)}{v\in\R^k}
                }{(x,a)\in\affineleaf \times \cone{\basedir}},
\end{equation}
and $\InterFol$ is homogeneous if and only if $\Psi_x$ is a Lie group homomorphism
independent of the base point $x\in\R^k$, i.e.\ if and only if $\Fol$ is homogeneous.

\sep

Using the converse approach, we show how equidistant foliations $\Fol$ of $\R^{k+n}$
may be constructed from the data mentioned above.
In particular we give new examples of inhomogeneous equidistant foliations
of~$\R^{k+n}$.

So, let $\fol{G}$ be an equidistant foliation of $\S{n}$, $\basedir:=\S{n}/\fol{G}$
and $G:=\IsomNull{\fol{G}}$. Choose a smooth map
$\Psi_0\colon\R^k\to G\subset\Isom{\cone{\basedir}}$. Then,
setting $\intermediatespace:=\R^k\times\cone{\basedir}$ this yields a foliation
$\InterFol$ of $\intermediatespace$ with the leaf $A$ passing through
$(0,a)$ given by
$$
    A=\Set{(v,\Psi_0(v).a)}{v\in\R^k}.
$$
Viewing $G$ as a subgroup of $\SO{n}$ we can lift this construction to $\R^{k+n}$.
Thus we get the foliation $\Fol$ with the leaf $F\in\Fol$ passing through $(0,x)$
given by
$$
    F=\Set{(v,\Psi_0(v).y)}%
        {v\in\R^k,\; y \text{ in the same }\fol{G}\text{-leaf as }x}.
$$

This construction induces the two maps
$$
    \R^{k+n} \overset{\pi_1}{\longrightarrow}
    \R^{k+n}/\IndFol=\intermediatespace \overset{\pi_2}{\longrightarrow}
    \intermediatespace/\InterFol=:\base
$$
and $\Fol$ is given by the fibres of $\pi_2\circ\pi_1$.
Note that by construction $\IndFol$ is automatically equidistant, hence $\pi_1$
is a submetry.
So, $\Fol$ is equidistant if and only if $\InterFol$ is.

In general, equidistance of $\InterFol$ will be rather hard to check.
However, it follows immediately if $\InterFol$ is homogeneous, i.e.\ if $\Psi_0$
is a Lie group homomorphism.

\begin{Rem*}
Note that $\Fol$ inherits the remaining properties of an equidistant foliation
from $\fol{G}$ since $\Psi_0$ is smooth.
\end{Rem*}

Choosing $\Psi_0$ to be a group homomorphism means that $\Fol$ is homogeneous if
and only if $\fol{G}$ is.
Let us start then with $\fol{G}$ being inhomogeneous.
As said before the only known examples are the ones generated by isoparametric
hypersurfaces in spheres and the octonional Hopf fibration
$\S{7}\hookrightarrow\S{15}\to\S{8}$.
We already mentioned above that in the former case the leaf space is a compact
interval and hence $G$ is trivial. So here our construction yields nothing new.

\sep

Let us look at the Hopf fibration of $\S{15}$, which is given by
$$
\xymatrix{
    \S{7}   \ar@{^{(}->}[d] &= \Spin{8}/{\widetilde{\Spin{7}}} \\
    \S{15}  \ar[d] &= \Spin{9}/{\widetilde{\Spin{7}}} \\
    \S{8}   &= \Spin{9}/{\Spin{8}}
}
$$
and $\widetilde{\Spin{7}}$ is the image of the standard embedding of
$\Spin{7}$ in $\Spin{8}$ under a (non-trivial) triality automorphism of $\Spin{8}$.

\begin{Rem}\label{Rem:Hopf}
In general let $G$ be a Lie group and $K\subset H\subset G$ compact subgroups.
Thus we get the natural fibration $p\colon G/K\to G/H$ mapping $gK$ to~$gH$.

Then a result by B\'erard Bergery states that
we can find suitable $G$-invariant metrics on $G/K$ and $G/H$ and an $H$-invariant
metric on $H/K$ such that $p$ is a Riemannian submersion with totally geodesic
fibres isometric to $H/K$ (see~\cite[p.~256f]{Besse} for a detailed discussion).

Since the fibre through $gK$ is $(gH)K=\Set{ghK}{h\in H}\cong H/K$ the submersion~$p$
is obviously $G$-equivariant.
\end{Rem}

Note that in our case $\S{15}$ and $\S{7}$ bear just the standard metric
and $\S{8}$ is a Euclidean sphere of radius 1/2 (cf.~\cite[9.84]{Besse}).

We see that $\Spin{9}$ acting transitively on $\S{15}$ leaves the Hopf fibration
invariant. On the other hand let $N\subset\Spin{9}$ be the subgroup that maps
fibres into themselves, which hence has to be a normal subgroup.
But $\SO{9}=\Spin{9}/\set{\pm1}$ is simple so $N\subset\set{\pm1}$ and $-\id$
obviously does map the fibres into them\-selves.

This means that $\SO{9}$ acts transitively and effectively on the Hopf fibration.
Since $\SO{9}$ is the isometry group of the space of fibres $\S{8}$
it is already the full isometry group of the Hopf fibration.

Hence, we have proved:

\begin{Prop}
Taking any Lie group homomorphism $\Psi_0\colon\R^k\to\SO{9}$ the above construction
yields an inhomogeneous non-compact equidistant foliation of $\R^{k+n}$ with the induced
foliation being given by the Hopf fibration $\S{7}\hookrightarrow\S{15}\to\S{8}$.
\end{Prop}

Of course we can limit ourselves to $k\leq4$ since $\SO{9}$ has rank 4 and
the kernel of $\Psi_0$ splits off as a Euclidean factor
(cf.~Lemma~\ref{Lem:Ker(Phi) splits off}).

\section{Homogeneity}

We now present the main result of this chapter. The idea underlying it is
that we do not have to know too much about $\basedir$ to understand $\InterFol$
and thus $\Fol$.
The important thing is rather how $\IsomNull{\IndFol[\star]}$ acts on $\basedir$.
If this action is ``similar'' to a representation acting transitively on a sphere
we can use Gromoll and Walschap's result to prove homogeneity of $\InterFol$ and
thus of $\Fol$:

\begin{Thm}\label{Thm:Homogeneity}
Let $\Fol$ and $\IndFol$ be equidistant and let $\IndFol[\star]$ be homogeneous.
If the action of $H:=\IsomNull{\IndFol[\star]}$ on $\cone{\basedir}$ has an orbit~$B$
isometric to a round sphere and $H$ acts effectively on $B$ then $\Fol$ is
homogeneous.
\end{Thm}

\begin{proof}
Since $H$ acts on $\cone{\basedir}$ by isometries, the partition $\fol{B}$
of~$\intermediatespace$ by the $\affineleaf$-cylinders over these $H$-orbits is
equidistant.
Moreover, $\InterFol$ is a refinement of $\fol{B}$, so Lemma~\ref{Lem:factorizing submetries}
implies that the restriction $\InterFol_B$ of $\InterFol$ to $\affineleaf\times B$
is equidistant as well.

Now, by assumption, $\affineleaf\times B$ is isometric to a round cylinder
$\R^k\times\S{l}_r\subset\R^{k+l+1}$ for some $l\geq1$. Let us call this
isometry~$\varphi$.
Consequently, the image of $\InterFol_B$ under~$\varphi$ is equidistant and may
be described via the maps $\bar\Psi_x$ with
$$
    \bar\Psi_x\colon\R^k\to\SO{l},\qquad
    \bar\Psi_x(v).\varphi(b):= \varphi(\Psi_x(v).b),\qquad
    \forall v\in\affineleaf, b\in B
$$
such that the leaf $\bar{A}$ of $\varphi(\InterFol_B)$ passing through $(x,y)$
is given by
$$
    \bar{A}=\Set{(x+v,\bar\Psi_x(v).y)}{v\in\R^k}.
$$

Now $\varphi(\InterFol_B)$ can be extended to an equidistant foliation of $\R^{k+l+1}$
and this foliation is regular. Thus, by \cite{GW:2}
this foliation is homogeneous.
In particular \cite[Thm.~2.6]{GW:1} implies that the maps $\bar{\Psi}_x$ must
be Lie group homomorphisms independent of $x$.
But then the same holds for the maps $\Psi_x$ and so $\Fol$ is homogeneous.
\end{proof}

We immediately get the following important application for $\IndFol[\star]$
having small isometry group:

\begin{Cor}
If $\dim\left(\Isom{\IndFol[\star]}\right)\leq1$, in particular if
\begin{enumerate}[(i)]
\item $\IndFol[\star]$ is given by the orbits of an irreducible representation
of real or complex type or
\item $\IndFol[\star]$ is given by an irreducible polar action
\end{enumerate}
then $\Fol$ is homogeneous.
\end{Cor}

\begin{proof}
Assume $H:=\IsomNull{\IndFol[\star]}=\U{1}$ then the $H$-orbits on $\cone{\basedir}$
are either single points or diffeomorphic and hence isometric to $\S{1}_r$.
The latter holds if and only if $H$ acts effectively on that orbit.
So if there is an effective $H$-orbit on $\cone{\basedir}$ Theorem~\ref{Thm:Homogeneity}
implies homogeneity of $\Fol$.
On the other hand, if there is no effective $H$-orbit
the action of $H$ is trivial and hence $\Fol$ splits off $\affineleaf$.

Now let us consider the special cases mentioned:
If the representation is of real type we have already seen that $\IsomNull{\IndFol[\star]}$
is trivial and hence $\Fol$ splits off~$\affineleaf$.
If it is of complex type $H$ is a subgroup of $\U{1}$ and we are done by what we
mentioned above.

If $\IndFol[\star]$ is given by a polar representation
$\cone{\basedir}=\layer{\star}/\IndFol[\star]$ is the Weyl chamber of a principal
orbit.
In particular its isometry group is discrete, so $\Fol$ splits off $\affineleaf$
again.
\end{proof}

However, $\IsomNull{\IndFol[\star]}$ being small is not necessary as the following
result shows:

\begin{Cor}
If $\IndFol[\star]$ is given by the complex or quaternionic Hopf fibrations
$\S{1}\hookrightarrow\S{3}\to\S{2}$ or $\S{3}\hookrightarrow\S{7}\to\S{4}$
then $\Fol$ is homogeneous.
\end{Cor}

\begin{proof}
In both cases $\basedir$ is a sphere so to apply Theorem~\ref{Thm:Homogeneity}
we show that $\IsomNull{\IndFol[\star]}$ acts transitively and effectively on
$\basedir$. This can be done using Remark~\ref{Rem:Hopf}. However, a more direct
approach is possible:

Consider the $\U{1}$-action on $\S{3}\subset\C^{2}$ by complex multiplication with
unit complex numbers: $\lambda.(z_1,z_2)=(\lambda z_1,\lambda z_2)$.
The complex Hopf fibration is then the natural projection to the orbit space
$\C P^1\cong\S{2}$. We show that $\Isom{\IndFol[\star]}=\SO{3}$:

Let $G:=(\SU{2}\times\U{1})/_\sim$ where we identify $(A,\lambda)$ with $(-A,-\lambda)$.
Then $G$ acts on $\S{3}\subset\C^2$ in the following way:
$(A,\lambda).(z_1,z_2):=A(z_1\lambda,z_2\lambda)=\lambda A(z_1,z_2)$ and this
action is effective.


Obviously $G$ leaves $\IndFol[\star]$ invariant as the $G$-action commutes
with the $\U{1}$-action. On the other hand, it is clear that the only elements
of $G$ leaving each leaf of $\IndFol[\star]$ invariant are of the form $(\id,\lambda)$.
So $G/_{(\set{\id}\times\U{1})}\cong\SU{2}/_{\set{\pm1}}$ acts effectively on
the foliation and hence it is contained in $\Isom{\IndFol[\star]}$.
But
$$
    \SO{3}  \cong   G/_{(\set{\id}\times\U{1})}    \subset
    \IsomNull{\IndFol[\star]}   \subset\IsomNull{\basedir}  =\SO{3}
$$
and thus equality holds at every step.

The quaternionic case is rather similar. Here we consider the action of $\Sp{1}$
on $\S{7}\subset\Qt^2$ by quaternionic multiplication from the right:
$h.(q_1,q_2):=(q_1h^{-1},q_2h^{-1})$.
The orbits form the foliation $\IndFol[\star]$.
The remainder is analogous to the complex case:

Let $H:=(\Sp{2}\times\Sp{1})/_\sim$ with $(A,h)\sim(-A,-h)$ and $H$ acts effectively on
$\S{7}\subset\Qt^2$ via $(A,h).(q_1,q_2):=A(q_1h^{-1},q_2h^{-1})$.

Again it is clear that $H$ leaves $\IndFol[\star]$ invariant and the only elements
of $H$ fixing each leaf of $\IndFol[\star]$ are of the form $(\id,h)$.
The latter can easily be seen by letting $(A,h)$ act on $(a,b)$ with $a,b\in\set{0,1,i,j,k}$.
As before $H/_{(\set{\id}\times\Sp{1})}\cong \Sp{2}/_{\set{\pm1}}$ acts
effectively on $\IndFol[\star]$ and
$$
    \SO{5} \cong H/_{(\set{\id}\times\Sp{1})}
    \subset \IsomNull{\IndFol[\star]} \subset\IsomNull{\basedir}  =\SO{5}
$$
implies that $\IsomNull{\IndFol[\star]}$ acts effectively and transitively on
$\basedir=\S{4}$.
\end{proof}

\subsection*{Open Questions}

Some problems that were addressed in this thesis still remain open. In particular
it has been essential for our homogeneity results to assume the induced
foliation~$\IndFol$ to be equidistant. Based on the findings of Chapter~\ref{Chap:Horizontal}
it is my conjecture that indeed equidistance of $\Fol$ implies that of $\IndFol$.
I even conjecture that equidistance of $\Fol$ together with homogeneity
of $\IndFol[\star]$ already implies $\Fol$ to be homogeneous.

At the very least this should be true for $\IndFol$ equidistant and $\IndFol[\star]$
homogeneous and irreducible. To see this one would have to show that the orbits
of $\IsomNull{\IndFol[\star]}$ can only be $\S{1}$, $\S{2}$, $\S{3}$ or one of
the corresponding projective spaces. One could then try to modify the proof of
Theorem~\ref{Thm:Homogeneity} or indeed the approach used in \cite{GW:2}
to work in the projective case as well.

The first conjecture is obviously necessary for the second but is also interesting
in itself. For example it implies that there are no further examples of noncompact
inhomogeneous equidistant foliations of $\R^n$ than those given in Section~\ref{Sec:New from Old};
in particular the \cite{FKM}-examples cannot appear as induced foliation of
an irreducible~$\Fol$.

On the other hand, proving this conjecture wrong would be most interesting as well
since it would provide a whole new class of inhomogeneous equidistant foliations.


\backmatter

\bibliographystyle{alpha}
\bibliography{Main}

\end{document}